\documentclass{amsart}

\usepackage{graphicx}
\usepackage{amsthm,amsmath,amssymb,amsfonts}
\usepackage{pb-diagram}
\usepackage{latexsym}
\usepackage[colorlinks=true,hypertexnames=false,linkcolor=blue,citecolor=blue]{hyperref}

\newtheorem{theorem}{Theorem}[section]

\newtheorem{lema}[theorem]{Lemma}
\newtheorem{prop}[theorem]{Proposition}
\newtheorem{coro}[theorem]{Corollary}
\newtheorem*{Lan}{Rigidity Conjecture}

\theoremstyle{definition}
\newtheorem{definition}[theorem]{Definition}

\theoremstyle{remark}

\numberwithin{equation}{section}

\theoremstyle{plain}
\newtheorem{maintheorem}{Theorem}

\newcommand{\A}{\ensuremath{\mathbb{A}}}
\newcommand{\D}{\ensuremath{\mathbb{D}}}
\newcommand{\Q}{\ensuremath{\mathbb{Q}}}
\newcommand{\h}{\ensuremath{\mathbb{H}}}
\newcommand{\R}{\ensuremath{\mathbb{R}}}
\newcommand{\Z}{\ensuremath{\mathbb{Z}}}
\newcommand{\C}{\ensuremath{\mathbb{C}}}
\newcommand{\sph}{\widehat{\mathbb{C}}}
\newcommand{\nt}{\ensuremath{\mathbb{N}}}

\DeclareMathOperator{\interior}{interior}

\DeclareMathOperator{\Leb}{Leb}
\DeclareMathOperator{\Diff}{Diff}

\DeclareMathOperator{\modulo}{mod}
\DeclareMathOperator{\diam}{diam}

\begin{document}

\title{Rigidity of Smooth Critical Circle Maps}

\author{Pablo Guarino}
\address{Instituto de Matem\'atica e Estat\'istica, Universidade de S\~ao Paulo}
\curraddr{Rua do Mat\~ao 1010, 05508-090, S\~ao Paulo SP, Brasil}
\email{guarino@ime.usp.br}

\author{Welington de Melo}
\address{IMPA, Rio de Janeiro, Brazil}
\curraddr{Estrada Dona Castorina 110, 22460-320}
\email{demelo@impa.br}

\thanks{The authors are grateful to Artur Avila, Trevor Clark, Edson de Faria, Daniel Smania, Sebastian van Strien and Charles Tresser for several useful conversations. The first author is indebted to Enrique Pujals for his constant encouragement. The authors also would like to thank IMPA, Universidade de S\~ao Paulo and University of Warwick for their warm hospitality during the preparation of this article. This work was partially supported by CNPq (grant $\#$300813/2010-4), FAPERJ (grant $\#$E-26/101.584/2010), the DynEurBraz program and the Balzan Research project of J. Palis. The first author was supported by FAPESP (postdoctoral fellow $\#$2012/06614-8).}

\subjclass[2000]{Primary }


\date{}

\dedicatory{}

\begin{abstract} We prove that any two $C^3$ critical circle maps with the same irrational rotation number of bounded type and the same odd criticality are conjugate to each other by a $C^{1+\alpha}$ circle diffeomorphism, for some universal $\alpha>0$.
\end{abstract}

\maketitle

\section{Introduction}

In the theory of real one-dimensional dynamics there exist many levels of equivalence between two systems: combinatorial, topological, quasi-symmetric and smooth equivalence are major examples.

In the circle case, a classical result of Poincar\'e \cite[Chapter 1, Theorem 1.1]{livrounidim} states that circle homeomorphisms with the same irrational rotation number are \emph{combinatorially equivalent}: for each $n\in\nt$ the first $n$ elements of an orbit are ordered in the same way for any homeomorphism with a given rotation number. This implies that circle homeomorphisms with irrational rotation number are semi-conjugate to the corresponding rigid rotation and, therefore, they admit a unique invariant Borel probability measure.

By Denjoy's theorem \cite{denjoy}, any two $C^2$ circle diffeomorphisms with the same irrational rotation number are conjugate to each other by a $C^0$ homeomorphism (actually we just need $C^1$ maps such that the logarithm of the modulus of the derivative has bounded variation). This implies that $C^2$ diffeomorphisms with irrational rotation number are minimal, and therefore, the support of its unique invariant probability measure is the whole circle.

By a fundamental result of Herman \cite{hermanihes}, improved by Yoccoz \cite{yoccoz0}, any two $C^{2+\varepsilon}$ circle diffeomorphisms whose common rotation number $\rho$ satisfies the Diophantine condition:
\begin{equation}\label{diof}
\left|\rho - \frac{p}{q}\right|\geq\frac{C}{q^{2+\delta}}\,,
\end{equation}
for some $\delta \in [0,1)$ and $C>0$, and for every positive coprime integers $p$ and $q$, are conjugate to each other by a circle diffeomorphism. More precisely, if $0 \leq \delta < \varepsilon \leq 1$ and $\varepsilon-\delta\neq 1$, any such diffeomorphism is conjugate to the corresponding rigid rotation by a $C^{1+\varepsilon-\delta}$ diffeomorphism \cite{khaninteplinsky2}. This implies that its invariant probability measure is absolutely continuous with respect to Lebesgue, with H\"older continuous density with exponent $\varepsilon-\delta$. Moreover, any two $C^{\infty}$ circle diffeomorphisms with the same Diophantine rotation number are $C^{\infty}$-conjugate to each other, and real-analytic diffeomorphisms with the same Diophantine rotation number are conjugate to each other by a real-analytic diffeomorphism \cite[Chapter I, Section 3]{livrounidim}.

These are examples of \emph{rigidity} results: lower regularity of conjugacy implies higher regularity under certain conditions.

Since rigidity is totally understood in the setting of circle diffeomorphisms we continue in this article the study of rigidity problems for critical circle maps developed by de Faria, de Melo, Yampolsky, Khanin and Teplinsky among others.

By a \emph{critical circle map} we mean an orientation preserving $C^3$ circle homeomorphism with exactly one non-flat critical point of odd type (for simplicity, and for being the generic case, we will assume in this article that the critical point is of cubic type). As usual, a critical point $c$ is called \emph{non-flat} if in a neighbourhood of $c$ the map $f$ can be written as $f(t)=\pm\big|\phi(t)\big|^d+f(c)$, where $\phi$ is a $C^3$ local diffeomorphism with $\phi(c)=0$, and $d\in\nt$ with $d \geq 2$. The \emph{criticality} or \emph{type} of the critical point $c$ is $d$.

Classical examples of critical circle maps are obtained from the two-parameter family $\widetilde{f}_{a,b}:\C\to\C$ of entire maps in the complex plane:

\begin{equation}\label{arfam}
\widetilde{f}_{a,b}(z)=z+a-\left(\frac{b}{2\pi}\right)\sin(2\pi z)\quad\mbox{for $a\in[0,1)$ and $b \geq 0$.}
\end{equation}

Since each $\widetilde{f}_{a,b}$ commutes with unitary horizontal translation, it is the lift of a holomorphic map of the punctured plane $f_{a,b}:\C\setminus\{0\}\to\C\setminus\{0\}$ via the holomorphic universal cover $z \mapsto e^{2\pi iz}$. Since $\widetilde{f}_{a,b}$ preserves the real axis, $f_{a,b}$ preserves the unit circle $S^1=\big\{z\in\C:|z|=1\big\}$ and therefore induces a two-parameter family of real-analytic maps of the unit circle. This classical family was introduced by Arnold in \cite{arnold}, and is called the \emph{Arnold family}.

For $b=0$ the family $f_{a,b}:S^1 \to S^1$ is the family of rigid rotations $z \mapsto e^{2\pi ia}z$, and for $b \in (0,1)$ the family is still contained in the space of real-analytic circle diffeomorphisms.

For $b=1$ each $\widetilde{f}_{a,b}$ still restricts to an increasing real-analytic homeomorphism of the real line, that projects to an orientation-preserving real-analytic circle homeomorphism, presenting one critical point of cubic type at $1$, the projection of the integers. Denote by $\rho(a)$ the rotation number of the circle homeomorphism $f_{a,1}$. It is well-known that $a \mapsto \rho(a)$ is continuous, non-decreasing, maps $[0,1)$ onto itself and is such that the interval $\rho^{-1}(\theta) \subset [0,1)$ degenerates to a point whenever $\theta \in [0,1) \setminus \Q$ (see \cite{hermanihes}). Moreover the set $\big\{a\in[0,1):\rho(a)\in\R\setminus\Q\big\}$ has zero Lebesgue measure, see \cite{swiatek}. For $0 \leq p < q$ coprime integers we know that $\rho^{-1}\big(\{\frac{p}{q}\}\big)$ is always a non-degenerate closed interval. In the interior of this interval we find critical circle maps with two periodic orbits (of period $q$), one attracting and one repelling, which collapse to a single parabolic orbit in the boundary of the interval, see \cite{epskeentres}.

For $b>1$ the maps $f_{a,b}:S^1 \to S^1$ are not invertible any more (they present two critical points of even degree). These examples show how critical circle maps arise as bifurcations from circle diffeomorphisms to endomorphisms, and in particular, from zero to positive topological entropy (compare with infinitely renormalizable unimodal maps \cite[Chapter VI]{livrounidim}). This is one of the main reasons why critical circle maps attracted the attention of physicists and mathematicians interested in understanding the \emph{boundary of chaos} (\cite{dgk}, \cite{feigetal}, \cite{ksh}, \cite{lanford1}, \cite{lanford2}, \cite{mckay} \cite{mackay}, \cite{ostlundetal}, \cite{rand1}, \cite{rand2}, \cite{rand3}, \cite{sh}).

Another important class of critical circle maps is provided by the one-parameter family $f_{\gamma}:\C\to\C$ of Blaschke products in the complex plane:

\begin{equation}\label{formBlaschke}
f_{\gamma}(z)=e^{2\pi i\gamma}z^2\left(\frac{z-3}{1-3z}\right)\quad\mbox{for $\gamma\in[0,1)$.}
\end{equation}

Every map in this family leaves invariant the unit circle (Blaschke products are the rational maps leaving invariant the unit circle), and its restriction to $S^1$ is a real-analytic homeomorphism with a unique critical point at $1$, which is of cubic type (see Figure 1). Furthermore, for each irrational number $\theta$ in $[0,1)$ there exists a unique $\gamma$ in $[0,1)$ such that the rotation number of $f_{\gamma}|_{S^1}$ is $\theta$. With this family at hand, the developments on rigidity of critical circle maps were very useful in the study of local connectivity and Lebesgue measure of Julia sets associated to generic quadratic polynomials with Siegel disks (\cite{petersen}, \cite{mcacta}, \cite{yampolsky1}, \cite{petersenzakeri}).

\begin{figure}[ht!]
\begin{center}
\includegraphics[scale=0.6]{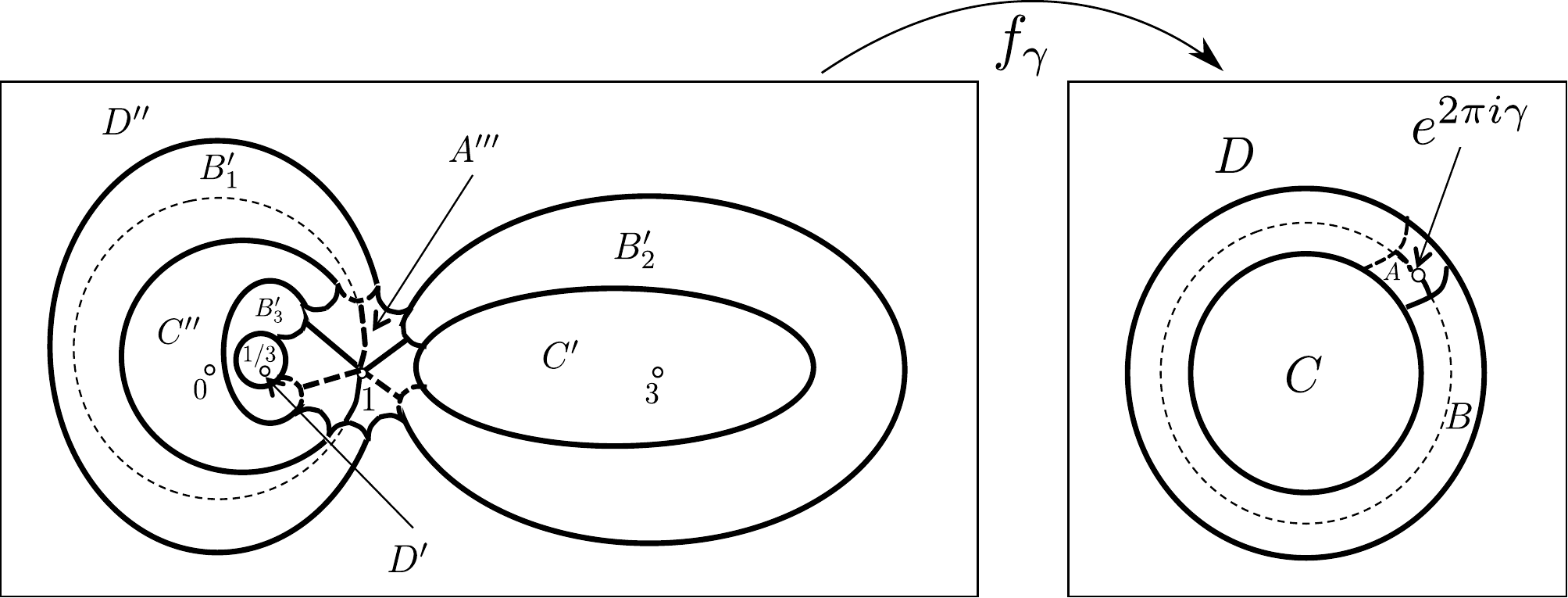}
\caption{Topological behaviour of the Blaschke product $f_{\gamma}$ \eqref{formBlaschke} around the unit circle, for $\gamma$ approximately equal to $1/8$. At the left of Figure 1 we see the preimage under $f_{\gamma}$ of the annulus around the unit circle drawn at the right (in both planes, the unit circle is dashed). The complement of the annulus $A \cup B$ in the complex plane has two connected components, $C$ and $D$. The preimage of $C$ is the union $C' \cup C''$, where the notation $C'$ means that $f_{\gamma}:C' \to C$ has topological degree $1$ (equivalently $f_{\gamma}:C'' \to C$ has topological degree $2$). In the same way, the preimage of $D$ is the union $D' \cup D''$, the preimage of $B$ is $B_1' \cup B_2' \cup B_3'$ and the preimage of $A$ is $A'''$.}
\end{center}
\end{figure}

Since our goal is to study smoothness of conjugacies we will focus on critical circle maps without periodic orbits, that is, the ones with irrational rotation number. In \cite{yoccoz} Yoccoz proved that the rotation number is the unique invariant of the topological classes. More precisely, any $C^3$ orientation preserving circle homeomorphism presenting only non-flat critical points (maybe more than one) and with irrational rotation number is topologically conjugate to the corresponding rigid rotation.

From the topological rigidity we get that any $C^3$ critical circle map with irrational rotation number is minimal and therefore the support of its unique invariant Borel probability measure is the whole circle. However let us point out that this invariant measure is always singular with respect to Lebesgue measure (see \cite[Theorem 4, page 182]{khanin} or \cite[Proposition 1, page 219]{singular}). We remark also that the condition of non-flatness on the critical points cannot be removed: in \cite{hall} Hall was able to construct $C^\infty$ homeomorphisms of the circle with no periodic points and no dense orbits.

Recall that an irrational number is of \emph{bounded type} if it satisfies the Diophantine condition (\ref{diof}) for $\delta=0$, that is, $\theta$ in $[0,1]$ is of bounded type if there exists $C>0$ such that:$$\left|\theta-\frac{p}{q}\right|\geq\frac{C}{q^2}\,,$$for any integers $p$ and $q \neq 0$. On one hand this is a respectable class: the set of numbers of bounded type is dense in $[0,1]$, with Hausdorff dimension equal to one. On the other hand, from the metrical viewpoint, this is a rather restricted class: while Diophantine numbers have full Lebesgue measure in $[0,1]$, the set of numbers of bounded type has zero Lebesgue measure.

Since a critical circle map cannot be smoothly conjugate to a rigid rotation, in order to study smooth-rigidity problems we must restrict to the class of critical circle maps. Numerical observations (\cite{feigetal}, \cite{ostlundetal}, \cite{sh}) suggested in the early eighties that smooth critical circle maps with rotation number of bounded type are geometrically rigid. This was posed as a conjecture in several works by Lanford (\cite{lanford1}, \cite{lanford2}), Rand (\cite{rand1}, \cite{rand2} and \cite{rand3}, see also \cite{ostlundetal}) and Shenker (\cite{sh}, see also \cite{feigetal}) among others:

\begin{Lan} Any two $C^3$ critical circle maps with the same irrational rotation number of bounded type and the same odd criticality are conjugate to each other by a $C^{1+\alpha}$ circle diffeomorphism, for some $\alpha>0$.
\end{Lan}

The conjecture has been proved by de Faria and de Melo for real-analytic critical circle maps \cite{edsonwelington2} and nowadays (after the work of Yampolsky, Khanin and Teplinsky) it is understood without any assumption on the irrational rotation number: inside each topological class of real-analytic critical circle maps the degree of the critical point is the unique invariant of the $C^1$-conjugacy classes. In the following result we summarize many contributions of the authors quoted above:

\begin{maintheorem}[de Faria-de Melo, Khmelev-Yampolsky, Khanin-Teplinsky]\label{riganalitico} Let $f$ and $g$ be two real-analytic circle homeomorphisms with the same irrational rotation number and with a unique critical point of the same odd type. Let $h$ be the conjugacy between $f$ and $g$ (given by Yoccoz's result) that maps the critical point of $f$ to the critical point of $g$ (note that this determines $h$). Then:

\begin{enumerate}
\item\label{Aitem1} $h$ is a $C^1$ diffeomorphism.
\item\label{Aitem2} $h$ is $C^{1+\alpha}$ at the critical point of $f$ for a universal $\alpha>0$.
\item\label{Aitem3} For a full Lebesgue measure set of rotation numbers (that contains all bounded type numbers) $h$ is globally $C^{1+\alpha}$.
\end{enumerate}
\end{maintheorem}

On one hand, the presence of the critical point gives us more rigidity than in the case of diffeomorphisms: smooth conjugacy is obtained for all irrational rotation numbers, with no Diophantine conditions. On the other hand, there exist examples (\cite{avila}, \cite{edsonwelington1}) showing that $h$ may not be globally $C^{1+\alpha}$ in general, even for real-analytic dynamics.

Item \eqref{Aitem1} of Theorem \ref{riganalitico} was proved by Khanin and Teplinsky in \cite{khaninteplinsky1}, building on earlier work of de Faria, de Melo and Yampolsky (\cite{tesisedson}, \cite{edson}, \cite{edsonwelington1}, \cite{edsonwelington2}, \cite{yampolsky1}, \cite{yampolsky2}, \cite{yampolsky3}, \cite{yampolsky4}). Item \eqref{Aitem2} was proved in \cite{khmelevyampolsky} and Item \eqref{Aitem3} is obtained combining \cite{edsonwelington1} with \cite{yampolsky4}. The proof of Theorem \ref{riganalitico} relies on methods coming from complex analysis and complex dynamics (\cite{mclivro2}, \cite{mcicm}), and that is why rigidity is well understood for real-analytic critical circle maps, but nothing was known yet for smooth ones (even in the $C^{\infty}$ setting). In this article we take the final step and solve positively the Rigidity Conjecture:

\begin{maintheorem}[Main result]\label{main} Any two $C^3$ critical circle maps with the same irrational rotation number of bounded type and the same odd criticality are conjugate to each other by a $C^{1+\alpha}$ circle diffeomorphism, for some universal $\alpha>0$.
\end{maintheorem}

The novelties of this article in order to transfer rigidity from real-analytic dynamics to (finitely) smooth ones are two: the first one is a \emph{bidimensional} version of the \emph{glueing procedure} (first introduced by Lanford \cite{lanford1}, \cite{lanford2}) developed in Section \ref{Secglueing}, and the second one is the notion of \emph{asymptotically holomorphic maps}, to be defined in Section \ref{Seccombounds} (Definition \ref{ashol}). Asymptotically holomorphic maps were already used in one-dimensional dynamics by Graczyk, Sands and \'Swi\c{a}tek in \cite{grasandsswia}, but as far as we know never for critical circle maps.

Let us discuss the main ideas of the proof of Theorem \ref{main}: a $C^3$ critical circle map $f$ with irrational rotation number generates a sequence $\big\{\mathcal{R}^n(f)=(\eta_n,\xi_n)\big\}_{n\in\nt}$ of commuting pairs of interval maps, each one being the renormalization of the previous one (see Definition \ref{renop}). To prove Theorem \ref{main} we need to prove the exponential convergence of the orbits generated by two critical circle maps with a given combinatorics of bounded type (see Theorem \ref{primeroedsonwelington}).

Our main task (see Theorem \ref{compacto} in Section \ref{Secred}) is to show the existence of a sequence $\big\{f_n=(\widetilde{\eta}_n,\widetilde{\xi}_n)\big\}_{n\in\nt}$ that belongs to a universal $C^{\omega}$-compact set of real-analytic critical commuting pairs, such that $\mathcal{R}^n(f)$ is $C^0$-exponentially close to $f_n$ at a universal rate, and both have the same rotation number. In Section \ref{Secred}, using the exponential contraction of the renormalization operator on the space of real-analytic critical commuting pairs (see Theorem \ref{uniform}), we conclude the exponential contraction of the renormalization operator in the space of $C^3$ critical commuting pairs with bounded combinatorics (see Theorem \ref{expconv} in Section \ref{Secfirstred}), and therefore the $C^{1+\alpha}$ rigidity as stated in Theorem \ref{main}.

To realize the main task we extend the initial commuting pair to a pair of $C^3$ maps in an open complex neighbourhood of each original interval (the so-called \emph{extended lift}, see Definition \ref{canext}), that are asymptotically holomorphic (see Definition \ref{ashol}), each having a unique cubic critical point at the origin.

Using the real bounds (see Theorem \ref{realbounds}), the Almost Schwarz inclusion (see Proposition \ref{schwarz}) and the asymptotic holomorphic property we prove that for all $n\in\nt$, greater or equal than some $n_0$, both $\eta_n$ and $\xi_n$ extend to a definite neighbourhood of their interval domains in the complex plane, giving rise to maps with a unique cubic critical point at the origin, and with exponentially small conformal distortion (see Theorem \ref{nuevoE}). Theorem \ref{nuevoE} gives us also some geometric control that will imply the desired compactness (we wont study the dynamics of these extensions, just their geometric behaviour).

Using Ahlfors-Bers theorem (see Proposition \ref{aprox2}) we construct for each $n \geq n_0$ a $C^3$ diffeomorphism $\Phi_n$, exponentially close to the identity in definite domains around the dynamical intervals, that conjugates $(\eta_n,\xi_n)$ to a $C^3$ critical commuting pair $(\widehat{\eta}_n,\widehat{\xi}_n)$ exponentially close to $(\eta_n,\xi_n)$, and such that $\widehat{\eta}_n^{-1}\circ\widehat{\xi}_n$ is an holomorphic diffeomorphism between complex neighbourhoods of the endpoints of the union of the dynamical intervals (see Subsection \ref{sub1}). Using this holomorphic diffeomorphism to glue the ends of a band around the union of the dynamical intervals we obtain a Riemann surface conformally equivalent to a rounds annulus $A_{R_n}$ around the unit circle. This identification gives rise to a holomorphic local diffeomorphism $P_n$ mapping the band onto the annulus and such that, via $P_n$, the pair $(\widehat{\eta}_n,\widehat{\xi}_n)$ induces a $C^3$ map $G_n$ from an annulus in $A_{R_n}$ to $A_{R_n}$, having exponentially small conformal distortion, that restricts to a critical circle map on $S^1$ (see Proposition \ref{Gn}). The commuting condition of each pair $(\widehat{\eta}_n,\widehat{\xi}_n)$ is equivalent to the continuity of the corresponding $G_n$, and that is why we project to the annulus $A_{R_n}$. The topological behaviour of each $G_n$ on its annular domain is the same as the restriction of the Blaschke product $f_{\gamma}$ \eqref{formBlaschke} to the annulus $A''' \cup B_1'$, as depicted in Figure 1.
	
Using again Ahlfors-Bers theorem we construct a holomorphic map $H_n$, on a smaller but definite annulus around the unit circle, that is exponentially close to $G_n$ and restricts to a real-analytic critical circle map with the same combinatorics as the restriction of $G_n$ to $S^1$ (see Proposition \ref{pertanillo} for much more properties).

Finally, using the projection $P_n$, we lift each $H_n$ to a real-analytic critical commuting pair $f_n=(\widetilde{\eta}_n,\widetilde{\xi}_n)$ exponentially close to $(\widehat{\eta}_n,\widehat{\xi}_n)$, having the same combinatorics and with complex extensions $C^0$-exponentially close to the ones of $\mathcal{R}^{n}(f)$ produced in Theorem \ref{nuevoE} (see Proposition \ref{shadow}). Compactness follows then from the geometric properties obtained in Theorem \ref{nuevoE} (see Lemma \ref{compact}).

The organization of this article is the following: in Section \ref{Secfirstred} we reduce Theorem \ref{main} to Theorem \ref{expconv}, which states the exponential convergence of the renormalization orbits of $C^3$ critical circle maps with the same bounded combinatorics. In Section \ref{Secren} we introduce the renormalization operator in the space of critical commuting pairs, and review its basic properties. In Section \ref{Secred} we reduce Theorem \ref{expconv} to Theorem \ref{compacto}, which states the existence of a $C^{\omega}$-compact piece of real-analytic critical commuting pairs such that for a given $C^3$ critical circle map $f$, with any irrational rotation number, there exists a sequence $\big\{f_n\big\}$, contained in that compact piece, such that $\mathcal{R}^n(f)$ is $C^0$-exponentially close to $f_n$ at a universal rate, and both have the same rotation number. In Section \ref{SecAB} we state a corollary of Ahlfors-Bers theorem (Proposition \ref{aprox2}) that will be fundamental in Section \ref{Secglueing} (its proof will be given in Appendix \ref{apA}). In Section \ref{Seccombounds} we construct the \emph{extended lift} of a $C^3$ critical circle map (see Definition \ref{canext}), and then we state and prove Theorem \ref{nuevoE} as described above. In Section \ref{Secglueing} we develop a bidimensional glueing procedure in order to prove Theorem \ref{compacto}. Finally in Section \ref{Secfinal} we review further questions and open problems in the area.

\section{A first reduction of the main result}\label{Secfirstred}

As in the case of unimodal maps, the main tool in order to obtain smooth conjugacy between critical circle maps is the use of \emph{renormalization} group methods \cite{dmicm}. As it was already clear in the early eighties (\cite{feigetal}, \cite{ostlundetal}) it is convenient to construct a renormalization operator $\mathcal{R}$ (see Definition \ref{renop}) acting not on the space of critical circle maps but on a suitable space of \emph{critical commuting pairs} (see Definition \ref{critpair}).

Just as in the case of unimodal maps (see for instance \cite[Chapter VI, Theorem 9.4]{livrounidim}), the principle that \emph{exponential convergence of the renormalization operator is equivalent to smooth conjugacy} also holds for critical circle maps. The following result is due to de Faria and de Melo \cite[First Main Theorem, page 341]{edsonwelington1}. For any $0 \leq r < \infty$ denote by $d_r$ the $C^r$ metric in the space of critical commuting pairs (see Definition \ref{Crmetric}):

\begin{theorem}[de Faria-de Melo 1999]\label{primeroedsonwelington} There exists a set $\A$ in $[0,1]$, having full Lebesgue measure and containing all irrational numbers of bounded type, for which the following holds: let $f$ and $g$ be two $C^3$ critical circle maps with the same irrational rotation number in the set $\A$ and with the same odd type at the critical point. If $d_0\big(\mathcal{R}^n(f),\mathcal{R}^n(g)\big)$ converge to zero exponentially fast when $n$ goes to infinity, then $f$ and $g$ are $C^{1+\alpha}$ conjugate to each other for some $\alpha>0$.
\end{theorem}

Roughly speaking, the full Lebesgue measure set $\A$ is composed by irrational numbers in $[0,1]$ whose coefficients in the continued fraction expansion may be unbounded, but their growth is less than quadratic (see Section \ref{Secfinal} or \cite[Appendix C]{edsonwelington1} for the precise definition). In sharp contrast with the case of diffeomorphisms, let us point out that $\A$ does not contain all Diophantine numbers, and contains some Liouville numbers (again see Section \ref{Secfinal}). The remaining cases were more recently solved by Khanin and Teplinsky \cite[Theorem 2, page 198]{khaninteplinsky1}:

\begin{theorem}[Khanin-Teplinsky 2007]\label{condition} Let $f$ and $g$ be two $C^3$ critical circle maps with the same irrational rotation number and the same odd type at the critical point. If $d_2\big(\mathcal{R}^n(f),\mathcal{R}^n(g)\big)$ converge to zero exponentially fast when $n$ goes to infinity, then $f$ and $g$ are $C^1$-conjugate to each other.
\end{theorem}

To obtain the smooth conjugacy (Item \eqref{Aitem1} of Theorem \ref{riganalitico}), Khanin and Teplinsky combined Theorem \ref{condition} with the following fundamental resut:

\begin{theorem}[de Faria-de Melo 2000, Yampolsky 2003]\label{uniform} There exists a universal constant $\lambda$ in $(0,1)$ with the following property: given two real-analytic critical commuting pairs $\zeta_1$ and $\zeta_2$ with the same irrational rotation number and the same odd type at the critical point, there exists a constant $C>0$ such that:$$d_{r}\big(\mathcal{R}^n(\zeta_1),\mathcal{R}^n(\zeta_2)\big) \leq C\lambda^n$$for all $n\in\nt$ and for any $0 \leq r < \infty$. Moreover given a $C^{\omega}$-compact set $\mathcal{K}$ of real-analytic critical commuting pairs, the constant $C$ can be chosen the same for any $\zeta_1$ and $\zeta_2$ in $\mathcal{K}$.
\end{theorem}

Theorem \ref{uniform} was proved by de Faria and de Melo \cite{edsonwelington2} for rotation numbers of bounded type, and extended by Yampolsky \cite{yampolsky4} to cover all irrational rotation numbers.

With Theorem \ref{primeroedsonwelington} at hand, our main result (Theorem \ref{main}) reduces to the following one:

\begin{maintheorem}\label{expconv} There exists $\lambda \in (0,1)$ such that given $f$ and $g$ two $C^3$ critical circle maps with the same irrational rotation number of bounded type and the same criticality, there exists $C>0$ such that for all $n\in\nt$:$$d_0\big(\mathcal{R}^n(f),\mathcal{R}^n(g)\big) \leq C\lambda^n\,,$$where $d_0$ is the $C^0$ distance in the space of critical commuting pairs.
\end{maintheorem}

This article is devoted to proving Theorem \ref{expconv}. Of course it would be desirable to obtain Theorem \ref{expconv} for $C^3$ critical circle maps with any irrational rotation number, but we have not been able to do this yet (see Section \ref{Secfinal} for more comments).

Let us fix some notation that we will use along this article: $\nt$, $\Z$, $\Q$, $\R$ and $\C$ denotes respectively the set of natural, integer, rational, real and complex numbers. The real part of a complex number $z$ will be denoted by $\Re(z)$, and its imaginary part by $\Im(z)$. $B(z,r)$ denotes the Euclidean open ball of radius $r>0$ around a complex number $z$. $\h$ and $\widehat{\C}$ denotes respectively the upper-half plane and the Riemann sphere. $\D=B(0,1)$ denotes the unit disk in the complex plane, and $S^1=\partial\,\D$ denotes its boundary, that is, the unit circle. $\Diff_{+}^{3}(S^1)$ denotes the group (under composition) of orientation-preserving $C^3$ diffeomorphisms of the unit circle. $\Leb(A)$ denotes the Lebesgue measure of a Borel set $A$ in the plane, and $\diam(A)$ denotes its Euclidean diameter. Given a bounded interval $I$ in the real line we denote its Euclidean length by $|I|$. Moreover, for any $\alpha>0$, let:$$N_{\alpha}(I)=\big\{z\in\C:d(z,I)<\alpha|I|\big\},$$where $d$ denotes the Euclidean distance in the complex plane.

\section{Renormalization of critical commuting pairs}\label{Secren}

In this section we define the space of $C^3$ critical commuting pairs (Definition \ref{critpair}), and we endow it with the $C^3$ metric (Definition \ref{Crmetric}). This metric space, which is neither compact nor locally-compact, contains the phase space of the renormalization operator (Definition \ref{renop}). Each $C^3$ critical circle map with irrational rotation number gives rise to an infinite renormalization orbit in this phase space, and the asymptotic behaviour of these orbits is the subject of this article.

We remark that, since there is no canonical differentiable structure (like a Banach manifold structure) in the space of $C^3$ critical commuting pairs endowed with the $C^3$ metric, we cannot apply the standard machinery from hyperbolic dynamics (see for instance \cite[Chapters 6, 18 and 19]{katokhass}) in order to obtain exponential convergence as stated in Theorem \ref{expconv}.

As we said in the introduction, a \emph{critical circle map} is an orientation-preserving $C^3$ circle homeomorphism $f$, with exactly one critical point $c \in S^1$ of odd type. For simplicity, and for being the generic case, we will assume in this article that the critical point is of cubic type. Suppose that the rotation number $\rho(f)=\theta$ in $[0,1)$ is irrational, and let $\big[a_0,a_1,...,a_n,a_{n+1},...\big]$ be its continued fraction expansion:$$\theta=\lim_{n\to+\infty}\dfrac{1}{a_0+\dfrac{1}{a_1+\dfrac{1}{a_2+\dfrac{1}{\ddots\dfrac{1}{a_n}}}}}$$

We define recursively the \emph{return times} of $\theta$ by:
\begin{center}
$q_0=1,\quad$ $\quad q_1=a_0\quad$ and $\quad q_{n+1}=a_nq_n+q_{n-1}\quad$ for $\quad n \geq 1$.
\end{center}

Recall that the numbers $q_n$ are also obtained as the denominators of the truncated expansion of order $n$ of $\theta$:$$\frac{p_n}{q_n}=[a_0,a_1,a_2,...,a_{n-1}]=\dfrac{1}{a_0+\dfrac{1}{a_1+\dfrac{1}{a_2+\dfrac{1}{\ddots\dfrac{1}{a_{n-1}}}}}}$$

Let $R_{\theta}$ be the rigid rotation of angle $2\pi\theta$ in the unit circle. The arithmetical properties of the continued fraction expansion imply that the iterates $\{R_{\theta}^{q_n}(c)\}_{n \in \nt}$ are the closest returns of the orbit of $c$ under the rotation $R_{\theta}$:$$d\big(c,R_{\theta}^{q_n}(c)\big)<d\big(c,R_{\theta}^j(c)\big)\quad\mbox{for any}\quad j\in\{1,...,q_n-1\}$$where $d$ denote the standard distance in $S^1$. The sequence of return times $\{q_n\}$ increase at least exponentially fast as $n \to \infty$, and the sequence of return distances $\{d(c,R_{\theta}^{q_n}(c))\}$ decrease to zero at least exponentially fast as $n \to \infty$. Moreover the sequence $\{R_{\theta}^{q_n}(c)\}_{n \in \nt}$ approach the point $c$ alternating the order:$$R_{\theta}^{q_{1}}(c)<R_{\theta}^{q_3}(c)<...<R_{\theta}^{q_{2k+1}}(c)<...<c<...<R_{\theta}^{q_{2k}}(c)<...<R_{\theta}^{q_2}(c)<R_{\theta}^{q_0}(c)$$

By Poincar\'e's result quoted at the beginning of the introduction, this information remains true at the combinatorial level for $f$: for any $n\in\nt$ the interval $[c,f^{q_n}(c)]$ contains no other iterates $f^j(c)$ for $j \in \{1,...,q_n-1\}$, and if we denote by $\mu$ the unique invariant Borel probability of $f$ we can say that $\mu\big([c,f^{q_n}(c)]\big)<\mu\big([c,f^j(c)]\big)$ for any $j \in \{1,...,q_n-1\}$. A priori we cannot say anything about the usual distance in $S^1$.

We say that $\rho(f)$ is of \emph{bounded type} if there exists a constant $M\in\nt$ such that $a_n < M$ for any $n\in\nt$ (it is not difficult to see that this definition is equivalent with the one given in the introduction, see \cite[Chapter II, Theorem 23]{khin}). As we said in the introduction, the set of numbers of bounded type has zero Lebesgue measure in $[0,1]$.

\subsection{Dynamical partitions} Denote by $I_n$ the interval $[c,f^{q_n}(c)]$ and define $\mathcal{P}_n$ as:$$\mathcal{P}_n=\big\{I_n,f(I_n),...,f^{q_{n+1}-1}(I_n)\big\} \bigcup \big\{I_{n+1},f(I_{n+1}),...,f^{q_n-1}(I_{n+1})\big\}$$

A crucial combinatorial fact is that $\mathcal{P}_n$ is a partition (modulo boundary points) of the circle for every $n \in \nt$. We call it the \emph{n-th dynamical partition} of $f$ associated with the point $c$. Note that the partition $\mathcal{P}_n$ is determined by the piece of orbit:$$\{f^{j}(c): 0 \leq j \leq q_n + q_{n+1}-1\}$$

The transitions from $\mathcal{P}_n$ to $\mathcal{P}_{n+1}$ can be described in the following easy way: the interval $I_n=[c,f^{q_n}(c)]$ is subdivided by the points $f^{jq_{n+1}+q_n}(c)$ with $1 \leq j \leq a_{n+1}$ into $a_{n+1}+1$ subintervals. This sub-partition is spreaded by the iterates of $f$ to all the $f^j(I_n)=f^j([c,f^{q_n}(c)])$ with $0 \leq j < q_{n+1}$. The other elements of the partition $\mathcal{P}_n$, which are the $f^j(I_{n+1})$ with $0 \leq j < q_n$, remain unchanged.

As we are working with critical circle maps, our partitions in this article are always determined by the critical orbit. A major result for critical circle maps is the following:

\begin{theorem}[real bounds]\label{realbounds} There exists $K>1$ such that given a $C^3$ critical circle map $f$ with irrational rotation number there exists $n_0=n_0(f)$ such that for all $n \geq n_0$ and for every pair $I,J$ of adjacent atoms of $\mathcal{P}_n$ we have:$$K^{-1}|I| \leq |J| \leq K|I|.$$
Moreover, if $Df$ denotes the first derivative of $f$, we have:$$\frac{1}{K}\leq\frac{\big|Df^{q_n-1}(x)\big|}{\big|Df^{q_n-1}(y)\big|}\leq K\quad\mbox{for all $x,y \in f(I_{n+1})$ and for all $n \geq n_0$, and:}$$
$$\frac{1}{K}\leq\frac{\big|Df^{q_{n+1}-1}(x)\big|}{\big|Df^{q_{n+1}-1}(y)\big|}\leq K\quad\mbox{for all $x,y \in f(I_{n})$ and for all $n \geq n_0$.}$$
\end{theorem}

Theorem \ref{realbounds} was proved by \'Swi\c{a}tek and Herman (see \cite{herman}, \cite{swiatek}, \cite{graswia} and \cite{edsonwelington1}). The control on the distortion of the return maps follows from Koebe distortion principle (see \cite[Section 3]{edsonwelington1}). Note that for a rigid rotation we have $|I_n|=a_{n+1}|I_{n+1}|+|I_{n+2}|$. If $a_{n+1}$ is big, then $I_n$ is much larger than $I_{n+1}$. Thus, even for rigid rotations, real bounds do not hold in general.

\subsection{Critical commuting pairs}\label{subccp} The first return map of the union of adjacent intervals $I_n \cup I_{n+1}$ is given respectively by $f^{q_{n+1}}$ and $f^{q_n}$. This pair of interval maps:$$\left(f^{q_{n+1}}|_{I_n},f^{q_n}|_{I_{n+1}}\right)$$motivates the following definition:

\begin{definition}\label{critpair} A \emph{critical commuting pair} $\zeta=(\eta,\xi)$ consists of two smooth orientation-preserving interval homeomorphisms $\eta:I_{\eta}\to\eta(I_{\eta})$ and $\xi:I_{\xi}\to\xi(I_{\xi})$ where:
\begin{enumerate}
\item $I_{\eta}=[0,\xi(0)]$ and $I_{\xi}=[\eta(0),0]$;
\item There exists a neighbourhood of the origin where both $\eta$ and $\xi$ have homeomorphic extensions (with the same degree of smoothness) which commute, that is, $\eta\circ\xi=\xi\circ\eta$;
\item $\big(\eta\circ\xi\big)(0)=\big(\xi\circ\eta\big)(0) \neq 0$;
\item $\eta'(0)=\xi'(0)=0$;
\item $\eta'(x) \neq 0$ for all $x \in I_{\eta}\setminus\{0\}$ and $\xi'(x) \neq 0$ for all $x \in I_{\xi}\setminus\{0\}$.
\end{enumerate}
\end{definition}

\begin{figure}[ht!]
\begin{center}
\includegraphics[scale=0.8]{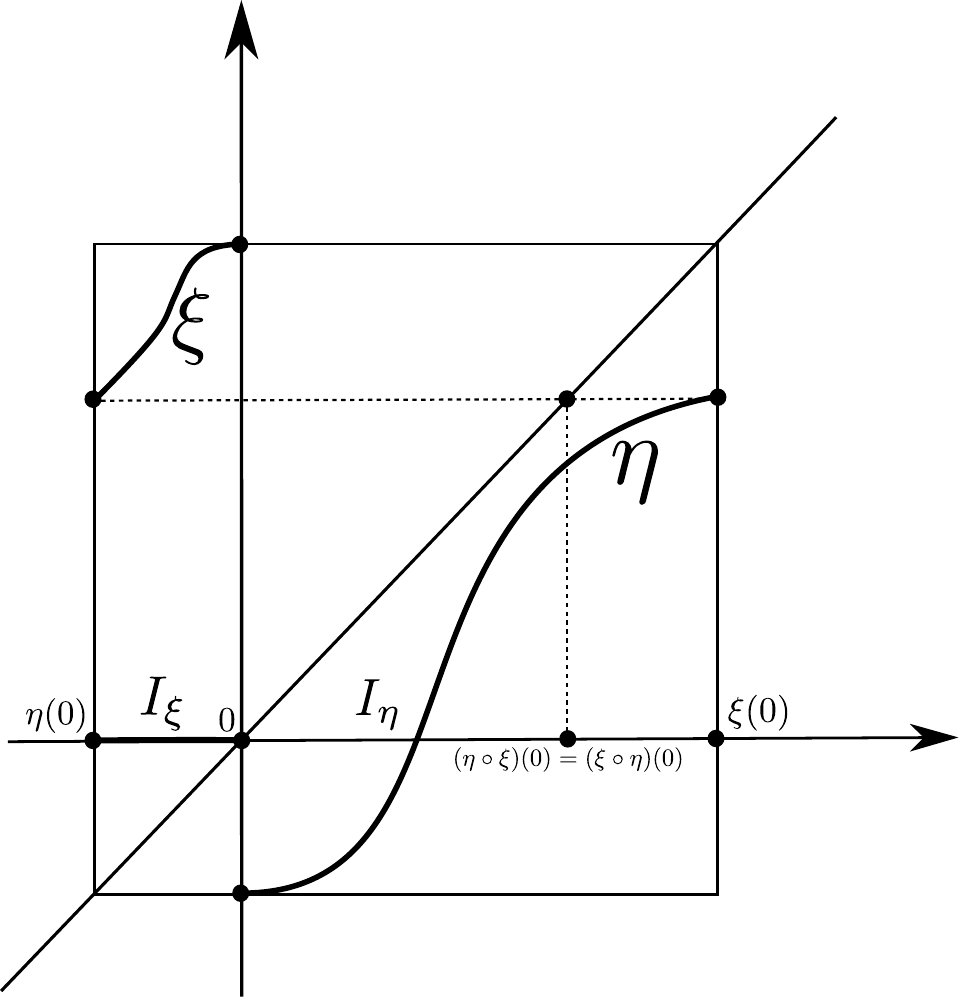}
\caption{A commuting pair.}
\end{center}
\end{figure}

Any critical circle map $f$ with irrational rotation number $\theta$ induces a sequence of critical commuting pair in a natural way: let $\widetilde{f}$ be the lift of $f$ to the real line (for the canonical covering $t \mapsto e^{2\pi it}$) satisfying $\widetilde{f}'(0)=0$ and $0<\widetilde{f}(0)<1$. For each $n \geq 1$ let $\widetilde{I_n}$ be the closed interval in the real line, adjacent to the origin, that projects to $I_n$. Let $T:\R \to \R$ be the translation $x \mapsto x+1$ and define $\eta:\widetilde{I_n} \to \R$ and $\xi:\widetilde{I_{n+1}} \to \R$ as:$$\eta=T^{-p_{n+1}}\circ\widetilde{f}^{q_{n+1}}\quad\mbox{and}\quad\xi=T^{-p_n}\circ\widetilde{f}^{q_n}$$

Then the pair $(\eta|_{\widetilde{I_n}}, \xi|_{\widetilde{I_{n+1}}})$ form a critical commuting pair that we denote by $(f^{q_{n+1}}|_{I_n},f^{q_n}|_{I_{n+1}})$ to simplify notation.

A converse of this construction was introduced by Lanford (\cite{lanford1}, \cite{lanford2}) and it is known as \emph{glueing procedure}: the map $\eta^{-1}\circ\xi$ is a diffeomorphism from a small neighbourhood of $\eta(0)$ onto a neighbourhood of $\xi(0)$. Identifying $\eta(0)$ and $\xi(0)$ in this way we obtain from the interval $\big[\eta(0),\xi(0)\big]$ a smooth compact boundaryless one-dimensional manifold $M$. The discontinuous piecewise smooth map:
$$f_{\zeta}(t)=\left\{\begin{array}{ll}
\xi(t)&\mbox{for } t\in\big[\eta(0),0\big)\\
\eta(t)&\mbox{for } t\in\big[0,\xi(0)\big]\\
\end{array}\right.$$projects to a smooth homeomorphism on the quotient manifold $M$. By choosing a diffeomorphism $\psi:M \to S^1$ we obtain a critical circle map in $S^1$, just by conjugating with $\psi$. Although there is no canonical choice for the diffeomorphism $\psi$, any two different choices give rise to smoothly-conjugate critical circle maps in $S^1$. Therefore any critical commuting pair induces a whole \emph{smooth conjugacy class} of critical circle maps. In Section \ref{Secglueing} we propose a bidimensional extension of this procedure, in order to prove our main result (Theorem \ref{main}).

\begin{figure}[ht!]
\begin{center}
\includegraphics[scale=0.5]{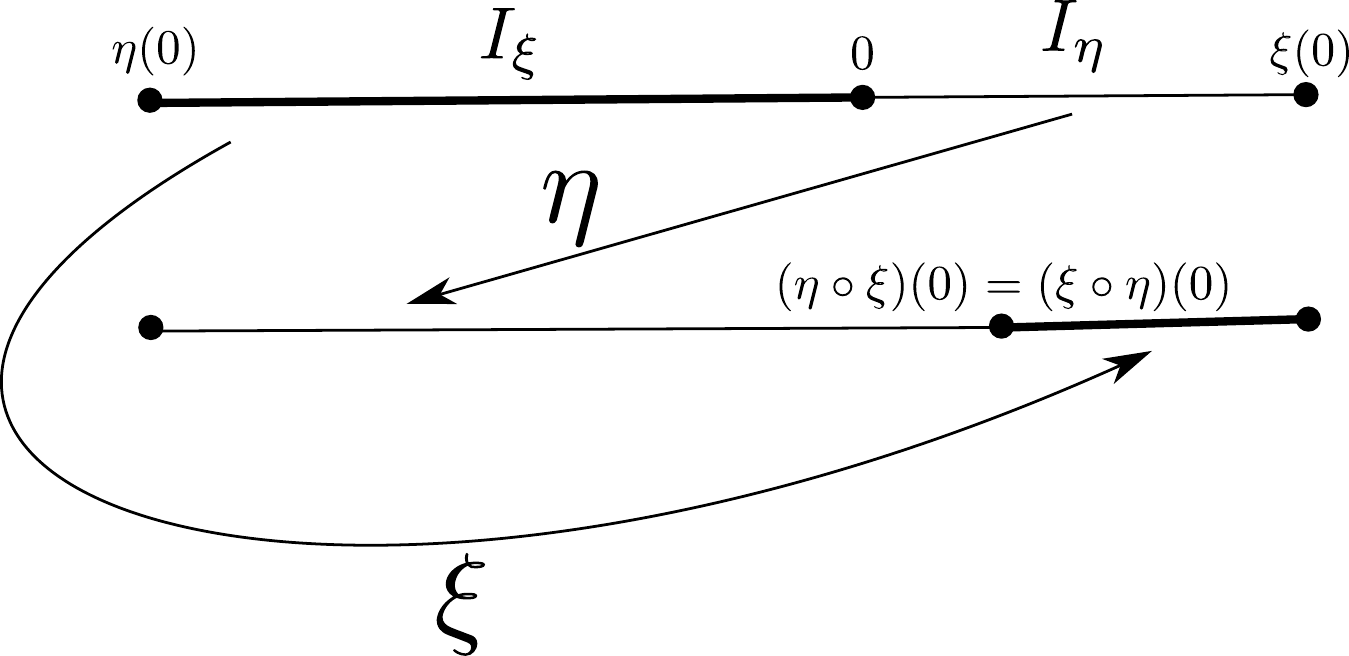}
\caption{Scheme of a commuting pair.}
\end{center}
\end{figure}

\subsection{The $C^r$ metric}\label{submetric} We endow the space of $C^3$ critical commuting pairs with the $C^3$ metric. Given two critical commuting pairs $\zeta_1=(\eta_1,\xi_1)$ and $\zeta_2=(\eta_2,\xi_2)$ let $A_1$ and $A_2$ be the M\"obius transformations such that for $i=1,2$:$$A_i\big(\eta_i(0)\big)=-1,\quad A_i(0)=0\quad\mbox{and}\quad A_i\big(\xi_i(0)\big)=1\,.$$

\begin{definition}\label{Crmetric} For any $0 \leq r < \infty$ define the $C^r$ metric on the space of $C^r$ critical commuting pairs in the following way:$$d_r(\zeta_1,\zeta_2)=\max\left\{\left|\frac{\xi_1(0)}{\eta_1(0)}-\frac{\xi_2(0)}{\eta_2(0)}\right|, \big\|A_1 \circ \zeta_1 \circ A_1^{-1}-A_2 \circ \zeta_2 \circ A_2^{-1}\big\|_r\right\}$$where $\|\cdot\|_r$ is the $C^r$-norm for maps in $[-1,1]$ with one discontinuity at the origin, and $\zeta_i$ is the piecewise map defined by $\eta_i$ and $\xi_i$:$$\zeta_i:I_{\xi_i} \cup I_{\eta_i} \to I_{\xi_i} \cup I_{\eta_i} \quad\mbox{such that}\quad \zeta_i|_{I_{\xi_i}}=\xi_i \quad\mbox{and}\quad \zeta_i|_{I_{\eta_i}}=\eta_i$$
\end{definition}

Note that $d_r$ is a pseudo-metric since it is invariant under conjugacy by homotheties: if $\alpha$ is a positive real number, $H_{\alpha}(t)=\alpha t$ and $\zeta_1=H_{\alpha} \circ \zeta_2 \circ H_{\alpha}^{-1}$ then $d_r(\zeta_1,\zeta_2)=0$. In order to have a metric we need to restrict to \emph{normalized} critical commuting pairs: for a commuting pair $\zeta=(\eta,\xi)$ denote by $\widetilde{\zeta}$ the pair $(\widetilde{\eta}|_{\widetilde{I_{\eta}}}, \widetilde{\xi}|_{\widetilde{I_{\xi}}})$ where tilde means rescaling by the linear factor $\lambda=\frac{1}{|I_{\xi}|}$. Note that $|\widetilde{I_{\xi}}|=1$ and $\widetilde{I_{\eta}}$ has length equal to the ratio between the lengths of $I_{\eta}$ and $I_{\xi}$. Equivalently $\widetilde{\eta}(0)=-1$ and $\widetilde{\xi}(0)=\frac{|I_{\eta}|}{|I_{\xi}|}=\xi(0)/\big|\eta(0)\big|$.

When we are dealing with real-analytic critical commuting pairs, we consider the $C^{\omega}$-topology defined in the usual way: we say that $\big(\eta_n,\xi_n\big)\to\big(\eta,\xi\big)$ if there exist two open sets $U_{\eta} \supset I_{\eta}$ and $U_{\xi} \supset I_{\xi}$ in the complex plane and $n_0\in\nt$ such that $\eta$ and $\eta_n$ for $n \geq n_0$ extend continuously to $\overline{U_{\eta}}$, are holomorphic in $U_{\eta}$ and we have $\big\|\eta_n-\eta\big\|_{C^0(\overline{U_{\eta}})} \to 0$, and such that $\xi$ and $\xi_n$ for $n \geq n_0$ extend continuously to $\overline{U_{\xi}}$, are holomorphic in $U_{\xi}$ and we have $\big\|\xi_n-\xi\big\|_{C^0(\overline{U_{\xi}})} \to 0$. We say that a set $\mathcal{C}$ of real-analytic critical commuting pairs is closed if every time we have $\{\zeta_n\}\subset\mathcal{C}$ and $\{\zeta_n\}\to\zeta$, we have $\zeta\in\mathcal{C}$. This defines a Hausdorff topology, stronger than the $C^r$-topology for any $0 \leq r \leq \infty$.

\subsection{The renormalization operator} Let $\zeta=(\eta,\xi)$ be a $C^3$ critical commuting pair according to Definition \ref{critpair}, and recall that $\big(\eta\circ\xi\big)(0)=\big(\xi\circ\eta\big)(0) \neq 0$. Let us suppose that $\big(\xi\circ\eta\big)(0) \in I_{\eta}$ (just as in both Figure 2 and Figure 3 above) and define the \emph{height} $\chi(\zeta)$ of the commuting pair $\zeta=(\eta,\xi)$ as $r$ if:$$\eta^{r+1}(\xi(0)) \leq 0 \leq \eta^r(\xi(0))$$and $\chi(\zeta)=\infty$ if no such $r$ exists (note that in this case the map $\eta|_{I_\eta}$ has a fixed point, so when we are dealing with commuting pairs induced by critical circle maps with irrational rotation number we have finite height). Note also that the height of the pair $(f^{q_{n+1}}|_{I_n},f^{q_n}|_{I_{n+1}})$ induced by a critical circle maps $f$ is exactly $a_{n+1}$, where $\rho(f)=[a_0,a_1,a_2,...,a_n,a_{n+1},...]$ (because the combinatorics of $f$ are the same as for the rigid rotation $R_{\rho(f)}$).

For a pair $\zeta=(\eta,\xi)$ with $\big(\xi\circ\eta\big)(0) \in I_{\eta}$ and $\chi(\zeta)=r<\infty$ we see that the pair:$$\big(\eta|_{[0,\eta^r(\xi(0))]},\eta^r \circ \xi|_{I_\xi}\big)$$is again a commuting pair, and if $\zeta=(\eta,\xi)$ is induced by a critical circle map:$$\zeta=(\eta,\xi)=\big(f^{q_{n+1}}|_{I_n},f^{q_n}|_{I_{n+1}}\big)$$we have that:$$\big(\eta|_{[0,\eta^r(\xi(0))]},\eta^r \circ \xi|_{I_\xi}\big)=\big(f^{q_{n+1}}|_{I_{n+2}},f^{q_{n+2}}|_{I_{n+1}}\big)$$

This motivates the following definition (the definition in the case $\big(\xi\circ\eta\big)(0) \in I_{\xi}$ is analogue):

\begin{definition}\label{renop} Let $\zeta=(\eta,\xi)$ be a critical commuting pair with $\big(\xi\circ\eta\big)(0) \in I_{\eta}$. We say that $\zeta$ is \emph{renormalizable} if $\chi(\zeta)=r<\infty$. In this case we define the \emph{renormalization} of $\zeta$ as the critical commuting pair:$$\mathcal{R}(\zeta)=\left(\widetilde{\eta}|_{\widetilde{[0,\eta^r(\xi(0))]}},\widetilde{\eta^r \circ \xi}|_{\widetilde{I_\xi}}\right)$$
\end{definition}

A critical commuting pair is a special case of a \emph{generalized interval exchange map} of two intervals, and the renormalization operator defined above is just the restriction of the \emph{Zorich accelerated version} of the \emph{Rauzy-Veech renormalization} for interval exchange maps (see for instance \cite{yoccoziem}). However we will keep in this article the classical terminology for critical commuting pairs.

\begin{figure}[ht!]
\begin{center}
\includegraphics[scale=1.0]{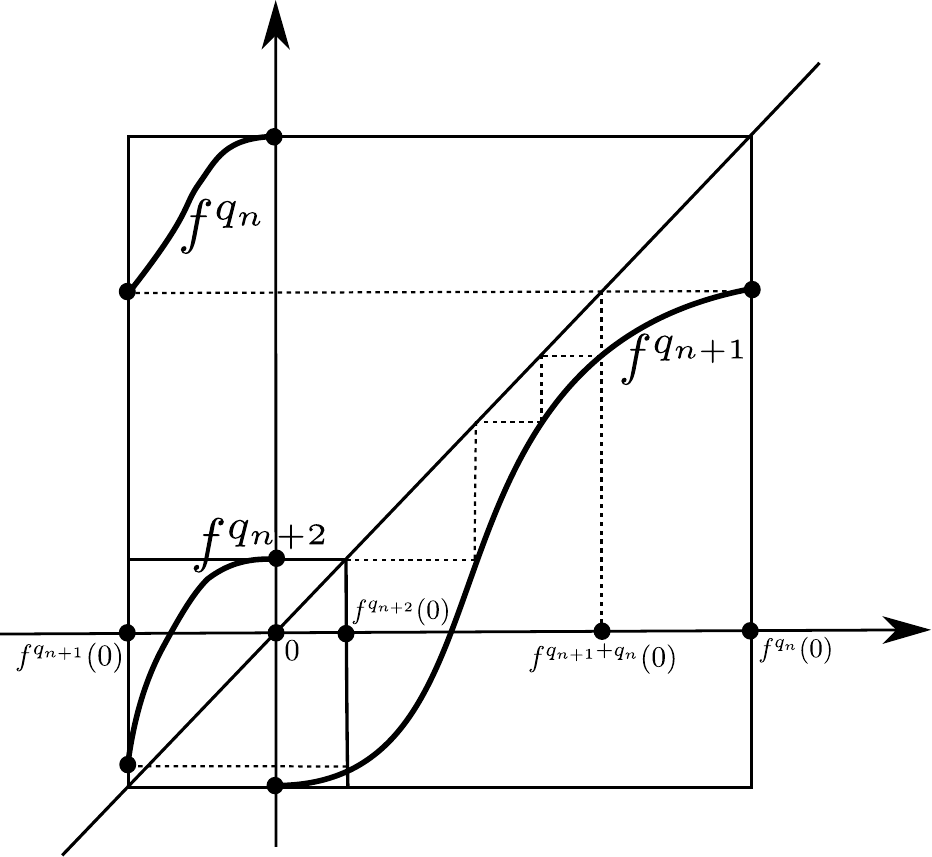}
\caption{Two consecutive renormalizations of $f$, without rescaling (recall that $f^{q_n}$ means $T^{-p_n}\circ\widetilde{f}^{q_n}$). In this example $a_{n+1}=4$.}
\end{center}
\end{figure}

\begin{definition}\label{rotnum} Let $\zeta$ be a critical commuting pair. If $\chi(\mathcal{R}^j(\zeta))<\infty$ for $j \in \{0,1,...,n-1\}$ we say that $\zeta$ is \emph{n-times renormalizable}, and if $\chi(\mathcal{R}^j(\zeta))<\infty$ for all $j\in\nt$ we say that $\zeta$ is \emph{infinitely renormalizable}. In this case the irrational number $\theta$ whose continued fraction expansion is equal to:$$\big[\chi\big(\zeta\big), \chi\big(\mathcal{R}(\zeta)\big),...,\chi\big(\mathcal{R}^n(\zeta)\big),\chi\big(\mathcal{R}^{n+1}(\zeta)\big),... \big]$$is called the \emph{rotation number} of the critical commuting pair $\zeta$, and denoted by $\rho(\zeta)=\theta$.
\end{definition}

The rotation number of a critical commuting pair can also be defined with the help of the glueing procedure described above, just as the rotation number of any representative of the conjugacy class obtained after glueing and uniformizing.

An immediate but very important remark is that when $\zeta$ is induced by a critical circle map with irrational rotation number, the pair $\zeta$ is automatically infinitely renormalizable (and both notions of rotation number coincide): any $C^3$ critical circle map $f$ with irrational rotation number gives rise to a well defined orbit $\big\{\mathcal{R}^n(f)\big\}_{n \geq 1}$ of infinitely renormalizable $C^3$ critical commuting pairs defined by:$$\mathcal{R}^n(f)=\left(\widetilde{f^{q_n}}|_{\widetilde{I_{n-1}}},\widetilde{f^{q_{n-1}}}|_{\widetilde{I_n}}\right)\quad\mbox{for all}\quad n \geq 1.$$
For any positive number $\theta$ denote by $\lfloor\theta\rfloor$ the \emph{integer part} of $\theta$, that is, $\lfloor\theta\rfloor\in\nt$ and $\lfloor\theta\rfloor\leq\theta<\lfloor\theta\rfloor+1$. Recall that the \emph{Gauss map} $G:[0,1]\to[0,1]$ is defined by:$$G(\theta)=\frac{1}{\theta}-\left\lfloor\frac{1}{\theta}\right\rfloor\quad\mbox{for}\quad\theta\neq 0\quad\mbox{and}\quad G(0)=0\,,$$and note that $\rho$ \emph{semi-conjugates} the renormalization operator with the Gauss map:$$\rho\big(\mathcal{R}^n(\zeta)\big)=G^n\big(\rho(f)\big)$$for any $\zeta$ at least $n$-times renormalizable. In particular the renormalization operator acts as a \emph{left shift} on the continued fraction expansion of the rotation number: if $\rho(\zeta)=[a_0,a_1,...]$ then $\rho\big(\mathcal{R}^n(\zeta)\big)=[a_n,a_{n+1},...]$.

\subsection{Lipschitz continuity along the orbits} For $K>1$ and $r\in\{0,1,...,\infty,\omega\}$ denote by $\mathcal{P}^r(K)$ the space of $C^r$ critical commuting pairs $\zeta=(\eta,\xi)$ such that $\eta(0)=-1$ (they are normalized) and $\xi(0) \in [K^{-1},K]$. Recall also that $T$ denotes the translation $t \mapsto t+1$ in the real line. Let $K_0>1$ be the universal constant given by the real bounds. In the next section we will use the following:

\begin{lema}\label{Lipschitz} Given $M>0$ and $K>K_0$ there exists $L>1$ with the following property: let $f$ be a $C^3$ critical circle map with irrational rotation number $\rho(f)=[a_0,a_1,...]$ satisfying $a_n<M$ for all $n\in\nt$. There exists $n_0=n_0(f)\in\nt$ such that for any $n \geq n_0$ and any renormalizable critical commuting pair $\zeta=(\eta,\xi)$ satisfying:
\begin{enumerate}
\item $\zeta,\mathcal{R}(\zeta)\in\mathcal{P}^3(K)$,
\item $$\left\lfloor\frac{1}{\rho(\zeta)}\right\rfloor=a_n\,,$$
\item\label{pololejos} If $\big(T^{-p_{n+1}}\circ\widetilde{f}^{q_{n+1}}\big)(0)<0<\big(T^{-p_n}\circ\widetilde{f}^{q_n}\big)(0)$ then:$$\left|\left|\frac{\big(T^{-p_n}\circ\widetilde{f}^{q_n}\big)(0)}{\big(T^{-p_{n+1}}\circ\widetilde{f}^{q_{n+1}}\big)(0)}\right|-\xi(0)\right|<\left(\frac{1}{K^2}\right)\left(\frac{K+1}{K-1}\right).$$
Otherwise, if $\big(T^{-p_n}\circ\widetilde{f}^{q_n}\big)(0)<0<\big(T^{-p_{n+1}}\circ\widetilde{f}^{q_{n+1}}\big)(0)$, then:$$\left|\left|\frac{\big(T^{-p_{n+1}}\circ\widetilde{f}^{q_{n+1}}\big)(0)}{\big(T^{-p_n}\circ\widetilde{f}^{q_n}\big)(0)}\right|-\xi(0)\right|<\left(\frac{1}{K^2}\right)\left(\frac{K+1}{K-1}\right),\quad\mbox{and}$$
\item\label{delmismolado} $\big(\eta\circ\xi\big)(0)$ and $\big(T^{-p_{n+1}-p_n}\circ\widetilde{f}^{q_{n+1}+q_n}\big)(0)$ have the same sign,
\end{enumerate}
then we have that:$$d_0\left(\mathcal{R}^{n+1}(f),\mathcal{R}(\zeta)\right) \leq L \cdot d_0\left(\mathcal{R}^n(f),\zeta\right),$$where $d_0$ is the $C^0$ distance in the space of critical commuting pairs.
\end{lema}

We postpone the proof of Lemma \ref{Lipschitz} until Appendix \ref{provaLip}.

\section{Reduction of Theorem \ref{expconv}}\label{Secred}

In this section we reduce Theorem \ref{expconv} to the following:

\begin{maintheorem}\label{compacto} There exist a $C^{\omega}$-compact set $\mathcal{K}$ of real-analytic critical commuting pairs and a constant $\lambda \in (0,1)$ with the following property: given a $C^3$ critical circle map $f$ with any irrational rotation number there exist $C>0$ and a sequence $\{f_n\}_{n\in\nt}$ contained in $\mathcal{K}$ such that:$$d_{0}\big(\mathcal{R}^n(f),f_n\big)\leq C\lambda^n\quad\mbox{for all $n\in\nt$,}$$and such that the pair $f_n$ has the same rotation number as the pair $\mathcal{R}^n(f)$ for all $n\in\nt$.
\end{maintheorem}

Note that $\mathcal{K}$ is $C^r$-compact for any $0 \leq r \leq \infty$ (see Section \ref{submetric}). Note also that Theorem \ref{compacto} is true for \emph{any} combinatorics. The following argument was inspired by \cite{demelopinto}:

\begin{proof}[Proof that Theorem \ref{compacto} implies Theorem \ref{expconv}] Let $\mathcal{K}$ be the $C^{\omega}$-compact set of real-analytic critical commuting pairs given by Theorem \ref{compacto}. By the real bounds there exists a uniform constant $n_0\in\nt$ such that $\mathcal{R}^n(\zeta)\in\mathcal{P}^{\omega}(K_0)$ for all $\zeta\in\mathcal{K}$ and all $n \geq n_0$. Therefore there exists $K>K_0$ such that $\mathcal{R}^n(\zeta)\in\mathcal{P}^{\omega}(K)$ for all $\zeta\in\mathcal{K}$ and all $n \geq 1$. Let $M>\max_{n\in\nt}\{a_n\}$ where $\rho(f)=\rho(g)=[a_0,a_1,...]$, and let $L>1$ given by Lemma \ref{Lipschitz}.

By Theorem \ref{compacto} there exist constants $\lambda_1 \in (0,1)$, $C_1(f),C_1(g)>0$ and two sequences $\{f_n\}_{n\in\nt}$ and $\{g_n\}_{n\in\nt}$ contained in $\mathcal{K}$ such that for all $n\in\nt$ we have $\rho(f_n)=\rho(g_n)=[a_n,a_{n+1},...]$ and:

\begin{equation}\label{perto}
d_0\big(\mathcal{R}^n(f),f_n\big)\leq C_1(f)\lambda_1^n \quad\mbox{and}\quad d_0\big(\mathcal{R}^n(g),g_n\big)\leq C_1(g)\lambda_1^n\,.
\end{equation}

Let $n_0(f),n_0(g)\in\nt$ given by Lemma \ref{Lipschitz}, and consider $n_0=\max\big\{n_0(f),n_0(g)\big\}$ and also $C_1=\max\big\{C_1(f),C_1(g)\big\}$. Fix $\alpha\in(0,1)$ such that $\alpha>\frac{\log L}{\log L-\log\lambda_1}$, and for all $n > (1/\alpha)n_0$ let $m=\lfloor\alpha n\rfloor$. By the choice of $K>K_0$, and since $f_m,g_m\in\mathcal{K}$ for all $m\in\nt$, we have that $\mathcal{R}^j(f_m)\in\mathcal{P}^3(K)$ for all $j\in\nt$. By the real bounds:$$\left|\frac{\big(T^{-p_{n+1}-p_n}\circ\widetilde{f}^{q_{n+1}+q_n}\big)(0)}{\big(T^{-p_{n+1}}\circ\widetilde{f}^{q_{n+1}}\big)(0)}\right|\in\left[\frac{1}{K},K\right]\quad\mbox{for all $n \geq n_0$}\,,$$and by \eqref{perto} we have Item \eqref{pololejos} and Item (\ref{delmismolado}) of Lemma \ref{Lipschitz} for $\zeta=f_n$, by taking $n_0$ big enough. Applying Lemma \ref{Lipschitz} we obtain:
\begin{align}
d_0\big(\mathcal{R}^n(f),\mathcal{R}^{n-m}(f_m)\big)&\leq L^{n-m}\cdot d_0\big(\mathcal{R}^m(f),f_m\big)\notag\\
&\leq C_1L^{n-m}\lambda_{1}^{m}\notag
\end{align}
and by the same reasons:
\begin{align}
d_0\big(\mathcal{R}^n(g),\mathcal{R}^{n-m}(g_m)\big)&\leq L^{n-m}\cdot d_0\big(\mathcal{R}^m(g),g_m\big)\notag\\
&\leq C_1L^{n-m}\lambda_{1}^{m}\,.\notag
\end{align}
Let $\lambda_2=L^{1-\alpha}\lambda_1^{\alpha}$, and note that $\lambda_2\in(0,1)$ by the choice of $\alpha$. Consider also $C_2=2C_1(1/\lambda_1)L>0$. Since $f_m$ and $g_m$ are real-analytic and they have the same combinatorics, we know by Yampolsky's result (Theorem \ref{uniform}) that there exist constants $\lambda_3 \in (0,1)$ and $C_3>0$ (uniform in $\mathcal{K}$) such that:$$d_0\big(\mathcal{R}^{n-m}(f_m),\mathcal{R}^{n-m}(g_m)\big)\leq C_3\lambda_3^{n-m}\quad\mbox{for all $n\in\nt$.}$$
Finally consider $C=C_2+C_3>0$ and $\lambda=\max\{\lambda_2,\lambda_3^{1-\alpha}\}\in(0,1)$. By the triangle inequality:$$d_0\big(\mathcal{R}^n(f),\mathcal{R}^n(g)\big) \leq C\lambda^n\quad\mbox{for all $n\in\nt$.}$$
\end{proof}

\section{Approximation by holomorphic maps.}\label{SecAB}

\subsection{The Beltrami equation} Until now we were working on the real line, now we start to work on the complex plane. We assume that the reader is familiar with the notion of \emph{quasiconformality} (the book of Ahlfors \cite{ahlfors} and the one of Lehto and Virtanen \cite{lehtovirt} are classical references of the subject). Recall the two basic differential operators of complex calculus:$$\frac{\partial}{\partial z}=\left(\frac{1}{2}\right)\left(\frac{\partial}{\partial x}-i\frac{\partial}{\partial y}\right)\quad\mbox{and}\quad\frac{\partial}{\partial\overline{z}}=\left(\frac{1}{2}\right)\left(\frac{\partial}{\partial x}+i\frac{\partial}{\partial y}\right)$$

Instead of $\frac{\partial F}{\partial z}$ and $\frac{\partial F}{\partial\overline{z}}$ we will use the more compact notation $\partial F$ and $\overline{\partial}F$ respectively. To be more explicit, if $F:\Omega\subset\C\to\C$ is differentiable at $w\in\Omega$ then $\big(DF(w)\big)(z)=\partial F(w)z+\overline{\partial}F(w)\overline{z}$ for any $z\in\C$.

Recall that a continuous real function $h:\R\to\R$ is \emph{absolutely continuous} if it has derivative at Lebesgue almost every point, its derivative is integrable and $h(b)-h(a)=\int_{a}^{b}h'(t)dt$. A continuous function $F:\Omega\subset\C\to\C$ is \emph{absolutely continuous on lines} in $\Omega$ if its real and imaginary parts are absolutely continuous on Lebesgue almost every horizontal line, and Lebesgue almost every vertical line.

\begin{definition}\label{defanalqc} Let $\Omega\subset\C$ be a domain and let $K \geq 1$. An orientation-preserving homeomorphism $F:\Omega \to F(\Omega)$ is \emph{$K$-quasiconformal} (from now on $K$-q.c.) if $F$ is absolutely continuous on lines and:$$\left|\overline{\partial}F(z)\right|\leq\left(\frac{K-1}{K+1}\right)\big|\partial F(z)\big|\quad\mbox{a.e. in $\Omega$.}$$
\end{definition}

Given a $K$-q.c. homeomorphism $F:\Omega \to F(\Omega)$ we define its \emph{Beltrami coefficient} as the measurable function $\mu_F:\Omega\to\D$ given by:$$\mu_F(z)=\frac{\overline{\partial}F(z)}{\partial F(z)}\quad\mbox{for a.e. $z\in\Omega$.}$$
Note that $\mu_F$ belongs to $L^{\infty}(\Omega)$ and satisfy $\|\mu_F\|_{\infty} \leq (K-1)/(K+1)<1$. Conversely any measurable function from $\Omega$ to $\C$ with $L^{\infty}$ norm less than one is the Beltrami coefficient of a quasiconformal homeomorphism:

\begin{theorem}[Morrey 1938]\label{morrey} Given any measurable function $\mu:\Omega\to\D$ such that $|\mu(z)|\leq(K-1)/(K+1)<1$ almost everywhere in $\Omega$ for some $K \geq 1$, there exists a $K$-q.c. homeomorphism $F:\Omega \to F(\Omega)$ which is a solution of the \emph{Beltrami equation}:$$\partial F(z)\,\mu(z)=\overline{\partial}F(z)\mbox{ a.e.}.$$
The solution is unique up to post-composition with conformal diffeomorphisms. In particular, if $\Omega$ is the entire Riemann sphere, there is a unique solution (called the \emph{normalized} solution) that fixes $0$, $1$ and $\infty$.
\end{theorem}

See \cite[Chapter V, Section B]{ahlfors} or \cite[Chapter V]{lehtovirt} for the proof. Note that Theorem \ref{morrey} not only states the existence of a solution of the Beltrami equation, but also the fact that any solution is a homeomorphism. Moreover we have:

\begin{prop}\label{convident} If $\mu_n \to 0$ in the unit ball of $L^{\infty}$, the normalized quasiconformal homeomorphisms $f^{\mu_n}$ converge to the identity uniformly on compact sets of $\C$. In general if $\mu_n\to\mu$ almost everywhere in $\C$ and $\|\mu_n\|_{\infty} \leq k<1$ for all $n\in\nt$, then the normalized quasiconformal homeomorphisms $f^{\mu_n}$ converge to $f^{\mu}$ uniformly on compact sets of $\C$.
\end{prop}

See \cite[Chapter V, Section C]{ahlfors}.

\subsection{Ahlfors-Bers theorem}

The Beltrami equation induces therefore a one-to-one correspondence between the space of quasiconformal homeomorphisms of $\sph$ that fix $0$, $1$ and $\infty$, and the space of Borel measurable complex-valued functions $\mu$ on $\sph$ for which $\|\mu\|_{\infty}<1$. The following classical result expresses the analytic dependence of the solution of the Beltrami equation with respect to $\mu$:

\begin{theorem}[Ahlfors-Bers 1960]\label{AB} Let $\Lambda$ be an open subset of some complex Banach space and consider a map $\Lambda\times\C\to\D$, denoted by $(\lambda,z)\mapsto\mu_\lambda(z)$, satisfying the following properties:

\begin{enumerate}
\item For every $\lambda$ the function $\C\to\D$ given by $z\mapsto\mu_\lambda(z)$ is measurable, and $\|\mu_\lambda\|_\infty \leq k$ for some fixed $k<1$.
\item For Lebesgue almost every $z\in\C$, the function $\Lambda\to\D$ given by $\lambda\mapsto\mu_\lambda(z)$ is holomorphic.
\end{enumerate}

For each $\lambda$ let $F_\lambda$ be the unique quasiconformal homeomorphism of the Riemann sphere that fixes $0$, $1$ and $\infty$, and whose Beltrami coefficient is $\mu_\lambda$ ($F_\lambda$ is given by Theorem \ref{morrey}). Then $\lambda \mapsto F_\lambda(z)$ is holomorphic for all $z\in\C$.
\end{theorem}

See \cite{ahlforsbers} or \cite[Chapter V, Section C]{ahlfors} for the proof. In Section \ref{Secglueing} we will make repeated use of the following corollary of Ahlfors-Bers theorem:

\begin{prop}\label{aprox2} For any bounded domain $U$ in the complex plane there exists a number $C(U)>0$, with $C(U) \leq C(W)$ if $U \subseteq W$, such that the following holds: let $\big\{G_n:U \to G_n(U)\big\}_{n\in\nt}$ be a sequence of quasiconformal homeomorphisms such that:

\begin{itemize}
\item The domains $G_n(U)$ are \emph{uniformly bounded}: there exists $R>0$ such that $G_n(U) \subset B(0,R)$ for all $n\in\nt$.
\item $\mu_n \to 0$ in the unit ball of $L^{\infty}$, where $\mu_n$ is the Beltrami coefficient of $G_n$ in $U$.
\end{itemize}

Then given any domain $V$ such that $\overline{V} \subset U$ there exist $n_0\in\nt$ and a sequence $\big\{H_n:V \to H_n(V)\big\}_{n \geq n_0}$ of biholomorphisms such that:$$\|H_n-G_n\|_{C^0(V)} \leq C(U)\left(\frac{R}{d\big(\partial V,\partial U\big)}\right)\|\mu_n\|_{\infty}\quad\mbox{for all $n \geq n_0$,}$$where $d\big(\partial V,\partial U\big)$ denote the Euclidean distance between the boundaries of $U$ and $V$ (which are disjoint compact sets in the complex plane, since $V$ is compactly contained in the bounded domain $U$).
\end{prop}

We postpone the proof of Proposition \ref{aprox2} until Appendix \ref{apA}. In the next section we will also use the following extension of the classical Koebe's one-quarter theorem \cite[Theorem 1.3]{cargam}:

\begin{prop}\label{koebe} Given $\varepsilon>0$ there exists $K>1$ for which the following holds: let $f:\D \to f(\D)\subset\C$ be a $K$-quasiconformal homeomorphism such that $f(0)=0$, $f\big((-1,1)\big)\subset\R$ and $f(\D) \subset B(0,1/\varepsilon)$. Suppose that $f|_{(-1/2,1/2)}$ is differentiable and that $\big|f'(t)\big|>\varepsilon$ for all $t\in(-1/2,1/2)$, where $f'$ denotes the real one-dimensional derivative of the restriction of $f$ to the interval $(-1/2,1/2)$. Then:$$B(0,\varepsilon/16) \subset f(\D).$$
\end{prop}

\begin{proof} Suppose, by contradiction, that there exist $\varepsilon>0$ and a sequence $\big\{f_n:\D\to f_n(\D)\subset\C\big\}_{n\in\nt}$ of quasiconformal homeomorphisms with the following properties:

\begin{enumerate}
\item Each $f_n$ is $K_n$-q.c. with $K_n \to 1$ as $n$ goes to infinity.
\item $f_n(0)=0$ and $f_n\big((-1,1)\big)\subset\R$ for all $n\in\nt$.
\item $f_n(\D) \subset B(0,1/\varepsilon)$ for all $n\in\nt$.
\item $f_n|_{(-1/2,1/2)}$ is differentiable and $\big|f_n'(t)\big|>\varepsilon$ for all $t\in(-1/2,1/2)$ and for all $n\in\nt$.
\item $B(0,\varepsilon/16)$ is not contained in $f_n(\D)$ for any $n\in\nt$.
\end{enumerate}

By compactness, since $K_n \to 1$ and $f_n(0)=0$ for all $n\in\nt$, we can assume by taking a subsequence that there exists $f:\D\to\C$ holomorphic such that $f_n \to f$ uniformly on compact sets of $\D$ as $n$ goes to infinity (see for instance \cite[Chapter II, Section 5]{lehtovirt}). Of course $f(0)=0$ and $f\big((-1,1)\big)\subset\R$. We claim that $\big|Df(0)\big|>\varepsilon/2$, where $Df$ denotes the complex derivative of the holomorphic map $f$. Indeed, note that Item $(3)$ implies that:$$\left(\frac{-\varepsilon}{m},\frac{\varepsilon}{m}\right) \subset f_n\left(\left[\frac{-1}{m},\frac{1}{m}\right]\right)\quad\mbox{for all $n,m\in\nt$,}$$and then by the uniform convergence we have:$$\left(\frac{-\varepsilon}{m},\frac{\varepsilon}{m}\right) \subset f\left(\left[\frac{-1}{m},\frac{1}{m}\right]\right)\quad\mbox{for all $m\in\nt$.}$$
Since $f$ is holomorphic this implies the claim. From the claim we see that $f$ is univalent in $\D$, since the uniform limit of quasiconformal homeomorphisms is either constant or a quasiconformal homeomorphism (again see \cite[Chapter II, Section 5]{lehtovirt}). Finally, by Koebe's one-quarter theorem we have $B(0,\varepsilon/8) \subset f(\D)$, but this contradicts that $B(0,\varepsilon/16)$ is not contained in $f_n(\D)$ for any $n\in\nt$.
\end{proof}

\section{Complex extensions of $\mathcal{R}^n(f)$}\label{Seccombounds}

For every $C^3$ critical circle map, with any irrational rotation number, we will construct in this section a suitable extension to an annulus around the unit circle in the complex plane, with the property that, after a finite number of renormalizations, this extension have good geometric bounds and exponentially small Beltrami coefficient. In the next section we will perturb this extension in order to get a holomorphic map with the same combinatorics and also good bounds.

Recall that given a bounded interval $I$ in the real line we denote its Euclidean length by $|I|$, and for any $\alpha>0$ we denote by $N_{\alpha}(I)$ the $\R$-symmetric topological disk:$$N_{\alpha}(I)=\big\{z\in\C:d(z,I)<\alpha|I|\big\},$$where $d$ denotes the Euclidean distance in the complex plane. The goal of this section is the following:

\begin{theorem}\label{nuevoE} There exist three universal constants $\lambda\in(0,1)$, $\alpha>0$ and $\beta>0$ with the following property: let $f$ be a $C^3$ critical circle map with any irrational rotation number. For all $n \geq 1$ denote by $\big(\eta_n,\xi_n\big)$ the components of the critical commuting pair $\mathcal{R}^n(f)$. Then there exist two constants $n_0\in\nt$ and $C>0$ such that for each $n \geq n_0$ both $\xi_n$ and $\eta_n$ extend (after normalized) to $\R$-symmetric orientation-preserving $C^3$ maps defined in $N_{\alpha}\big([-1,0]\big)$ and $N_{\alpha}\big([0,\xi_n(0)]\big)$ respectively, where we have the following seven properties:

\begin{enumerate}
\item Both $\xi_n$ and $\eta_n$ have a unique critical point at the origin, which is of cubic type.
\item The extensions $\eta_n$ and $\xi_n$ commute in $B(0,\lambda)$, that is, both compositions $\eta_n\circ\xi_n$ and $\xi_n\circ\eta_n$ are well defined in $B(0,\lambda)$, and they coincide.
\item $$N_{\beta}\big(\xi_n([-1,0])\big)\subset\xi_n\big(N_{\alpha}\big([-1,0]\big)\big).$$
\item $$N_{\beta}\big([-1,(\eta_n\circ\xi_n)(0)]\big)\subset\eta_n\big(N_{\alpha}\big([0,\xi_n(0)]\big)\big).$$
\item $$\eta_n\big(N_{\alpha}\big([0,\xi_n(0)]\big)\big)\cup\xi_n\big(N_{\alpha}\big([-1,0]\big)\big) \subset B(0,\lambda^{-1}).$$
\item $$\max_{z \in N_{\alpha}([-1,0])\setminus\{0\}}\left\{\frac{\big|\overline{\partial}\xi_n(z)\big|}{\big|\partial\xi_n(z)\big|}\right\} \leq C\lambda^n.$$
\item $$\max_{z \in N_{\alpha}([0,\xi_n(0)])\setminus\{0\}}\left\{\frac{\big|\overline{\partial}\eta_n(z)\big|}{\big|\partial\eta_n(z)\big|}\right\} \leq C\lambda^n.$$
\end{enumerate}
\end{theorem}

In this section we prove Theorem \ref{nuevoE} (see Subsection \ref{lapruebadelE}), and in Section \ref{Secglueing} we prove Theorem \ref{compacto}.

\subsection{Extended lifts of critical circle maps}\label{SubSext} In this subsection we lift a critical circle map to the real line, and then we extend this lift in a suitable way to a neighbourhood of the real line in the complex plane (see Definition \ref{canext} below).

Let $f$ and $g$ be two $C^3$ critical circle maps with cubic critical points $c_f$ and $c_g$, and critical values $v_f$ and $v_g$ respectively. Recall that $\Diff_{+}^{3}(S^1)$ denotes the group (under composition) of orientation-preserving $C^3$ diffeomorphisms of the unit circle, endowed with the $C^3$ topology. Let $\mathcal{A}$ and $\mathcal{B}$ in $\Diff_+^3(S^1)$ defined by:$$\mathcal{A}=\left\{\psi\in\Diff_+^3(S^1):\psi(c_f)=c_g\right\}\quad\mbox{and}\quad\mathcal{B}=\left\{\phi\in\Diff_+^3(S^1):\phi(v_g)=v_f\right\}.$$

There is a canonical homeomorphism between $\mathcal{A}$ and $\mathcal{B}$:$$\psi \mapsto R_{\theta_1} \circ \psi \circ R_{\theta_2}\,,$$where $R_{\theta_1}$ is the rigid rotation that takes $c_g$ to $v_f$, and $R_{\theta_2}$ is the rigid rotation that takes $v_g$ to $c_f$. We will be interested, however, in another identification between $\mathcal{A}$ and $\mathcal{B}$:

\begin{lema}\label{cartasunidim} There exists a homeomorphism $T:\mathcal{A}\to\mathcal{B}$ such that for any $\psi\in\mathcal{A}$ we have:$$f=T(\psi)\circ g \circ\psi\,.$$
\end{lema}

$$\begin{diagram}
\node{S^1}\arrow{e,t}{f}\arrow{s,l}{\psi}\node{S^1}\\
\node{S^1}\arrow{e,t}{g}
\node{S^1}\arrow{n,r}{T(\psi)}
\end{diagram}$$

The lemma is true precisely because the maps $f$ and $g$ have the same degree at their respective critical points:

\begin{proof} Let $\psi$ in $\Diff_+^3(S^1)$ such that $\psi(c_f)=c_g$, and consider the orientation-preserving circle homeomorphism:$$T(\psi)=f \circ \psi^{-1} \circ g^{-1},$$that maps the critical value of $g$ to the critical value of $f$. To see that $T(\psi)$ is in $\Diff_+^3(S^1)$ note that when $z \neq v_g$ we have that $T(\psi)$ is smooth at $z$, with non-vanishing derivative equal to:$$\big(DT(\psi)\big)(z)=D\psi^{-1}\big(g^{-1}(z)\big)\left(\frac{Df\big(\big(\psi^{-1} \circ g^{-1}\big)(z)\big)}{Dg\big(g^{-1}(z)\big)}\right).$$

In the limit we have:$$\lim_{z \to v_g}\left[D\psi^{-1}\big(g^{-1}(z)\big)\left(\frac{Df\big(\big(\psi^{-1} \circ g^{-1}\big)(z)\big)}{Dg\big(g^{-1}(z)\big)}\right)\right]=D\psi^{-1}(c_g)\left(\frac{\big(D^3f\big)\big(c_f\big)}{\big(D^3g\big)(c_g)}\right),$$a well-defined number in $(0,+\infty)$. This proves that $T(\psi)$ is in $\mathcal{B}$ for every $\psi\in\mathcal{A}$. Moreover $T$ is invertible with inverse $T^{-1}:\mathcal{B}\to\mathcal{A}$ given by $T^{-1}(\phi)=g^{-1} \circ \phi^{-1} \circ f$.
\end{proof}

Let $A: S^1\to S^1$ be the map corresponding to the parameters $a=0$ and $b=1$ in the Arnold family (\ref{arfam}), defined in the introduction of this article, and recall that the lift of $A$ to the real line, by the covering $\pi:\R \to S^1:\pi(t)=\exp(2\pi it)$, fixing the origin is given by:$$\widetilde{A}(t)=t-\left(\frac{1}{2\pi}\right)\sin(2\pi t).$$

The critical point of $A$ in the unit circle is at $1$, and it is of cubic type (the critical point is also a fixed point for $A$). Now let $f$ be a $C^3$ critical circle map with a unique cubic critical point at $1$, and let $\widetilde{f}$ be the unique lift of $f$ to the real line under the covering $\pi$ satisfying $\widetilde{f}'(0)=0$ and $0<\widetilde{f}(0)<1$. By Lemma \ref{cartasunidim} we can consider two $C^3$ orientation preserving circle diffeomorphisms $h_1$ and $h_2$, with $h_1(1)=1$ and $h_2(1)=f(1)$, such that the composition $h_2 \circ A \circ h_1$ agrees with the map $f$, that is, the following diagram commutes:

$$\begin{diagram}
\node{S^1}\arrow{e,t}{f}\arrow{s,l}{h_1}\node{S^1}\\
\node{S^1}\arrow{e,t}{A}
\node{S^1}\arrow{n,r}{h_2}
\end{diagram}$$

For each $i\in\{1,2\}$ let $\widetilde{h}_i$ be the lift of $h_i$ to the real line under the covering $\pi$ determined by $\widetilde{h}_i(0)\in[0,1)$. In Proposition \ref{ext} below we will extend both $\widetilde{h_1}$ and $\widetilde{h_2}$ to complex neighbourhoods of the real line in a suitable way. For that purposes we recall the definition of asymptotically holomorphic maps:

\begin{definition}\label{ashol} Let $I$ be a compact interval in the real line, let $U$ be a neighbourhood of $I$ in $\R^2$ and let $H:U\to\C$ be a $C^1$ map (not necessarily a diffeomorphism). We say that $H$ is \emph{asymptotically holomorphic} of order $r \geq 1$ in $I$ if for every $x \in I$:$$\overline{\partial}H(x,0)=0\quad\mbox{and}\quad\frac{\overline{\partial}H(x,y)}{\big(d((x,y),I)\big)^{r-1}} \to 0$$uniformly as $(x,y) \in U \setminus I$ converge to $I$. We say that $H$ is $\R$-\emph{asymptotically holomorphic} of order $r$ if it is asymptotically holomorphic of order $r$ in compact sets of $\R$.
\end{definition}

The sum or product of $\R$-asymptotically holomorphic maps is also $\R$-asymptotically holomorphic. The inverse of an asymptotically holomorphic diffeomorphism of order $r$ is asymptotically holomorphic map of order $r$. Composition of asymptotically holomorphic maps is asymptotically holomorphic.

In the following proposition we suppose $r \geq 1$ even though we will apply it for $r \geq 3$. In the proof we follow the exposition of Graczyk, Sands and \'Swi\c{a}tek in \cite[Lemma 2.1, page 623]{grasandsswia}.

\begin{prop}\label{ext} For $i=1,2$ there exists $H_i:\C\to\C$ of class $C^r$ such that:
\begin{enumerate}
\item $H_i$ is an extension of $\widetilde{h_i}$: $H_i|_{\R}=\widetilde{h_i}$;
\item $H_i$ commutes with unitary horizontal translation: $H_i \circ T = T \circ H_i$;
\item $H_i$ is asymptotically holomorphic in $\R$ of order $r$;
\item $H_i$ is $\R$-symmetric: $H_i(\bar{z})=\overline{H_i(z)}$.
\end{enumerate}

Moreover there exist $R>0$ and four domains $B_R$, $U_R$, $V_R$ and $W_R$ in $\C$, symmetric about the real line, and such that:

\begin{itemize}
\item $B_R=\big\{z\in\C:-R<\Im(z)<R\big\}$;
\item $H_1$ is an orientation preserving diffeomorphism between $B_R$ and $U_R$;
\item $\widetilde{A}(U_R)=V_R$;
\item $H_2$ is an orientation preserving diffeomorphism between $V_R$ and $W_R$.
\item Both $\inf_{z \in B_R}\big|\partial H_1(z)\big|$ and $\inf_{z \in V_R}\big|\partial H_2(z)\big|$ are positive numbers.
\end{itemize}
\end{prop}

\begin{proof} For $z=x+iy\in\C$, with $y \neq 0$, let $P_{x,y}$ be the degree $r$ polynomial map that coincide with $\widetilde{h_i}$ in the $r+1$ real numbers:$$\left\{x+\left(\frac{j}{r}\right)y\right\}_{j\in\{0,1,...,r\}}$$

Recall that $P_{x,y}$ can be given by the following linear combination (the so-called \emph{Lagrange's form of the interpolation polynomial}):
\begin{align}
P_{x,y}(z)&=\sum_{j=0}^{j=r}\widetilde{h_i}\big(x+(j/r)y\big)\prod_{\substack{l=0\\l \neq j}}^{l=r}\frac{z-\big(x+(l/r)y\big)}{\big(x+(j/r)y\big)-\big(x+(l/r)y\big)}\notag\\
&=\sum_{j=0}^{j=r}\widetilde{h_i}\big(x+(j/r)y\big)\prod_{\substack{l=0\\l \neq j}}^{l=r}\frac{z-x-(l/r)y}{\big((j-l)/r\big)y}\notag
\end{align}

We define $H_i(x+iy)=P_{x,y}(x+iy)$, that is:$$H_i(x+iy)=P_{x,y}(x+iy)=\sum_{j=0}^{j=r}\widetilde{h_i}\big(x+(j/r)y\big)\prod_{\substack{l=0\\l \neq j}}^{l=r}\frac{ir-l}{j-l}$$

After computation we obtain:$$H_i(x+iy)=P_{x,y}(x+iy)=\frac{1}{N}\sum_{j=0}^{j=r}\left(\frac{(-1)^j\begin{pmatrix} r \\ j \end{pmatrix}}{1+i(j/r)}\right)\widetilde{h_i}\big(x+(j/r)y\big)$$where:$$N=\sum_{j=0}^{j=r}\left(\frac{(-1)^j\begin{pmatrix} r \\ j \end{pmatrix}}{1+i(j/r)}\right) \neq 0$$

Note that $H_i$ is as smooth as $\widetilde{h_i}$, and $H_i(x)=\widetilde{h_i}(x)$ for any real number $x$ (item $(1)$). Since $\widetilde{h_i}$ is a lift we have for any $j\in\{0,1,...,r\}$ that $\widetilde{h_i}\big(x+1+(j/r)y\big)=\widetilde{h_i}\big(x+(j/r)y\big)+1$, but then $P_{x+1,y}\big(x+1+(j/r)y\big)=P_{x,y}\big(x+(j/r)y\big)+1$ for any $j\in\{0,1,...,r\}$ and this implies $P_{x+1,y} \circ T=T \circ P_{x,y}$ in the whole complex plane. This proves item $(2)$.

To prove that $H_i$ is asymptotically holomorphic of order $r$ in $\R$ note that:$$\overline{\partial}H_i(x+iy)=\frac{1}{2N}\sum_{j=0}^{j=r}(-1)^j\begin{pmatrix} r \\ j \end{pmatrix}\widetilde{h'_i}\big(x+(j/r)y\big)$$and for any $k\in\{0,...,r\}$:$$\frac{\partial^k}{\partial y^k}\,\overline{\partial}H_i(x+iy)=\left(\frac{1}{2N}\right)\left(\frac{1}{r^k}\right)\sum_{j=0}^{j=r}(-1)^jj^k\begin{pmatrix} r \\ j \end{pmatrix}\widetilde{h_i}^{(k+1)}\big(x+(j/r)y\big)$$

Now we claim that for any $k\in\{0,...,r-1\}$ we have $\sum_{j=0}^{j=r}(-1)^jj^k\begin{pmatrix} r \\ j \end{pmatrix}=0$. Indeed, for any $j\in\{0,...,r\}$ we have $\frac{\partial^j}{\partial t^j}(1-t)^r=(-1)^j\left(\frac{r!}{(r-j)!}\right)(1-t)^{r-j}$, and this gives us the equality $(1-t)^r=\sum_{j=0}^{j=r}(-1)^j\begin{pmatrix} r \\ j \end{pmatrix}t^j$ for $r \geq 1$. Putting $t=1$ we obtain the claim for $k=0$. Since $t\frac{\partial}{\partial t}(1-t)^r=\sum_{j=0}^{j=r}(-1)^jjt^j\begin{pmatrix} r \\ j \end{pmatrix}$, we obtain the claim for $k=1$ if we put $t=1$. Putting $t=1$ in $t\frac{\partial}{\partial t}\big[t\frac{\partial}{\partial t}(1-t)^r\big]$ we obtain the claim for $k=2$, and so forth until $k=r-1$.

With the claim we obtain for any $x\in\R$ that:$$\overline{\partial}H_i(x)=\left(\frac{1}{2N}\right)\widetilde{h'_i}\big(x\big)\sum_{j=0}^{j=r}(-1)^j\begin{pmatrix} r \\ j \end{pmatrix}=0$$and for any $k\in\{0,...,r-1\}$:$$\frac{\partial^k}{\partial y^k}\,\overline{\partial}H_i(x)=\left(\frac{1}{2N}\right)\left(\frac{\widetilde{h_i}^{(k+1)}(x)}{r^k}\right)\sum_{j=0}^{j=r}(-1)^jj^k\begin{pmatrix} r \\ j \end{pmatrix}=0$$

By Taylor theorem:$$\lim_{y \to 0}\frac{\overline{\partial}H_i(x+iy)}{y^{r-1}}=0$$uniformly on compact sets of the real line, and from this follows that $H_i$ is asymptotically holomorphic of order $r$ in $\R$ (item $(3)$). To obtain the symmetry as in item $(4)$ we can take $z\mapsto\frac{H_i(z)+\overline{H_i(\bar{z})}}{2}$, since this preserves all the other properties.

Finally note that the Jacobian of $H_i$ at a point $x$ in $\R$ is equal to $|\widetilde{h'_i}(x)|^2 \neq 0$. This gives us a complex neighbourhood of the real line where $H_i$ is an orientation preserving diffeomorphism, and the positive constant $R$. Since we also have $\big|\partial H_i\big|=\big|\widetilde{h'_i}\big|$ at the real line, and each $\widetilde{h_i}$ is the lift of a circle diffeomorphism, we obtain the last item of Proposition \ref{ext}.
\end{proof}

These are the extensions that we will consider:

\begin{definition}\label{canext} The map $F:B_R\to W_R$ defined by $F=H_2 \circ \widetilde{A} \circ H_1$ is called the \emph{extended lift} of the critical circle map $f$.
\end{definition}

$$\begin{diagram}
\node{B_R}\arrow{e,t}{F}\arrow{s,l}{H_1}\node{W_R}\\
\node{U_R}\arrow{e,t}{\widetilde{A}}
\node{V_R}\arrow{n,r}{H_2}
\end{diagram}$$

We have the following properties:

\begin{itemize}
\item $F$ is $C^r$ in the horizontal band $B_R$;
\item $T \circ F = F \circ T$ in $B_R$;
\item $F$ is $\R$-symmetric (in particular $F$ preserves the real line), and $F$ restricted to the real line is $\widetilde{f}$;
\item $F$ is asymptotically holomorphic in $\R$ of order $r$;
\item The critical points of $F$ in $B_R$ are the integers (the same as $\widetilde{A}$), and they are of cubic type.
\end{itemize}

We remark that the extended lift of a real-analytic critical circle map will be $C^{\infty}$ in the corresponding horizontal strip, but not necessarily holomorphic.

The pre-image of the real axis under $F$ consists of $\R$ itself together with two families of $C^r$ curves $\{\gamma_1(k)\}_{k\in\Z}$ and $\{\gamma_2(k)\}_{k\in\Z}$ arising as solutions of $\Im(F(x+iy))=0$. Note that $\gamma_1(k)$ and $\gamma_2(k)$ meet at the critical point $c_k=k$.


Let $\gamma^+_i(k)=\gamma_i(k)\cap\h$ and $\gamma^-_i(k)=\gamma_i(k)\cap\h^-$ for $i=1,2$. We also denote $\gamma^+_i(0)$ just by $\gamma^+_i$.

\begin{lema}\label{theta0} We can choose $R$ small enough to have that $\gamma^+_1$ is contained in $T=\big\{\arg(z)\in\big(\frac{\pi}{6},\frac{\pi}{2}\big)\big\} \cap B_R$ (that is, the open triangle with vertices $0$, $iR$ and $(\sqrt{3}+i)R$), $\gamma^+_2$ is contained in $-\overline{T}$, $\gamma^-_1$ is contained in $-T$ and $\gamma^-_2$ is contained in $\overline{T}$.
\end{lema}

\begin{proof} The derivative of $H_1$ at real points is conformal, so the angle between $\gamma_1$ and $\gamma_2$ with the real line at zero is $\frac{\pi}{3}$.
\end{proof}

\subsection{Poincar\'e disks} Besides the notion of asymptotically holomorphic maps, the main tool in order to prove Theorem \ref{nuevoE} is the notion of Poincar\'e disk, introduced into the subject by Sullivan in his seminal article \cite{sullivan}.

Given an open interval $I=(a,b)\subset\R$ let $\C_I=\big(\C\setminus\R\big) \cup I=\C\setminus\big(\R \setminus I\big)$. For $\theta\in(0,\pi)$ let $D$ be the open disk in the plane intersecting the real line along $I$ and for which the angle from $\R$ to $\partial D$ at the point $b$ (measured anticlockwise) is $\theta$. Let $D^+=D\cap\{z:\Im(z)>0\}$ and let $D^-$ be the image of $D^+$ under complex conjugation.

Define the \emph{Poincar\'e disk} of angle $\theta$ based on $I$ as $D_{\theta}(a,b)=D^+ \cup I \cup D^-$, that is, $D_{\theta}(a,b)$ is the set of points in the complex plane that \emph{view} $I$ under an angle $\geq\theta$ (see Figure 5). Note that for $\theta=\frac{\pi}{2}$ the Poincar\'e disk $D_{\theta}(a,b)$ is the Euclidean disk whose diameter is the interval $(a,b)$.

We denote by $\diam\big(D_{\theta}(a,b)\big)$ the Euclidean diameter of $D_{\theta}(a,b)$. For $\theta\in\big[\frac{\pi}{2},\pi\big)$ the diameter of $D_{\theta}(a,b)$ is always $|b-a|$. When $\theta\in\big(0,\frac{\pi}{2}\big)$ we have that:$$\frac{\diam\big(D_{\theta}(a,b)\big)}{|b-a|}$$is an orientation-reversing diffeomorphism between $\big(0,\frac{\pi}{2}\big)$ and $\big(1,+\infty\big)$, which is real-analytic. Indeed, when $\theta\in\big(0,\frac{\pi}{2}\big)$ the center of $D^+$ is $\left(\frac{a+b}{2}\right)+i\left(\frac{b-a}{2\tan\theta}\right)$, and its radius is $\frac{b-a}{2\sin\theta}$, thus we obtain:
\begin{align}
\diam\big(D_{\theta}(a,b)\big)&=2\left(\frac{b-a}{2\tan\theta}\right)+2\left(\frac{b-a}{2\sin\theta}\right)\notag\\
&=\left(\frac{1}{\tan\theta}+\frac{1}{\sin\theta}\right)(b-a)\notag\\
&=\left(\frac{1+\cos\theta}{\sin\theta}\right)(b-a).\notag
\end{align}

Therefore we have:$$\frac{\diam\big(D_{\theta}(a,b)\big)}{|b-a|}=\frac{1+\cos\theta}{\sin\theta}\quad\mbox{for any}\quad\theta\in\left(0,\frac{\pi}{2}\right).$$

In particular when $\theta$ goes to zero the ratio $\diam\big(D_{\theta}(a,b)\big)/|b-a|$ goes to infinity like $2/\theta$.

\begin{figure}[ht!]
\begin{center}
\includegraphics[scale=0.6]{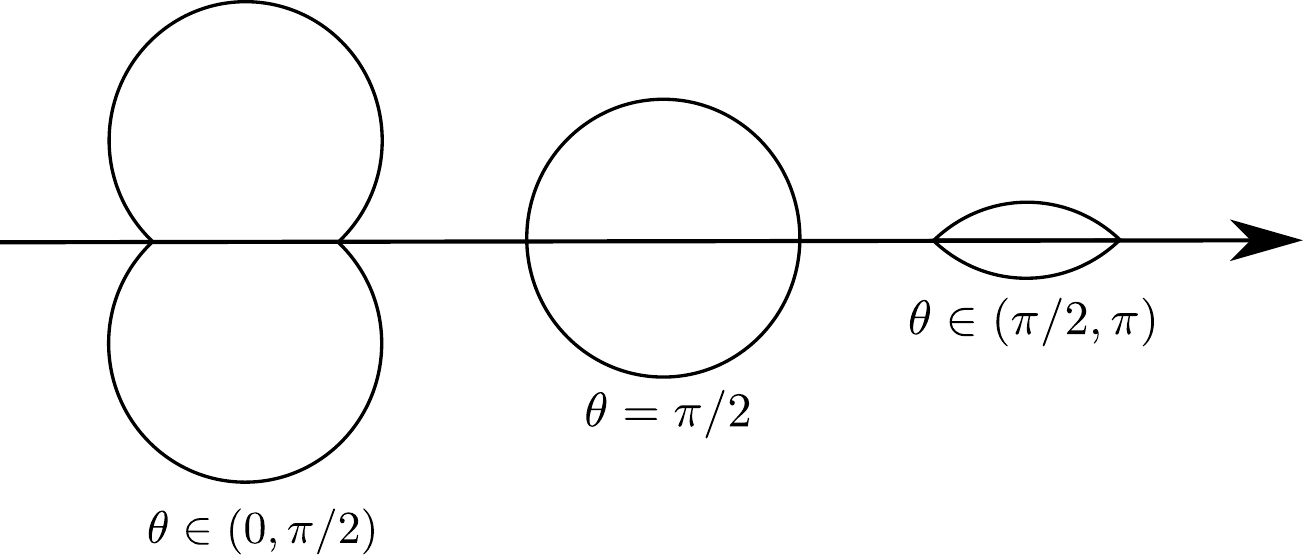}
\caption{Poincar\'e disks.}
\end{center}
\end{figure}

Poincar\'e disks have a geometrical meaning: $\C_I$ is an open, connected and simply connected set which is not the whole plane. By the Riemann mapping theorem we can endow $\C_I$ with a complete and conformal Riemannian metric of constant curvature equal to $-1$, just by pulling back the Poincar\'e metric of $\D$ by any conformal uniformization. Note that $I$ is always a hyperbolic geodesic by symmetry.

For a given $\theta\in(0,\pi)$ consider $\varepsilon(\theta)=\log\tan\big(\frac{\pi}{2}-\frac{\theta}{4}\big)$, which is an orientation-reversing real-analytic diffeomorphism between $(0,\pi)$ and $(0,+\infty)$. An elementary computation shows that the set of points in $\C_I$ whose hyperbolic distance to $I$ is less than $\varepsilon$ is precisely $D_{\theta}(a,b)$.

In particular we can state Schwarz lemma in the following way: let $I$ and $J$ be two intervals in the real line and let $\phi:\C_I\to\C_J$ be a holomorphic map such that $\phi(I) \subset J$. Then for any $\theta\in(0,\pi)$ we have that $\phi\big(D_{\theta}(I)\big) \subset D_{\theta}(J)$.

With this at hand (and a very clever inductive argument, see also \cite{lyuyam}), Yampolsky was able to obtain \emph{complex bounds} for critical circle maps in the \emph{Epstein} class \cite[Theorem 1.1]{yampolsky1}. The reason why we chose asymptotically holomorphic maps to extend our (finitely smooth) one-dimensional dynamics (see Proposition \ref{ext} and Definition \ref{canext} above) is the following asymptotic Schwarz lemma, obtained by Graczyk, Sands and \'Swi\c{a}tek in \cite[Proposition 2, page 629]{grasandsswia} for asymptotically holomorphic maps:

\begin{prop}[Almost Schwarz inclusion]\label{schwarz} Let $h:I\to\R$ be a $C^3$ diffeomorphism from a compact interval $I$ with non-empty interior into the real line. Let $H$ be any $C^3$ extension of $h$ to a complex neighbourhood of $I$, which is asymptotically holomorphic of order $3$ on $I$. Then there exist $M>0$ and $\delta>0$ such that if $a,c \in I$ are different, $\theta\in(0,\pi)$ and $\diam\big(D_{\theta}(a,c)\big)<\delta$ then:$$H\big(D_{\theta}(a,c)\big) \subseteq D_{\tilde{\theta}}\big(h(a),h(c)\big)$$where $\tilde{\theta}=\theta-M|c-a|\diam\big(D_{\theta}(a,c)\big)$. Moreover, $\tilde{\theta}>0$.
\end{prop}

Let us point out that a predecessor of this almost Schwarz inclusion, for real-analytic maps, already appeared in the work of de Faria and de Melo \cite[Lemma 3.3, page 350]{edsonwelington2}.

\subsection{Proof of Theorem \ref{nuevoE}}\label{lapruebadelE} With Proposition \ref{schwarz} at hand, we are ready to start the proof of Theorem \ref{nuevoE}. We will work with $\widetilde{f}^{q_{n+1}}|_{I_n}$, the proof for $\widetilde{f}^{q_n}|_{I_{n+1}}$ being the same.

\begin{prop}\label{3.4} Let $f$ be a $C^3$ critical circle map with irrational rotation number, and let $F$ be its extended lift (according to Definition \ref{canext}). There exists $n_0\in\nt$ such that for any $n \geq n_0$ there exist two numbers $K_n \geq 1$ and $\theta_n>0$ satisfying $K_n \to 1$ and $\theta_n \to 0$ as $n \to +\infty$, and:$$\lim_{n\to+\infty}\left|\frac{\diam\big(D_{\theta_n/K_n}\big(\widetilde{f}(I_n)\big)\big)}{\big|\widetilde{f}(I_n)\big|}-\frac{\diam\big(D_{\theta_n}\big(\widetilde{f}^{q_{n+1}}(I_n)\big)\big)}{\big|\widetilde{f}^{q_{n+1}}(I_n)\big|}\right|=0$$with the following property: let $\theta\geq\theta_n$, $1 \leq j \leq q_{n+1}$ and let $J$ be an open interval such that:$$I_n \subseteq\overline{J}\subseteq \big(\widetilde{f}^{q_{n-1}-q_{n+1}}(0),\widetilde{f}^{q_{n}-q_{n+1}}(0)\big).$$
Then the inverse branch $F^{-j+1}$ mapping $\widetilde{f}^j(J)$ back to $\widetilde{f}(J)$ is well defined over $D_{\theta}\big(\widetilde{f}^j(J)\big)$, and maps this neighbourhood diffeomorphically onto an open set contained in $D_{\theta/K_n}\big(\widetilde{f}(J)\big)$.
\end{prop}

To simplify notation we will prove Proposition \ref{3.4} for the case $J=I_n$ and $j=q_{n+1}$.

\begin{proof} For each $n\in\nt$ and $j\in\{1,...,q_{n+1}-1\}$ we know by combinatorics that $\widetilde{f}$ is a $C^3$ diffeomorphism from $\widetilde{f}^{j}(I_n)$ to $\widetilde{f}^{j+1}(I_n)$. Let $M_{j,n}>0$ and $\delta_{j,n}>0$ given by Proposition \ref{schwarz} applied to the corresponding inverse branch of the extended lift $F$. Moreover, let $M_n=\max_{j\in\{1,...,q_{n+1}-1\}}\{M_{j,n}\}$ and $\delta_n=\min_{j\in\{1,...,q_{n+1}-1\}}\{\delta_{j,n}\}$. For each $n\in\nt$ let $A_n$ and $B_n$ be the affine maps given by:$$A_n(t)=\big(1/\big|\widetilde{f}^{q_{n+1}}(I_n)\big|\big)\big(t-\widetilde{f}^{q_{n+1}}(0)\big)\mbox{ and }B_n(t)=\big(1/\big|\widetilde{f}(I_n)\big|\big)\big(t-\widetilde{f}(0)\big).$$

By the real bounds, the $C^3$ diffeomorphism $T_n:[0,1] \to [0,1]$ given by:$$T_n=B_n \circ \widetilde{f}^{-q_{n+1}+1} \circ A_n^{-1}$$has universally bounded distortion, and therefore:$$\inf_{\substack{t\in[0,1]\\n \geq n_0}}\big\{\big|T_n'(t)\big|\big\}>0.$$

In particular $M=\sup_{n \geq n_0}\{M_n\}$ is finite, and $\delta=\inf_{n \geq n_0}\{\delta_n\}$ is positive. Let $d_n=\max_{1 \leq j \leq q_{n+1}}\big|\widetilde{f}^j(I_n)\big|$, and recall that by the real bounds the sequence $\{d_n\}_{n \geq 1}$ goes to zero exponentially fast when $n$ goes to infinity. In particular we can choose a sequence $\big\{\alpha_n\big\}_{n \geq 1}\subset\big(0,\frac{\pi}{2}\big)$ also convergent to zero but such that:$$\lim_{n\to+\infty}\left(\frac{d_n}{(\alpha_n)^3}\right)=0\,.$$
Let $\psi:(0,\pi)\to[1,+\infty)$ defined by:$$\psi(\theta)=\max\left\{1,\frac{1+\cos\theta}{\sin\theta}\right\}=\left\{\begin{array}{ll}
\frac{1+\cos\theta}{\sin\theta} & \mbox{for }\theta\in\big(0,\frac{\pi}{2}\big)\\
1 & \mbox{for }\theta\in\big[\frac{\pi}{2},\pi\big)\\
\end{array} \right.$$

Note that $\psi$ is an orientation-reversing real-analytic diffeomorphism between $\big(0,\frac{\pi}{2}\big)$ and $\big(1,+\infty\big)$. As we said before, for any $\theta\in(0,\pi)$ and any real numbers $a<b$, we have that $\diam\big(D_{\theta}(a,b)\big)=\psi(\theta)|b-a|$. Now define:$$\theta_{n}=\alpha_n+\psi(\alpha_n)(\delta M)\sum_{j=0}^{q_{n+1}-1}\big|\widetilde{f}^{j+1}(I_n)\big|^2>\alpha_n>0$$and:$$K_n=\frac{\theta_n}{\alpha_n}=1+\left(\frac{\psi(\alpha_n)}{\alpha_n}\right)(\delta M)\sum_{j=0}^{q_{n+1}-1}\big|\widetilde{f}^{j+1}(I_n)\big|^2>1\,.$$

By the choice of $\alpha_n$ we have:$$\lim_{n\to+\infty}\left(\frac{\psi(\alpha_n)}{\alpha_n}\right)d_n=0\,,$$and since:$$\sum_{j=0}^{q_{n+1}-1}\big|\widetilde{f}^{j+1}(I_n)\big|^2 \leq d_n$$we have that $\theta_n \to 0$ and $K_n \to 1$ when $n$ goes to infinity. We also have:
\begin{align}
\big|\psi\big(\theta_n/K_n\big)-\psi\big(\theta_n\big)\big|&\leq\left(\max_{\theta\in\big[\theta_n/K_n,\theta_n\big]}\big|\psi'(\theta)\big|\right)\big|\theta_n-\theta_n/K_n\big|\notag\\
&=\big|\psi'(\theta_n/K_n)\big|\big|\theta_n-\theta_n/K_n\big|\notag\\
&=\left(\frac{\psi(\theta_n/K_n)}{\sin(\theta_n/K_n)}\right)\big|\theta_n-\theta_n/K_n\big|\notag\\
&=\left(\frac{\psi(\alpha_n)}{\sin(\alpha_n)}\right)\big|\theta_n-\alpha_n\big|\notag\\
&=(\delta M)\left(\frac{\big(\psi(\alpha_n)\big)^2}{\sin(\alpha_n)}\right)\sum_{j=0}^{q_{n+1}-1}\big|\widetilde{f}^{j+1}(I_n)\big|^2\notag\\
&\leq(\delta M)\left(\frac{\big(\psi(\alpha_n)\big)^2}{\sin(\alpha_n)}\right)d_n\,,\notag
\end{align}
and this goes to zero by the choice of $\alpha_n$. In particular:$$\lim_{n\to+\infty}\left|\frac{\diam\big(D_{\theta_n/K_n}\big(\widetilde{f}(I_n)\big)\big)}{\big|\widetilde{f}(I_n)\big|}-\frac{\diam\big(D_{\theta_n}\big(\widetilde{f}^{q_{n+1}}(I_n)\big)\big)}{\big|\widetilde{f}^{q_{n+1}}(I_n)\big|}\right|=0$$as stated. We choose $n_0\in\nt$ such that for all $n \geq n_0$ we have $\psi(\alpha_n)d_n<\delta$.

Define inductively $\{\theta_j\}_{j=1}^{j=q_{n+1}}$ by $\theta_{q_{n+1}}=\theta_n$ and for $1 \leq j \leq q_{n+1}-1$ by:$$\theta_j=\theta_{j+1}-M\big|\widetilde{f}^{j+1}(I_n)\big|\diam\big(D_{\theta_{j+1}}\big(\widetilde{f}^{j+1}(I_n)\big)\big)=\theta_{j+1}-M\psi(\theta_{j+1})\big|\widetilde{f}^{j+1}(I_n)\big|^2$$

We want to show that $\theta_j>\alpha_n=\frac{\theta_n}{K_n}$ for all $1 \leq j \leq q_{n+1}$. For this we claim that for any $1 \leq j \leq q_{n+1}$ we have that:$$\theta_j\geq\alpha_n+\psi(\alpha_n)(\delta M)\sum_{k=0}^{j-1}\big|\widetilde{f}^{k+1}(I_n)\big|^2>\alpha_n\,.$$

The claim follows by (reverse) induction in $j$ (the case $j=q_{n+1}$ holds by definition). If the claim is true for $j+1$ we have $\psi(\theta_{j+1})<\psi(\alpha_m)$, this implies $\theta_j>\theta_{j+1}-\psi(\alpha_n)(\delta M)\big|\widetilde{f}^{j+1}(I_n)\big|^2$ and with this the claim is true for $j$. It follows that:$$\diam\big(D_{\theta_{j}}\big(\widetilde{f}^{j}(I_n)\big)\big)=\psi(\theta_j)\big|\widetilde{f}^j(I_n)\big|<\psi(\alpha_n)d_n<\delta\leq\delta_j\quad\mbox{for all}\quad 1 \leq j \leq q_{n+1}\,.$$

By Proposition \ref{koebe} the inverse branch $F^{-1}$ mapping $\widetilde{f}^{j+1}(I_n)$ back to $\widetilde{f}^j(I_n)$ is a well-defined diffeomorphism from the Poincar\'e disk $D_{\theta_{j+1}}\big(\widetilde{f}^{j+1}(I_n)\big)$ onto its image, and by Proposition \ref{schwarz} we know that:$$F^{-1}\big(D_{\theta_{j+1}}\big(\widetilde{f}^{j+1}(I_n)\big)\big)\subseteq\big(D_{\theta_{j}}\big(\widetilde{f}^{j}(I_n)\big)\big).$$

The claim also gives us:$$\theta_1\geq\alpha_n+\psi(\alpha_n)(\delta M)\big|\widetilde{f}(I_n)\big|^2>\alpha_n=\frac{\theta_n}{K_n}\,,$$and this finish the proof.
\end{proof}

\begin{coro}\label{laest} There exist constants $\alpha>0$, $C_1,C_2>0$ and $\lambda\in(0,1)$ with the following property: let $f$ be a $C^3$ critical circle map with irrational rotation number, and let $F$ be its extended lift. There exists $n_0 \in \nt$ such that for each $n \geq n_0$ there exists an $\R$-symmetric topological disk $Y_n$ with:$$N_{\alpha}\big(\widetilde{f}(I_n)\big) \subset Y_n\,,$$such that the composition $F^{q_{n+1}-1}:Y_n \to F^{q_{n+1}-1}(Y_n)$ is a well defined $C^3$ diffeomorphism and we have:

\begin{enumerate}
\item\label{1del69} $$C_1<\frac{\diam\big(F^{q_{n+1}-1}(Y_n)\big)}{\big|\widetilde{f}^{q_{n+1}}(I_n)\big|}<C_2\,,\quad\mbox{and}$$
\item\label{2del69} $$\sup_{z \in Y_n}\left\{\frac{\big|\overline{\partial}F^{q_{n+1}-1}(z)\big|}{\big|\partial F^{q_{n+1}-1}(z)\big|}\right\} \leq C_2\lambda^n\,.$$
\end{enumerate}
\end{coro}

\begin{proof} For each $n\in\nt$ let:

\begin{itemize}
\item $I_n$ be the closed interval whose endpoints are $0$ and $\big(T^{-p_n}\circ\widetilde{f}^{q_n}\big)(0)$,
\item $J_n$ be the open interval containing the origin that projects to $\big(f^{q_{n+1}}(1),f^{q_n-q_{n+3}}(1)\big)$ under the covering $\pi(t)=e^{2\pi it}$, and
\item $K_n$ be the open interval containing the origin that projects to $\big(f^{q_{n-1}-q_{n+1}}(1),f^{q_n-q_{n+1}}(1)\big)$ under the covering $\pi$.
\end{itemize}

Note that $I_n \cup I_{n+1}\subset J_n\subset\overline{J_n}\subset K_n$ (see Figure 6). By combinatorics, the map $\widetilde{f}:\widetilde{f}^j(K_n)\to\widetilde{f}^{j+1}(K_n)$ is a diffeomorphism for all $j\in\{1,...,q_{n+1}-1\}$, and therefore all restrictions $\widetilde{f}:\widetilde{f}^j(J_n)\to\widetilde{f}^{j+1}(J_n)$ are diffeomorphisms for any $j\in\{1,...,q_{n+1}-1\}$ (just as in the proof of Proposition \ref{3.4}).

\begin{figure}[ht!]
\begin{center}
\includegraphics[scale=0.6]{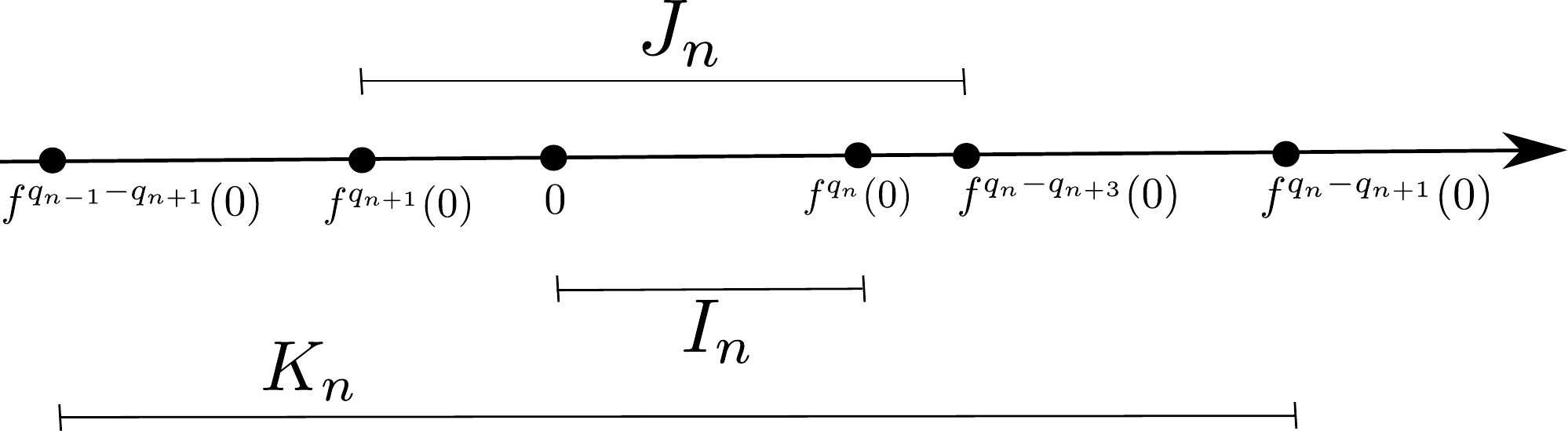}
\caption{Relative positions of the relevant points in the proof of Corollary \ref{laest}.}
\end{center}
\end{figure}

Recall that the extended lift $F:B_R\to W_R$ is given by the composition $F=H_2 \circ \widetilde{A} \circ H_1$ (see Definition \ref{canext}). Let $n_0\in\nt$ given by Proposition \ref{3.4}, and for each $n \geq n_0$ let $K_n \geq 1$ and $\theta_n>0$ also given by Proposition \ref{3.4}. Fix $\theta\in(0,\pi)$ such that $\theta>\theta_n$ for all $n \geq n_0$ and such that:$$\big|\mu_{H_i}(z)\big|<\left(\frac{1}{2}\right)\left(d\big(z,\widetilde{f}^j(J_n)\big)\right)^2$$for any $z \in D_{\theta/K_n}\big(\widetilde{f}^j(J_n)\big)$, any $j\in\{1,...,q_{n+1}-1\}$ and any $i\in\{1,2\}$ (as before $\mu_{H_i}$ denotes the Beltrami coefficient of the quasiconformal homeomorphism $H_i$, and $d$ denotes the Euclidean distance in the complex plane). The existence of such $\theta$ is guaranteed by Proposition \ref{3.4}, the fact that both $H_i$ are asymptotically holomorphic in $\R$ of order $3$, and the last item in Proposition \ref{ext}.

Let $Y_n \subset F^{-q_{n+1}+1}\big(D_{\theta}\big(\widetilde{f}^{q_{n+1}}(J_n)\big)\big)$ be the preimage of $D_{\theta}\big(\widetilde{f}^{q_{n+1}}(J_n)\big)$ under $F^{q_{n+1}-1}$ given by Proposition \ref{3.4}, and note that:

\begin{itemize}
\item $Y_n$ is an $\R$-symmetric topological disk,
\item $\overline{\widetilde{f}(I_n)} \subset Y_n$,
\item $\widetilde{f}(I_{n+1}) \subset Y_n$.
\item By Proposition \ref{3.4}, $F^j(Y_n) \subset D_{\theta/K_n}\big(\widetilde{f}^{j+1}(J_n)\big)$ for all $j\in\{0,1,...,q_{n+1}-1\}$.
\end{itemize}

Moreover:$$\diam\big(F^{q_{n+1}-1}(Y_n)\big)=\diam\big(D_{\theta}\big(\widetilde{f}^{q_{n+1}}(J_n)\big)\big)=\psi(\theta)\big|\widetilde{f}^{q_{n+1}}(J_n)\big|\,,$$and by the real bounds $\big|\widetilde{f}^{q_{n+1}}(J_n)\big|$ and $\big|\widetilde{f}^{q_{n+1}}(I_n)\big|$ are comparable (with universal constants independent of $n \geq n_0$). Again the map $\psi$ is the same as in the proof of Proposition \ref{3.4}. This gives us Item (\ref{1del69}), and now we prove Item (\ref{2del69}). For each $n \geq n_0$ let $k_n \in [0,1)$ be the conformal distortion of $F^{q_{n+1}-1}$ at $Y_n$, that is:$$k_n=\sup_{z \in Y_n}\left\{\frac{\big|\overline{\partial}F^{q_{n+1}-1}(z)\big|}{\big|\partial F^{q_{n+1}-1}(z)\big|}\right\}\,.$$
Moreover, for each $j\in\{1,...,q_{n+1}-1\}$ let $K_{n,j}$, $K_{n,j}(1)$ and $K_{n,j}(2)$ in $[1,+\infty)$ be the quasiconformality of $F$ at $F^{j-1}(Y_n)$, of $H_1$ also at $F^{j-1}(Y_n)$, and of $H_2$ at $(\widetilde{A} \circ H_1)\big(F^{j-1}(Y_n)\big)$ respectively. Since $\widetilde{A}$ is conformal we have that:
\begin{align}
k_n&\leq\log\left(\prod_{j=1}^{q_{n+1}-1}K_{n,j}\right)=\sum_{j=1}^{q_{n+1}-1}\log K_{n,j}\notag\\
&=\sum_{j=1}^{q_{n+1}-1}\big(\log K_{n,j}(1)+\log K_{n,j}(2)\big)\notag\\
&\leq\sum_{j=1}^{q_{n+1}-1}M_0\,\big(\diam\big(F^{j-1}(Y_n)\big)\big)^2\quad\quad\mbox{(for some $M_0>1$)}\notag\\
&\leq\sum_{j=1}^{q_{n+1}-1}M_0\,\big(\diam\big(D_{\theta/K_n}(\widetilde{f}^j(J_n))\big)\big)^2\notag\\
&=\sum_{j=1}^{q_{n+1}-1}M_0\,\big(\psi(\theta/K_n)\big)^2\big|\widetilde{f}^j(J_n)\big|^2<M_1\left(\sum_{j=1}^{q_{n+1}-1}\big|\widetilde{f}^j(J_n)\big|^2\right)\,.\notag
\end{align}

The last inequality follows from the fact that $K_n \to 1$ when $n$ goes to $\infty$. By combinatorics the projection of the family $\big\{\widetilde{f}^j(J_n)\big\}_{j=1}^{q_{n+1}-1}$ to the unit circle has finite multiplicity of intersection (independent of $n \geq n_0$), and therefore:
\begin{equation}\label{fecho}
\sum_{j=1}^{q_{n+1}-1}\left|\widetilde{f}^j(J_n)\right|^2<M_2\left(\max_{j\in\{1,...,q_{n+1}-1\}}\left|\widetilde{f}^j(J_n)\right|\right)\,,
\end{equation}
where the constant $M_2>0$ only depends on the multiplicity of intersection of the projection of the family $\big\{\widetilde{f}^j(J_n)\big\}_{j=1}^{q_{n+1}-1}$ to the unit circle. By the real bounds, the right hand of \eqref{fecho} goes to zero exponentially fast at a universal rate (independent of $f$), and therefore we obtain constants $\lambda\in(0,1)$ and $C>0$ such that:$$k_n=\sup_{z \in Y_n}\left\{\frac{\big|\overline{\partial}F^{q_{n+1}-1}(z)\big|}{\big|\partial F^{q_{n+1}-1}(z)\big|}\right\} \leq C\lambda^n\quad\mbox{for all $n \geq n_0$.}$$

To finish the proof of Corollary \ref{laest} we need to obtain definite domains around $\widetilde{f}(I_n)$ contained in $Y_n$. As in the proof of Proposition \ref{3.4}, for each $n \geq n_0$ let $A_n$ and $B_n$ be the affine maps given by:$$A_n(z)=\big(1/\big|\widetilde{f}^{q_{n+1}}(I_n)\big|\big)\big(z-\widetilde{f}^{q_{n+1}}(0)\big)\mbox{ and }B_n(z)=\big(1/\big|\widetilde{f}(I_n)\big|\big)\big(z-\widetilde{f}(0)\big),$$and also let $Z_n=A_n\big(D_{\theta}\big(\widetilde{f}^{q_{n+1}}(J_n)\big)\big)$. By the real bounds there exists a universal constant $\alpha_0>0$ such that:$$N_{\alpha_0}\big([0,1]\big) \subset Z_n\quad\mbox{for all $n \geq n_0$.}$$

The $\R$-symmetric orientation preserving $C^3$ diffeomorphism $T_n:Z_n \to T_n(Z_n)$ given by:$$T_n=B_n \circ F^{-q_{n+1}+1} \circ A_n^{-1}$$ induces a diffeomorphism in $[0,1]$ which, again by the real bounds, has universally bounded distortion. In particular there exists $\varepsilon>0$ such that $\big|T_n'(t)\big|>\varepsilon$ for all $t\in[0,1]$ and for all $n \geq n_0$. By Proposition \ref{koebe} there exists $\alpha>0$ (only depending on $\alpha_0$ and $\varepsilon$) such that (by taking $n_0$ big enough):$$N_{\alpha}\big([0,1]\big)\subset T_n(Z_n)\quad\mbox{for all $n \geq n_0$,}$$and therefore:$$N_{\alpha}\big(\widetilde{f}(I_n)\big)\subset Y_n\quad\mbox{for all $n \geq n_0$.}$$
\end{proof}

\begin{prop}\label{estdist} There exist constants $\alpha>0$, $C_1,C_2>0$ and $\lambda\in(0,1)$ with the following property: let $f$ be a $C^3$ critical circle map with irrational rotation number, and let $F$ be its extended lift. There exists $n_0 \in \nt$ such that for each $n \geq n_0$ there exists an $\R$-symmetric topological disk $X_n$ with:$$N_{\alpha}(I_n) \subset X_n\,,\quad\mbox{where $I_n=\big[0,(T^{-p_n}\circ\widetilde{f}^{q_n})(0)\big]$,}$$such that the composition $F^{q_{n+1}}$ is well defined in $X_n$, it has a unique critical point at the origin, and we have:

\begin{enumerate}
\item\label{1del610} $$C_1<\frac{\diam\big(F^{q_{n+1}}(X_n)\big)}{\big|\widetilde{f}^{q_{n+1}}(I_n)\big|}<C_2\,,\quad\mbox{and}$$
\item\label{2del610} $$\sup_{z \in X_n\setminus\{0\}}\left\{\frac{\big|\overline{\partial}F^{q_{n+1}}(z)\big|}{\big|\partial F^{q_{n+1}}(z)\big|}\right\} \leq C_2\lambda^n\,.$$
\end{enumerate}
\end{prop}

\begin{proof} From the construction of the extended lift $F$ in Subsection \ref{SubSext} (see also Lemma \ref{theta0}) there exists a complex neighbourhood $\Omega$ of the origin such that the restriction $F:\Omega \to F(\Omega)$ is of the form $Q\circ\psi$, where $Q(z)=z^3+\widetilde{f}(0)$, and $\psi:\Omega \to Q^{-1}\big(F(\Omega)\big)$ is an $\R$-symmetric orientation preserving $C^3$ diffeomorphism fixing the origin. In particular there exist $\varepsilon>0$ and $\delta>0$ such that if $t\in(-\delta,\delta)$ then $\big|(\psi^{-1})'(t)\big|>\varepsilon$, where $(\psi^{-1})'$ denotes the one-dimensional derivative of the restriction of $\psi^{-1}$ to $Q^{-1}\big(F(\Omega)\big)\cap\R$. Let $K>1$ given by Proposition \ref{koebe} applied to $\varepsilon>0$. Since $\psi$ is asymptotically holomorphic of order $3$ in $\Omega$, we can choose $\Omega$ small enough in order to have that $\psi$ is $K$-quasiconformal. By taking $n_0\in\nt$ big enough we can assume that $\big|\psi(I_n)\big|<\delta$ and $Y_n \subset F(\Omega)$ for all $n \geq n_0$, where the topological disk $Y_n$ is the one given by Corollary \ref{laest}. By Corollary \ref{laest} and elementary properties of the cube root map (see for instance \cite[Lemma 2.2]{yampolsky1}) there exists a universal constant $\alpha_0>0$ such that for all $n \geq n_0$ we have that:

\begin{equation}\label{cont}
N_{\alpha_0}\big(\psi(I_n)\big)\subset Q^{-1}(Y_n)\,.
\end{equation}
	
Define $X_n\subset\Omega$ as the preimage of $Y_n$ under $F$, that is, $X_n=F^{-1}(Y_n)=\psi^{-1}\big(Q^{-1}(Y_n)\big)$. Item (\ref{1del610}) follows directly from Item (\ref{1del69}) in Corollary \ref{laest} since $F^{q_{n+1}}(X_n)=F^{q_{n+1}-1}(Y_n)$. By \eqref{cont} and Proposition \ref{koebe} there exists a universal constant $\alpha>0$ such that:$$N_{\alpha}(I_n)\subset X_n\subset\Omega\quad\mbox{for all $n \geq n_0$.}$$

To obtain Item (\ref{2del610}) recall that by Item (\ref{2del69}) in Corollary \ref{laest} we have:$$\sup_{z \in Y_n}\left\{\frac{\big|\overline{\partial}F^{q_{n+1}-1}(z)\big|}{\big|\partial F^{q_{n+1}-1}(z)\big|}\right\} \leq C\lambda^n\,.$$

Since $Q$ is a polynomial, it is conformal at its regular points, and since $\big\|\mu_{\psi}\big\|_{\infty} \leq \frac{K-1}{K+1}<1$ in $\Omega$ we have:$$\sup_{z \in X_n\setminus\{0\}}\left\{\frac{\big|\overline{\partial}F^{q_{n+1}}(z)\big|}{\big|\partial F^{q_{n+1}}(z)\big|}\right\} \leq C\lambda^n\,.$$
\end{proof}

Theorem \ref{nuevoE} follows directly from Proposition \ref{estdist} and its analogue statement for $\widetilde{f}^{q_n}|_{I_{n+1}}$.

\section{Proof of Theorem \ref{compacto}}\label{Secglueing}

As its tittle indicates, this section is entirely devoted to the proof of Theorem \ref{compacto}, and recall that Theorem \ref{compacto} implies our main result (Theorem \ref{main}) as we saw in Section \ref{Secred}.

First let us fix some notation and terminology. By a \emph{topological disk} we mean an open, connected and simply connected set properly contained in the complex plane. Let $\pi:\C\to\C\setminus\{0\}$ be the holomorphic covering $z\mapsto\exp(2\pi iz)$, and let $T:\C\to\C$ be the horizontal translation $z \mapsto z+1$ (which is a generator of the group of automorphisms of the covering). For any $R>1$ consider the \emph{band}:$$B_R=\big\{z\in\C:-\log R<2\pi\Im(z)<\log R\big\},$$which is the universal cover of the \emph{round annulus}:$$A_R=\left\{z\in\C:\frac{1}{R}<\big|z\big|<R\right\}$$via the holomorphic covering $\pi$. Since $B_R$ is $T$-invariant, the translation generates the group of automorphisms of the covering. The restriction $\pi:\R \to S^1=\partial\D$ is also a covering map, the automorphism $T$ preserves the real line, and again generates the group of automorphisms of the covering.

More generally, an \emph{annulus} is an open and connected set $A$ in the complex plane whose fundamental group is isomorphic to $\Z$. By the Uniformization Theorem such an annulus is conformally equivalent either to the punctured disk $\D\setminus\{0\}$, to the punctured plane $\C\setminus\{0\}$, or to some round annulus $A_R=\left\{z\in\C:1/R<\big|z\big|<R\right\}$. In the last case the value of $R>1$ is unique, and there exists a holomorphic covering from $\D$ to $A$ whose group of deck transformations is infinite cyclic, and such that any generator is a M\"obius transformation that has exactly two fixed points at the boundary of the unit disk.

Since the deck transformations are M\"obius transformations, they are isometries of the Poincar\'e metric on $\D$ and therefore there exists a unique Riemannian metric on $A$ such that the covering map provided by the Uniformization Theorem is a local isometry. This metric is complete, and in particular, any two points can be joined by a minimizing geodesic. There exists a unique simple closed geodesic in $A$, whose hyperbolic length is equal to $\pi^2/\log R$. The length of this closed geodesic is therefore a conformal invariant.

We denote by $\Theta$ the antiholomorphic involution $z \mapsto 1/\bar{z}$ in the punctured plane $\C\setminus\{0\}$, and we say that a map is \emph{$S^1$-symmetric} if it commutes with $\Theta$. An annulus is \emph{$S^1$-symmetric} if it is invariant under $\Theta$ (for instance, the round annulus $A_R$ described above is $S^1$-symmetric). In this case, the unit circle is the \emph{core curve} (the unique simple closed geodesic) for the hyperbolic metric in $A$. In this section we will deal only with $S^1$-symmetric annulus. In particular any time that some annulus $A_0$ is contained in some other annulus $A_1$, we have that $A_0$ separates the boundary components of $A_1$ (more technically, the inclusion is \emph{essential} in the sense that the fundamental group $\pi_1(A_0)$ injects into $\pi_1(A_1)$).

Besides Theorem \ref{nuevoE} (stated and proved in Section \ref{Seccombounds}), the main tool in order to prove Theorem \ref{compacto} is Proposition \ref{aprox2} (stated in Section \ref{SecAB}, and proved in Appendix \ref{apA} as a corollary of Ahlfors-Bers Theorem). The proof of Theorem \ref{compacto} will be divided in three subsections. Along the proof, $C$ will denote a positive constant (independent of $n\in\nt$) and $n_0$ will denote a positive (big enough) natural number. At first, let $n_0\in\nt$ given by Theorem \ref{nuevoE}. Moreover let us use the following notation: $W_1=N_{\alpha}\big([-1,0]\big)$, $W_2=W_2(n)=N_{\alpha}\big([0,\xi_n(0)]\big)$, $W_0=B(0,\lambda)$ and $\mathcal{V}=B(0,\lambda^{-1})$, where $\alpha>0$ and $\lambda\in(0,1)$ are the universal constants given by Theorem \ref{nuevoE}. Recall that $\eta_n(0)=-1$ for all $n \geq 1$ after normalization.

\subsection{A first perturbation and a bidimensional glueing procedure}\label{sub1}

From Theorem \ref{nuevoE} we have:

\begin{lema}\label{UniformU} There exists an $\R$-symmetric topological disk $U$ with:$$-1 \in U \subset W_1 \setminus W_0,$$such that for all $n \geq n_0$ the composition:$$\eta_n^{-1}\circ\xi_n:U\to\big(\eta_n^{-1}\circ\xi_n\big)(U)$$is an $\R$-symmetric orientation-preserving $C^3$ diffeomorphism.
\end{lema}

For each $n \geq n_0$ denote by $A_n$ the diffeomorphism $\eta_n^{-1}\circ\xi_n$. Note that $\|\mu_{A_n}\|_{\infty} \leq C\lambda^n$ in $U$ for all $n \geq n_0$, and that the domains $\big\{A_n(U)\big\}_{n \geq n_0}$ are uniformly bounded since they are contained in $\cup_jW_2^j$. Fix $\varepsilon>0$ and $\delta>0$ such that the rectangle:$$V=\big(-1-\varepsilon,-1+\varepsilon\big)\times\big(-i\delta,i\delta\big)$$is compactly contained in $U$, and apply Proposition \ref{aprox2} to the sequence of $\R$-symmetric orientation-preserving $C^3$ diffeomorphisms:$$\left\{A_n:U \to A_n(U)\right\}_{n \geq n_0}$$to obtain a sequence of $\R$-symmetric biholomorphisms:$$\big\{B_n:V \to B_n(V)\big\}_{n \geq n_0}$$such that:$$\big\|A_n-B_n\big\|_{C^0(V)} \leq C\lambda^n\quad\mbox{for all $n \geq n_0$.}$$

From the commuting condition we obtain:

\begin{lema}\label{domain} For each $n \geq n_0$ there exist three $\R$-symmetric topological disks $V_i(n)$ for $i\in\{1,2,3\}$ with the following five properties:

\begin{itemize}
\item $0 \in V_1(n) \subset W_0$;
\item $\big(\eta_n\circ\xi_n\big)(0)=\big(\xi_n\circ\eta_n\big)(0)=\xi_n(-1) \in V_2(n) \subset W_2$;
\item $\xi_n(0) \in V_3(n) \subset W_2$;
\item When restricted to $V_1(n)$, both $\eta_n$ and $\xi_n$ are orientation-preserving three-fold $C^3$ branched coverings onto $V$ and $V_3(n)$ respectively, with a unique critical point at the origin;
\item Both restrictions $\xi_n|_V$ and $\eta_n|_{V_3(n)}$ are orientation-preserving $C^3$ diffeomorphisms onto $V_2(n)$.
\end{itemize}

In particular the composition $\eta_n^{-1}\circ\xi_n$ is an orientation-preserving $C^3$ diffeomorphism from $V$ onto $V_3(n)$ for all $n \geq n_0$.
\end{lema}

For each $n \geq n_0$ let $U_1(n)$, $U_2(n)$ and $U_3(n)$ be three $\R$-symmetric topological disks such that:
\begin{itemize}
\item $\overline{U_1(n)}$, $\overline{U_2(n)}$ and $\overline{U_3(n)}$ are pairwise disjoint;
\item $V \bigcap U_j(n)=\emptyset$ and $V_i(n) \bigcap U_j(n)=\emptyset$ for $i,j\in\{1,2,3\}$;
\item $\overline{U_1(n)} \subset W_1$ and $\overline{U_2(n)}\bigcup\overline{U_3(n)} \subset W_2$;
\end{itemize}
and such that:$$\mathcal{U}_n=\interior\left[V\bigcup\left(\bigcup_{i=1}^{i=3}V_i(n)\right)\bigcup\left(\bigcup_{j=1}^{j=3}\overline{U_j(n)}\right)\right]$$is an $\R$-symmetric topological disk (see Figure 7). Note that:$$\overline{I_{\xi_n} \cup I_{\eta_n}}\subset\mathcal{U}_n\subset W_1 \cup W_2\quad\mbox{for all $n \geq n_0$,}$$and that $\mathcal{U}_n\setminus\big(\overline{V \cup V_1(n) \cup V_2(n) \cup V_3(n)}\big)$ has three connected components, which are precisely $U_1(n)$, $U_2(n)$ and $U_3(n)$. By Theorem \ref{nuevoE} we can choose $U_1(n)$, $U_2(n)$ and $U_3(n)$ in order to also have:$$\overline{N_{\delta}\big([-1,0]\big) \cup N_{\delta}\big([0,\xi_n(0)]\big)}\subset\mathcal{U}_n\quad\mbox{for all $n \geq n_0$,}$$for some universal constant $\delta>0$, independent of $n \geq n_0$. Note also that each $\mathcal{U}_n$ is uniformly bounded since it is contained in $N_{\alpha}\big([-1,K]\big)$, where $\alpha>0$ is given by Theorem \ref{nuevoE}, and $K>1$ is the universal constant given by the real bounds.

\begin{figure}[ht!]
\begin{center}
\includegraphics[scale=0.5]{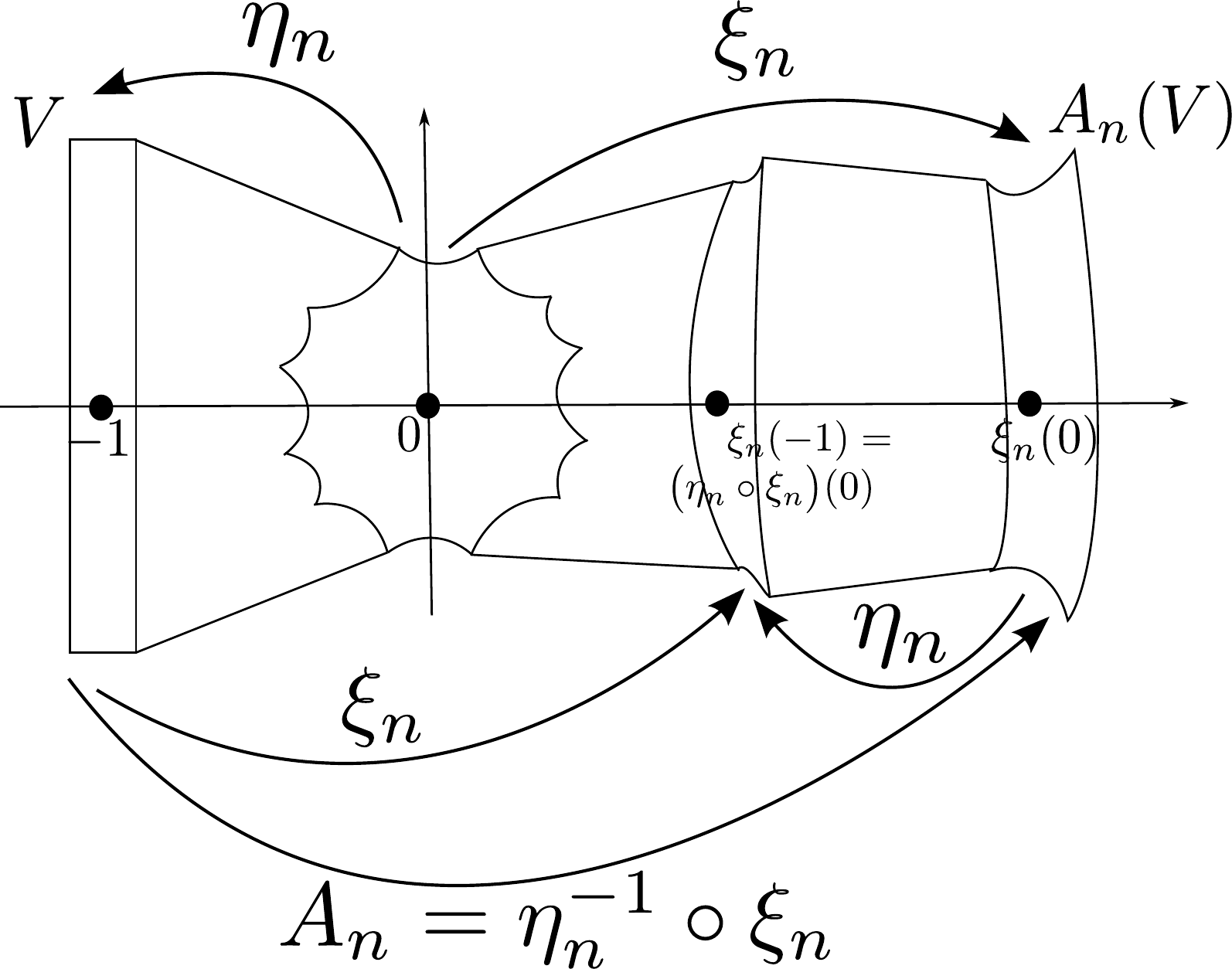}
\caption{The domain $\mathcal{U}_n$.}
\end{center}
\end{figure}

%

For each $n \geq n_0$ let $\mathcal{T}_n$ be an $\R$-symmetric topological disk such that:
\begin{itemize}
\item $V$, $V_1(n)$, $V_2(n)$ and $B_n(V)$ are contained in $\mathcal{T}_n$,
\item $\mathcal{T}_n\setminus\big(V \cup B_n(V)\big)$ is connected and simply connected,
\item The Hausdorff distance between $\overline{\mathcal{T}_n}$ and $\overline{\mathcal{U}_n}$ is less or equal than:$$\big\|A_n-B_n\big\|_{C^0(V)}\leq C\lambda^n,$$
\end{itemize}

\begin{lema}\label{defphin} For each $n \geq n_0$ there exists an orientation-preserving $\R$-symmetric $C^3$ diffeomorphism $\Phi_n:\mathcal{U}_n \to \mathcal{T}_n$ such that:

\begin{itemize}
\item $\Phi_n \equiv Id$ in the interior of $V\cup\overline{U_1(n)}\cup V_1(n)$, in particular $\Phi_n(0)=0$.
\item $B_n=\Phi_n\circ\big(\eta_n^{-1}\circ\xi_n\big)\circ\Phi_n^{-1}$ in $V$, that is, $\Phi_n \circ A_n=B_n \circ \Phi_n$ in $V$.
\item $\big\|\Phi_n-Id\big\|_{C^0(\mathcal{U}_n)} \leq C\lambda^n$.
\item $\|\mu_{\Phi_n}\|_{\infty} \leq C\lambda^n$ in $\mathcal{U}_n$.
\end{itemize}
\end{lema}

\begin{proof}[Proof of Lemma \ref{defphin}] For each $n \geq n_0$ we have $\|A_n-B_n\|_{C^0(V)} \leq C\lambda^n$ and therefore:$$\big\|Id-\big(B_n \circ A_n^{-1}\big)\big\|_{C^0\big(V_3(n)\big)} \leq C\lambda^n.$$If we define $\Phi_n|_{V_3(n)}=B_n \circ A_n^{-1}$ we also have $\|\mu_{\Phi_n}\|_{\infty}=\|\mu_{A_n^{-1}}\|_{\infty}$ in $V_3(n)$, which is equal to $\|\mu_{A_n}\|_{\infty}$ in $V$. In particular $\|\mu_{\Phi_n}\|_{\infty}\leq C\lambda^n$ in $V_3(n)$, and then we define $\Phi_n$ in the whole $\mathcal{U}_n$ by interpolating $B_n \circ A_n^{-1}$ in $V_3(n)$ with the identity in the interior of $V\cup\overline{U_1(n)}\cup V_1(n)$.
\end{proof}

Consider the seven topological disks:$$X_1(n)=\interior\big(V\cup\overline{U_1(n)}\cup V_1(n)\big)\subset W_1\cap\mathcal{U}_n\,,$$
$$X_2(n)=\interior\big(V_1(n)\cup\overline{U_2(n)}\cup V_2(n)\cup\overline{U_3(n)}\cup V_3(n)\big)\subset W_2\cap\mathcal{U}_n\,,$$
$$\widehat{X}_1(n)=\left\{z \in X_1(n):\xi_n(z)\in\mathcal{U}_n\right\},\quad\widehat{X}_2(n)=\left\{z \in X_2(n):\eta_n(z)\in\mathcal{U}_n\right\},$$
$$\widehat{\mathcal{T}}_n=\Phi_n\big(\widehat{X}_1(n)\big)\cup\Phi_n\big(\widehat{X}_2(n)\big)\subset\mathcal{T}_n\,,$$
$$Y_1(n)=X_1(n)\cap\Phi_n\big(\widehat{X}_1(n)\big)\quad\mbox{and}\quad Y_2(n)=X_2(n)\cap\Phi_n\big(\widehat{X}_2(n)\big).$$
Note that $V$, $V_1(n)$ and $B_n(V)$ are contained in $\widehat{\mathcal{T}}_n$ for all $n \geq n_0$. Moreover, we have the following two corollaries of Theorem \ref{nuevoE}:

\begin{lema}\label{vert} There exists $\delta>0$ such that for all $n \geq n_0$ we have:$$N_{\delta}\big([-1,0]\big) \subset Y_1(n) \quad\mbox{and}\quad N_{\delta}\big([0,\xi_n(0)]\big) \subset Y_2(n)\,.$$
\end{lema}

\begin{lema}\label{contdet} Both:$$\sup_{n \geq n_0}\left\{\sup_{z \in Y_1(n)}\left\{\det\big(D\xi_n(z)\big)\right\}\right\}\quad\mbox{and}\quad\sup_{n \geq n_0}\left\{\sup_{z \in Y_2(n)}\left\{\det\big(D\eta_n(z)\big)\right\}\right\}$$are finite, where $\det(\cdot)$ denotes the determinant of a square matrix.
\end{lema}

Let:$$\widehat{\xi}_n:\Phi_n\big(\widehat{X}_1(n)\big)\to\big(\Phi_n\circ\xi_n\big)\big(\widehat{X}_1(n)\big)\mbox{ defined by }\widehat{\xi}_n=\Phi_n\circ\xi_n\circ\Phi_n^{-1},$$and:$$\widehat{\eta}_n:\Phi_n\big(\widehat{X}_2(n)\big)\to\big(\Phi_n\circ\eta_n\big)\big(\widehat{X}_2(n)\big)\mbox{ defined by }\widehat{\eta}_n=\Phi_n\circ\eta_n\circ\Phi_n^{-1}.$$

Since each $\Phi_n$ is an $\R$-symmetric $C^3$ diffeomorphism, the pair $\big(\widehat{\eta}_n,\widehat{\xi}_n\big)$ restrict to a critical commuting pair with the same rotation number as $(\eta_n,\xi_n)$, and the same criticality (that we are assuming to be cubic, in order to simplify). Note also that $\widehat{\eta}_n(0)=-1$ for all $n \geq n_0$. Moreover, from Lemma \ref{contdet} and $\big\|\Phi_n-Id\big\|_{C^0(\mathcal{U}_n)} \leq C\lambda^n$ we have:$$\left\|\xi_n-\widehat{\xi}_n\right\|_{C^0\big(Y_1(n)\big)} \leq C\lambda^n\quad\mbox{and}\quad\left\|\eta_n-\widehat{\eta}_n\right\|_{C^0\big(Y_2(n)\big)} \leq C\lambda^n\quad\mbox{for all $n \geq n_0$.}$$

Therefore is enough to shadow the sequence $\big(\widehat{\eta}_n,\widehat{\xi}_n\big)$ in the domains $Y_1(n)$ and $Y_2(n)$, instead of $(\eta_n,\xi_n)$ (the shadowing sequence will be constructed in Subsection \ref{sub3} below). The main advantage of working with the sequence $\big(\widehat{\eta}_n,\widehat{\xi}_n\big)$ is precisely the fact that $\widehat{\eta}_n^{-1}\circ\widehat{\xi}_n$ is univalent in $V$ for all $n \geq n_0$ (since it coincides with $B_n$). In particular we can choose each topological disk $\mathcal{U}_n$ and $\mathcal{T}_n$ defined above with the additional property that, identifying $V$ with $B_n(V)$ via the biholomorphism $B_n$, we obtain from $\mathcal{T}_n$ an abstract annular Riemann surface $\mathcal{S}_n$ (with the complex structure induced by the quotient).

Let us denote by $p_n:\mathcal{T}_n\to\mathcal{S}_n$ the canonical projection (note that $p_n$ is not a covering map, just a surjective local diffeomorphism). The projection of the real line, $p_n(\R \cap \mathcal{T}_n)$, is real-analytic diffeomorphic to the unit circle $S^1$. We call it the \emph{equator} of $\mathcal{S}_n$.

Since complex conjugation leaves $\mathcal{T}_n$ invariant and commutes with $B_n$, it induces an antiholomorphic involution $F_n:\mathcal{S}_n\to\mathcal{S}_n$ acting as the identity on the equator $p_n(\R \cap \mathcal{T}_n)$. Note that $F_n$ has a continuous extension to $\partial\mathcal{S}_n$ that switches the boundary components.
	
Since $\mathcal{S}_n$ is obviously not biholomorphic to $\D\setminus\{0\}$ neither to $\C\setminus\{0\}$ we have $\modulo(\mathcal{S}_n)<\infty$ for all $n \geq n_0$, where $\modulo(\cdot)$ denotes the conformal modulus of an annular Riemann surface. For each $n \geq n_0$ define a constant $R_n$ in $(1,+\infty)$ by:$$R_n=\exp\left(\frac{\modulo\big(\mathcal{S}_n\big)}{2}\right),$$that is, $\mathcal{S}_n$ is conformally equivalent to $A_{R_n}=\big\{z\in\C:R_n^{-1}<|z|<R_n\big\}$. Any biholomorphism between $\mathcal{S}_n$ and $A_{R_n}$ must send the equator $p_n\big(\R \cap \mathcal{T}_n\big)$ onto the unit circle $S^1$ (because the equator is invariant under the antiholomorphic involution $F_n$, and the unit circle is invariant under the antiholomorphic involution $z \mapsto 1/\bar{z}$ in $A_{R_n}$). Let $\Psi_n:\mathcal{S}_n \to A_{R_n}$ be the conformal uniformization determined by $\Psi_n\big(p_n(0)\big)=1$, and let $P_n:\mathcal{T}_n \to A_{R_n}$ be the holomorphic surjective local diffeomorphism:$$P_n=\Psi_n \circ p_n.$$

See Figure 8. Note that $P_n(0)=1$ and $P_n(\mathcal{T}_n\cap\R)=S^1$ for all $n \geq n_0$. Moreover $P_n(z)\overline{P_n(\overline{z})}=1$ for all $z\in\mathcal{T}_n$ and all $n \geq n_0$. From now on we forget about the abstract cylinder $\mathcal{S}_n$.

\begin{lema}\label{koebedist} There exist two constants $\delta>0$ and $C>1$ such that for all $n \geq n_0$ and for all $z \in N_{\delta}\big([-1,\widetilde{\xi}_n(0)]\big)$ we have $z\in\widehat{\mathcal{T}}_n\subset\mathcal{T}_n$ and:$$\frac{1}{C}<\big|P_n'(z)\big|<C\,.$$
\end{lema}

\begin{proof}[Proof of Lemma \ref{koebedist}] By the real bounds there exists a universal constant $C_0>1$ such that for each $n \geq n_0$ there exists $w_n\in\big[-1,\widetilde{\xi}_n(0)\big]$ such that:$$\frac{1}{C_0}<\big|P_n'(w_n)\big|<C_0\,.$$

To prove Lemma \ref{koebedist} we need to construct a definite complex domain around $\big[-1,\widetilde{\xi}_n(0)\big]$ where $P_n$ has universally bounded distortion. Again by the real bounds there exist $\delta>0$ and $l\in\nt$ with the following properties: for each $n \geq n_0$ there exists $z_1,z_2,...,z_{k_n}\in\big[-1,\widetilde{\xi}_n(0)\big]$ with $k_n<l$ for all $n \geq n_0$ such that:

\begin{itemize}
\item $\big[-1,\widetilde{\xi}_n(0)\big]\subset\cup_{i=1}^{k_n}B(z_i,\delta)$.
\item $B(z_i,2\delta)\subset\widehat{\mathcal{T}}_n\subset\mathcal{T}_n$ for all $i\in\{1,...,k_n\}$.
\item $P_n|_{B(z_i,2\delta)}$ is univalent for all $i\in\{1,...,k_n\}$.
\end{itemize}

By convexity we have for all $n \geq n_0$ and for all $i\in\{1,...,k_n\}$ that:$$\sup_{v,w \in B(z_i,\delta)}\left\{\frac{\big|P_n'(v)\big|}{\big|P_n'(w)\big|}\right\}\leq\exp\left(\sup_{w \in B(z_i,\delta)}\left\{\frac{\big|P_n''(w)\big|}{\big|P_n'(w)\big|}\right\}\right)\,,$$and by Koebe distortion theorem (see for instance \cite[Section I.1, Theorem 1.6]{cargam}) we have:$$\sup_{w \in B(z_i,\delta)}\left\{\frac{\big|P_n''(w)\big|}{\big|P_n'(w)\big|}\right\}\leq\frac{2}{\delta}\quad\mbox{for all $n \geq n_0$ and for all $i\in\{1,...,k_n\}$.}$$
\end{proof}

\begin{figure}[ht!]
\begin{center}
\includegraphics[scale=0.5]{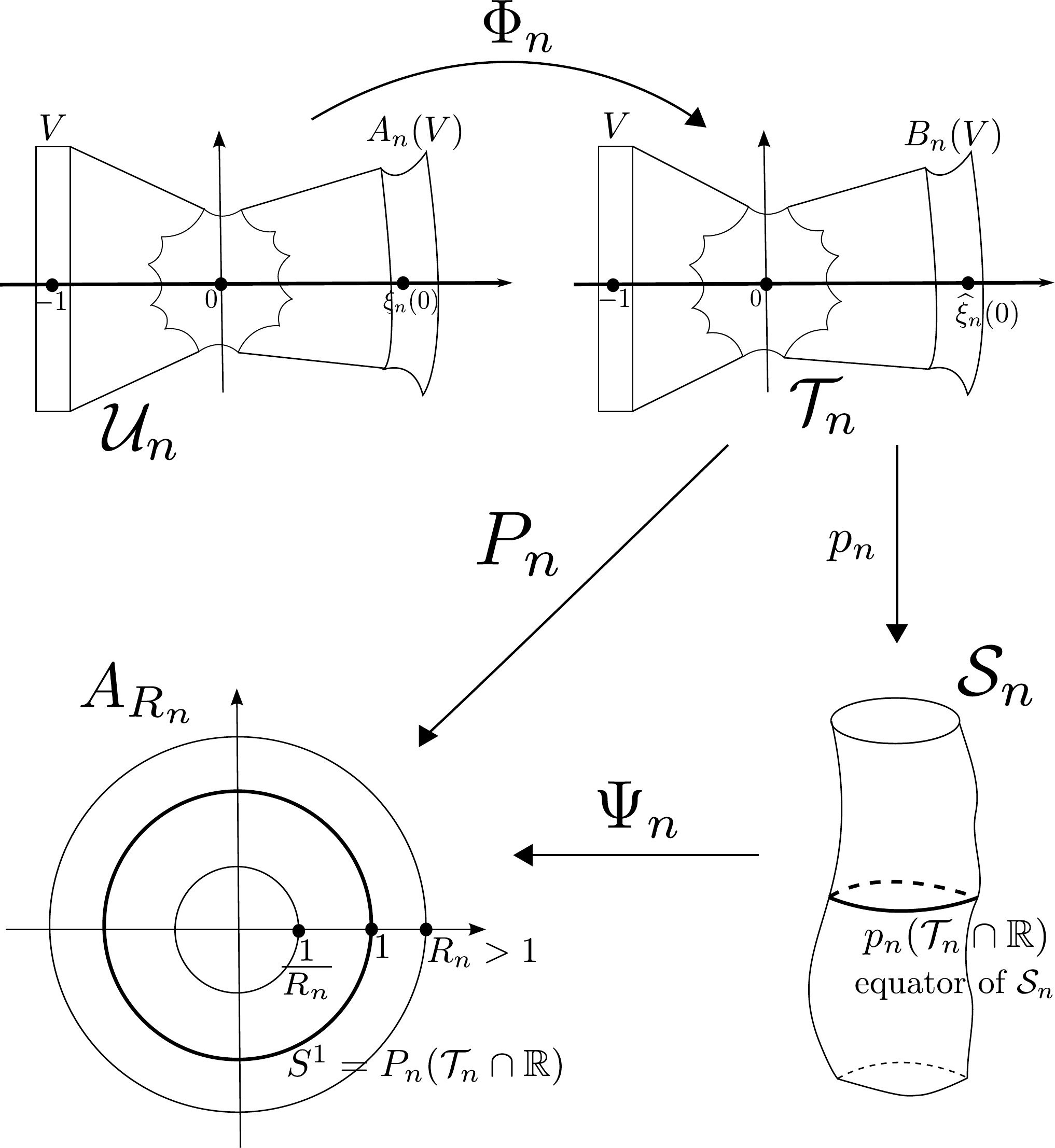}
\caption{Bidimensional Glueing procedure.}
\end{center}
\end{figure}

Now we project each commuting pair $(\widetilde{\eta}_n,\widetilde{\xi}_n)$ from $\widehat{{\mathcal{T}}}_n$ to the round annulus $A_{R_n}$.

\begin{prop}[Glueing procedure]\label{Gn} The pair:$$\widehat{\xi}_n:\Phi_n\big(\widehat{X}_1(n)\big)\to\mathcal{T}_n\quad\mbox{and}\quad\widehat{\eta}_n:\Phi_n\big(\widehat{X}_2(n)\big)\to\mathcal{T}_n$$projects under $P_n$ to a well-defined orientation-preserving $C^3$ map:$$G_n:P_n\big(\widehat{\mathcal{T}}_n\big) \subset A_{R_n}\to A_{R_n}.$$

For each $n \geq n_0$, $P_n(\widehat{\mathcal{T}}_n)$ is a $\Theta$-invariant annulus with positive and finite modulus. Each $G_n$ is $S^1$-symmetric, in particular $G_n$ preserves the unit circle.

When restricted to the unit circle, $G_n$ produce a $C^3$ critical circle map $g_n:S^1 \to S^1$ with cubic critical point at $P_n(0)=1$, and with rotation number $\rho(g_n)=\rho\big(\mathcal{R}^n(f)\big)\in\R\setminus\Q$.
$$\begin{diagram}
\node{\widehat{\mathcal{T}}_n\subset\mathcal{T}_n}\arrow{e,t}{\big(\widehat{\eta}_n,\widehat{\xi}_n\big)}\arrow{s,l}{P_n}
\node{\mathcal{T}_n}\arrow{s,r}{P_n}\\
\node{P_n\big(\widehat{\mathcal{T}}_n\big) \subset A_{R_n}}\arrow{e,t}{G_n}\node{A_{R_n}}
\end{diagram}$$
Moreover the unique critical point of $G_n$ in $P_n\big(\widehat{\mathcal{T}}_n\big)$ is the one in the unit circle (at the point $1$) and:$$\left|\overline{\partial}G_n(z)\right| \leq C\lambda^n\left|\partial G_n(z)\right|\quad\mbox{for all $z \in P_n\big(\widehat{\mathcal{T}}_n\big)\setminus\{1\}$, that is:}$$
$$\|\mu_{G_n}\|_{\infty}\leq C\lambda^n\mbox{ in $P_n\big(\widehat{\mathcal{T}}_n\big)$.}$$
\end{prop}

\begin{proof}[Proof of Proposition \ref{Gn}] This follows from:
\begin{itemize}
\item The construction of $\mathcal{U}_n$ and $\mathcal{T}_n$.
\item The property $B_n=\Phi_n\circ\big(\eta_n^{-1}\circ\xi_n\big)\circ\Phi_n^{-1}$ in $V$.
\item The commuting condition in $V_1(n)$.
\item The symmetry $P_n(z)\overline{P_n(\overline{z})}=1$ for all $z\in\mathcal{T}_n$ and all $n \geq n_0$.
\item The fact that $P_n:\mathcal{T}_n \to A_{R_n}$ is holomorphic, $P_n(0)=1$ and $P_n(\mathcal{T}_n\cap\R)=S^1$ for all $n \geq n_0$.
\end{itemize}
\end{proof}

Note that each $g_n$ belongs to the smooth conjugacy class obtained with the glueing procedure (described in Section \ref{subccp}) applied to the $C^3$ critical commuting pair $\big(\widehat{\eta}_n,\widehat{\xi}_n\big)$. As we said in the introduction, the topological behaviour of each $G_n$ on its annular domain is the same as the restriction of the Blaschke product $f_{\gamma}$ \eqref{formBlaschke} to the annulus $A''' \cup B_1'$, as depicted in Figure 1. In the next subsection we will construct a sequence of real-analytic critical circle maps, with the desired combinatorics, that extend to holomorphic maps exponentially close to $G_n$ in a definite annulus around the unit circle (see Proposition \ref{pertanillo} below).

\subsection{Main perturbation}\label{sub2}

The goal of this subsection is to construct the following sequence of perturbations:

\begin{prop}[Main perturbation]\label{pertanillo} There exist a constant $r>1$ and a sequence of holomorphic maps defined in the annulus $A_r$:$$\left\{H_n:A_r\to\C\right\}_{n \geq n_0}$$such that for all $n \geq n_0$ the following holds:

\begin{itemize}
\item $A_r \subset P_n(\widehat{\mathcal{T}}_n)\subset P_n(\mathcal{T}_n)=A_{R_n}$.
\item $\big\|H_n-G_n\big\|_{C^0(A_r)} \leq C\lambda^n$.
\item $H_n(A_r)\subset\big(G_n \circ P_n\big)\big(\widehat{\mathcal{T}}_n\big)\subset P_n(\mathcal{T}_n)=A_{R_n}$.
\item $H_n$ preserves the unit circle and, when restricted to the unit circle, $H_n$ produces a real-analytic critical circle map $h_n:S^1 \to S^1$ such that:
\begin{itemize}
\item The unique critical point of $h_n$ is at $P_n(0)=1$, and is of cubic type.
\item The critical value of $h_n$ coincide with the one of $g_n$, that is, $h_n(1)=g_n(1) \in P_n(V\cap\R)$.
\item $\rho(h_n)=\rho(g_n)=\rho\big(\mathcal{R}^n(f)\big)\in\R\setminus\Q$.
\end{itemize}
\item The unique critical point of $H_n$ in $A_r$ is the one in the unit circle.
\end{itemize}
\end{prop}

The remainder of this subsection is devoted to proving Proposition \ref{pertanillo}. We wont perturb the maps $G_n$ directly (basically because they are non invertible). Instead, we will decompose them (see Lemma \ref{coef} below), and then we will perturb on their \emph{coefficients} (see the definition after the statement of Lemma \ref{coef}). Those perturbations will be done, again, with the help of Proposition \ref{aprox2} of Section \ref{SecAB}.

Let $A:\C\setminus\{0\}\to\C\setminus\{0\}$ be the map corresponding to the parameters $a=0$ and $b=1$ in the Arnold family (\ref{arfam}), defined in the introduction. The lift of $A$ to the complex plane by the holomorphic covering $z\mapsto\exp(2\pi iz)$ is the entire map $\widetilde{A}:\C\to\C$ given by:$$\widetilde{A}(z)=z-\left(\frac{1}{2\pi}\right)\sin(2\pi z).$$

Then $A$ preserves the unit circle, and its restriction $A: S^1\to S^1$ is a real-analytic critical circle map. The critical point of $A$ in the unit circle is at $1$, and is of cubic type (the critical point is also a fixed point for $A$). The following is a bidimensional version of Lemma \ref{cartasunidim} in Section \ref{Seccombounds}:

\begin{lema}\label{coef} For each $n \geq n_0$ there exist:
\begin{itemize}
\item $S_n>1$,
\item an $S^1$-symmetric orientation-preserving $C^3$ diffeomorphism $\psi_n:P_n\big(\widehat{\mathcal{T}}_n\big) \to A_{S_n}$ and
\item an $S^1$-symmetric biholomorphism $\phi_n:A(A_{S_n}) \to \big(G_n \circ P_n\big)(\widehat{\mathcal{T}}_n)$ such that:
\end{itemize}
$$G_n=\phi_n \circ A \circ \psi_n\quad\mbox{in $P_n\big(\widehat{\mathcal{T}}_n\big)$.}$$
\end{lema}

The diffeomorphisms $\psi_n$ and $\phi_n$ are called the \emph{coefficients} of $G_n$ in $P_n\big(\widehat{\mathcal{T}}_n\big)$.

$$\begin{diagram}
\node{P_n\big(\widehat{\mathcal{T}}_n\big)}\arrow{e,t}{G_n}\arrow{s,l}{\psi_n}\node{\big(G_n \circ P_n\big)\big(\widehat{\mathcal{T}}_n\big)}\\
\node{A_{S_n}}\arrow{e,t}{A}
\node{A(A_{S_n})}\arrow{n,r}{\phi_n}
\end{diagram}$$

\begin{proof}[Proof of Lemma \ref{coef}] For each $n \geq n_0$ let $S_n>1$ such that $A(A_{S_n})$ is a $\Theta$-invariant annulus with:$$\modulo\big(A(A_{S_n})\big)=\modulo\big(\big(G_n \circ P_n\big)(\widehat{\mathcal{T}}_n)\big).$$

In particular there exists a biholomorphism $\phi_n:A(A_{S_n}) \to \big(G_n \circ P_n\big)(\widehat{\mathcal{T}}_n)$ that commutes with $\Theta$. Each $\phi_n$ preserves the unit circle and we can choose it such that $\phi_n(1)=G_n(1)$, that is, $\phi_n$ takes the critical value of $A$ into the critical value of $G_n$.

Since both $G_n$ and $A$ are three-fold branched coverings around their critical points and local diffeomorphisms away from them, the equation $G_n=\phi_n\circ A\circ\psi_n$ induces an orientation-preserving $C^3$ diffeomorphism $\psi_n:P_n\big(\widehat{\mathcal{T}}_n\big) \to A_{S_n}$, that commutes with $\Theta$ and such that $\psi_n(1)=1$, that is, $\psi_n$ takes the critical point of $G_n$ into the one of $A$. The fact that $\psi_n$ is smooth at $1$ with non-vanishing derivative follows from the fact that the critical points of $G_n$ and $A$ have the same degree (see Lemma \ref{cartasunidim} in Section \ref{Seccombounds}).
\end{proof}

Note that, at the beginning of the proof of Lemma \ref{coef}, we have used the fact that the image under the Arnold map $A$ of a small round annulus around the unit circle is also an annulus. This is true, even that $A$ has a critical point in the unit circle (placed at $1$, and being also a fixed point of $A$). Even more is true: the conformal modulus of the annulus $A(A_s)$ depends continuously on $s>1$ (and we also used this fact in the proof). The topological behaviour of the restriction of $A$ to each round annulus $A_{S_n}$ is the same as the restriction of the Blaschke product $f_{\gamma}$ \eqref{formBlaschke} to the annulus $A''' \cup B_1'$, as depicted in Figure 1 in the introduction of this article.

As we said, the idea in order to prove Proposition \ref{pertanillo} is to perturb each diffeomorphism $\psi_n$ with Proposition \ref{aprox2}. In order to control the $C^0$ size of those perturbations we will need some geometric control, that we state in four lemmas, before entering into the proof of Proposition \ref{pertanillo}. From Lemma \ref{koebedist} we have:

\begin{lema}\label{controlRn} $$1<\inf_{n \geq n_0}\{R_n\}\quad\mbox{and}\quad\sup_{n \geq n_0}\{R_n\}<+\infty.$$
\end{lema}

\begin{lema}\label{arranca52} For all $n \geq n_0$ both $P_n(\widehat{\mathcal{T}}_n)$ and $\big(G_n \circ P_n\big)(\widehat{\mathcal{T}}_n)$ are $\Theta$-invariant annulus with finite modulus. Moreover there exists a universal constant $K>1$ such that:$$\frac{1}{K}<\modulo\big(P_n(\widehat{\mathcal{T}}_n)\big)<K\quad\mbox{for all $n \geq n_0$.}$$
\end{lema}

\begin{proof}[Proof of Lemma \ref{arranca52}] By Lemma \ref{controlRn} we know that $R=\sup_{n \geq n_0}\{R_n\}$ is finite, and since for all $n \geq n_0$ both $P_n(\widehat{\mathcal{T}}_n)$ and $\big(G_n \circ P_n\big)(\widehat{\mathcal{T}}_n)$ are contained in the corresponding $A_{R_n}$, we obtain at once that both $P_n(\widehat{\mathcal{T}}_n)$ and $\big(G_n \circ P_n\big)(\widehat{\mathcal{T}}_n)$ have finite modulus, and also that $\sup_{n \geq n_0}\big\{\modulo\big(P_n(\widehat{\mathcal{T}}_n)\big)\big\}$ is finite. Just as in Lemma \ref{controlRn}, the fact that $\inf_{n \geq n_0}\big\{\modulo\big(P_n(\widehat{\mathcal{T}}_n)\big)\big\}$ is positive follows from Lemma \ref{vert} and Lemma \ref{koebedist}.
\end{proof}

\begin{lema}\label{elr0} There exists a constant $r_0>1$ such that $\overline{A_{r_0}} \subset P_n\big(\widehat{\mathcal{T}}_n\big)$ for all $n \geq n_0$.
\end{lema}

\begin{proof}[Proof of Lemma \ref{elr0}] By the invariance with respect to the antiholomorphic involution $z \mapsto 1/\bar{z}$, the unit circle is the core curve (the unique closed geodesic for the hyperbolic metric) of each annulus $P_n\big(\widehat{\mathcal{T}}_n\big)$. Since $\inf_{n \geq n_0}\big\{\modulo\big(P_n(\widehat{\mathcal{T}}_n)\big)\big\}>0$ the statement is well-known, see for instance \cite[Chapter 2, Theorem 2.5]{mclivro1}.
\end{proof}

\begin{lema}\label{controlSn} We have:$$s=\inf_{n \geq n_0}\{S_n\}>1 \quad\mbox{and}\quad S=\sup_{n \geq n_0}\{S_n\}<+\infty.$$
\end{lema}

\begin{proof}[Proof of Lemma \ref{controlSn}] Since $\mu_{\psi_n}=\mu_{G_n}$ in $P_n\big(\widehat{\mathcal{T}}_n\big)$, we have $\|\mu_{\psi_n}\|_{\infty} \leq C\lambda^n$ in $P_n\big(\widehat{\mathcal{T}}_n\big)$ for all $n \geq n_0$. By the geometric definition of quasiconformal homeomorphisms (see for instance \cite[Chapter I, Section 7]{lehtovirt}) we have:$$\left(\frac{1-C\lambda^n}{1+C\lambda^n}\right)\modulo\big(P_n(\widehat{\mathcal{T}}_n)\big) \leq 2\log(S_n) \leq \left(\frac{1+C\lambda^n}{1-C\lambda^n}\right)\modulo\big(P_n(\widehat{\mathcal{T}}_n)\big)$$for all $n \geq n_0$, and we are done by Lemma \ref{arranca52}.
\end{proof}

With this geometric control at hand, we are ready to prove Proposition \ref{pertanillo}:

\begin{proof}[Proof of Proposition \ref{pertanillo}] Let $r_0>1$ given by Lemma \ref{elr0} (recall that $\overline{A_{r_0}} \subset P_n\big(\widehat{\mathcal{T}}_n\big)$ for all $n \geq n_0$), and fix $r\in\big(1,(1+r_0)/2\big)$. How small $r-1$ must be will be determined in the course of the argument (see Lemma \ref{HnGn} below). For any $r\in\big(1,(1+r_0)/2\big)$ consider $\underline{r}=r_0-(r-1)\in\big((1+r_0)/2,r_0\big)$.

The sequence of $S^1$-symmetric $C^3$ diffeomorphisms$$\big\{\psi_n:A_{r_0} \to \psi_n(A_{r_0})\big\}_{n \geq n_0}$$satisfy the hypothesis of Proposition \ref{aprox2} since:
\begin{itemize}
\item $\mu_{\psi_n}=\mu_{G_n}$ in $P_n\big(\widehat{\mathcal{T}}_n\big)$ and therefore $\|\mu_{\psi_n}\|_{\infty} \leq C\lambda^n$ for all $n \geq n_0$, and
\item $\psi_n(A_{r_0})\subset A_{S_n}\subset A_S$ for all $n \geq n_0$ (see Lemma \ref{controlSn} above).
\end{itemize}
Apply Proposition \ref{aprox2} to the bounded domain $A_{\underline{r}}$, compactly contained in $A_{r_0}$, to obtain a sequence of $S^1$-symmetric biholomorphisms$$\big\{\widehat{\psi}_n:A_{\underline{r}}\to\widehat{\psi}_n(A_{\underline{r}})\big\}_{n \geq n_0}$$such that:$$\big\|\widehat{\psi}_n-\psi_n\big\|_{C^0(A_{\underline{r}})} \leq C\lambda^n\quad\mbox{for all $n \geq n_0$.}$$
Fix $n_0$ big enough to have $\widehat{\psi}_n(A_{\underline{r}}) \subset A_{S_n}$, and note that we can suppose that each $\widehat{\psi}_n$ fixes the point $1$ (just as $\psi_n$) by considering:$$z \mapsto \left(\frac{1}{\widehat{\psi}_n(1)}\right)\widehat{\psi}_n(z)\,.$$
Since $\big|\widehat{\psi}_n(z)\big| \leq S$ for all $z\in A_{\underline{r}}$ and for all $n \geq n_0$ (where $S\in(1,+\infty)$ is given by Lemma \ref{controlSn}) and since $\left|\widehat{\psi}_n(1)-1\right| \leq C\lambda^n$ for all $n \geq n_0$, we know that this new map (that we will still denote by $\widehat{\psi}_n$ to simplify) satisfy all the properties that we want for $\widehat{\psi}_n$, and also fixes the point $z=1$.

For each $n \geq n_0$ consider the holomorphic map $H_n:A_{\underline{r}}\to\C$ defined by $H_n=\phi_n \circ A \circ \widehat{\psi}_n$. We have:

\begin{itemize}
\item $H_n(A_{\underline{r}})\subset\big(G_n \circ P_n\big)\big(\widehat{\mathcal{T}}_n\big)\subset A_{R_n}$.
\item $H_n$ is $S^1$-symmetric and therefore it preserves the unit circle.
\item When restricted to the unit circle, $H_n$ produces a real-analytic critical circle map $h_n:S^1 \to S^1$.
\item The unique critical point of $H_n$ in $A_{\underline{r}}$ is the one in the unit circle, which is at $P_n(0)=1$, and is of cubic type.
\item The critical value of $H_n$ coincide with the one of $G_n$, that is, $H_n(1)=G_n(1) \in P_n(V\cap\R)$.
\end{itemize}

We divide in four lemmas the rest of the proof of Proposition \ref{pertanillo}. We need to prove first that, for a suitable $r>1$, $H_n$ is $C^0$ exponentially close to $G_n$ in the annulus $A_r$ (Lemma \ref{HnGn} below), and then that we can choose each $H_n$ with the desired combinatorics for its restriction $h_n$ to the unit circle (Lemma \ref{controlcomb} below). This last perturbation will change the critical value of each $H_n$ (it wont coincide any more with the one of $G_n$). We will finish the proof of Proposition \ref{pertanillo} with Lemma \ref{conjmob}, that allow us to keep the critical point of $H_n$ at the point $P_n(0)=1$, and to place the critical value of $H_n$ at the point $g_n(1)$ for all $n \geq n_0$. This will be important in the following subsection, the last one of this section.

\begin{lema}\label{HnGn} There exists $r\in\big(1,(1+r_0)/2\big)$ such that in the annulus $A_r$ we have:$$\big\|H_n-G_n\big\|_{C^0(A_r)} \leq C\lambda^n\quad\mbox{for all $n \geq n_0$.}$$
\end{lema}

\begin{proof}[Proof of Lemma \ref{HnGn}] The proof is divided in three claims:

\emph{First claim}: There exists $\beta>1$ such that $\overline{A_{\beta}} \subset A(A_{S_n})$ for all $n \geq n_0$.

Indeed, by Lemma \ref{controlSn} the round annulus $A_{(1+s)/2}$ is compactly contained in $A_{S_n}$ for all $n \geq n_0$, and therefore the annulus $A\big(A_{(1+s)/2}\big)$ is contained in $A(A_{S_n})$ for all $n \geq n_0$. Thus we just take $\beta>1$ such that $\overline{A_{\beta}} \subset A\big(A_{(1+s)/2}\big)$ and the first claim is proved.

From now on we fix $\alpha\in(1,\beta)$.

\emph{Second claim}: There exists $r\in\big(1,(1+r_0)/2\big)$ close enough to one in order to simultaneously have $(A\circ\widehat{\psi}_n)(A_r) \subset A_\alpha$ and $(A\circ\psi_n)(A_r) \subset A_\alpha$ for all $n \geq n_0$.

Indeed, since $\overline{A_r}\subset A_{\underline{r}}$, $\widehat{\psi}_n$ is holomorphic, and $\widehat{\psi}_n(A_{\underline{r}}) \subset A_{S_n} \subset A_S$ for all $n \geq n_0$ (where $S\in(1,+\infty)$ is given by Lemma \ref{controlSn}), we have by Cauchy derivative estimate that $\sup_{n \geq n_0}\left\{\big|\widehat{\psi}_n'(z)\big|:z \in A_{r}\right\}$ is finite. Since each $\widehat{\psi}_n$ preserves the unit circle, and since $\big\|\widehat{\psi}_n-\psi_n\big\|_{C^0(A_{\underline{r}})} \leq C\lambda^n$ for all $n \geq n_0$, the second claim is proved.

Another way to prove the second claim is by noting that, since $\overline{A_{\alpha}} \subset A_{\beta} \subset \overline{A_{\beta}} \subset A(A_{S_n})$ for all $n \geq n_0$, the hyperbolic metric on any annulus $A(A_{S_n})$ and the Euclidean metric are comparable in $A_{\alpha}$ with universal parameters, that is, there exists a constant $K>1$ such that:$$\left(\frac{1}{K}\right)|z-w| \leq d_{A(A_{S_n})}(z,w) \leq K|z-w|$$for all $z,w \in A_{\alpha}$ and for all $n \geq n_0$, where $d_{A(A_{S_n})}$ denote the hyperbolic distance in the annulus $A(A_{S_n})$ (this is well-known, see for instance \cite[Section I.4, Theorem 4.3]{cargam}). Since each $A\circ\widehat{\psi}_n:A_{\underline{r}} \to A(A_{S_n})$ is holomorphic and preserves the unit circle, we know by Schwarz lemma that for all $z \in A_{\underline{r}}$ and for all $n \geq n_0$ we have:$$d_{A(A_{S_n})}\left((A\circ\widehat{\psi}_n)(z),S^1\right) \leq d_{A_{\underline{r}}}\left(z,S^1\right),$$where $d_{A_{\underline{r}}}$ denote the hyperbolic distance in the annulus $A_{\underline{r}}$. Since all distances $d_{A(A_{S_n})}$ are comparable with the Euclidean distance in $A_{\delta}$ with universal parameters, we have for all $z \in A_{\underline{r}}$ and for all $n \geq n_0$ that:$$d\left((A\circ\widehat{\psi}_n)(z),S^1\right) \leq Kd_{A_{\underline{r}}}\left(z,S^1\right),$$where $d$ is just the Euclidean distance in the plane. Fix $r\in\big(1,(1+r_0)/2\big)$ close enough to one in order to have that $z \in A_r$ implies $d_{A_{\underline{r}}}\left(z,S^1\right)<\frac{\alpha-1}{K\alpha}$ (and therefore $(A\circ\widehat{\psi}_n)(z) \in A_{\alpha}$ for all $n \geq n_0$). Again since $\big\|\widehat{\psi}_n-\psi_n\big\|_{C^0(A_{\underline{r}})} \leq C\lambda^n$ for all $n \geq n_0$, the second claim is proved.

\emph{Third claim}: There exists a positive number $M$ such that $\big|\phi_n'(z)\big|<M$ for all $z \in A_\alpha$ and for all $n \geq n_0$.

Indeed, recall that $\phi_n\big(A(A_{S_n})\big)=\big(G_n \circ P_n\big)\big(\widehat{\mathcal{T}}_n\big) \subset A_{R_n}$ for all $n \geq n_0$. By Lemma \ref{controlRn} there exists a (finite) number $\Delta$ such that $\phi_n\big(A(A_{S_n})\big) \subset B(0,\Delta)$ for all $n \geq n_0$. Since $\overline{A_{\alpha}} \subset A_{\beta} \subset \overline{A_{\beta}} \subset A(A_{S_n})$ for all $n \geq n_0$, the third claim follows from Cauchy derivative estimate.

With the three claims at hand, Lemma \ref{HnGn} follows.
\end{proof}

To control the combinatorics after perturbation we use the \emph{monotonicity} of the rotation number:

\begin{lema}\label{controlcomb} Let $f$ be a $C^3$ critical circle map and let $g$ be a real-analytic critical circle map that extends holomorphically to the annulus:$$A_R=\left\{z\in\C:\frac{1}{R}<\big|z\big|<R\right\}\quad\mbox{for some}\quad R>1.$$

There exists a real-analytic critical circle map $h$, with $\rho(h)=\rho(f)$, also extending holomorphically to $A_R$, where we have:$$\big\|h-g\big\|_{C^0(A_R)} \leq d_{C^0(S^1)}\big(f,g\big).$$

In particular:$$d_{C^r(S^1)}\big(h,g\big)\leq d_{C^0(S^1)}\big(f,g\big)\quad\mbox{for any}\quad 0 \leq r \leq \infty.$$
\end{lema}

\begin{proof}[Proof of Lemma \ref{controlcomb}] Let $F$ and $G$ be the corresponding lifts of $f$ and $g$ to the real line satisfying:$$\rho(f)=\lim_{n\to+\infty}\frac{F^n(0)}{n}\quad\mbox{and}\quad\rho(g)=\lim_{n\to+\infty}\frac{G^n(0)}{n}\,.$$

Consider the band $B_R=\left\{z\in\C:-\log R<2\pi\Im(z)<\log R\right\}$, which is the universal cover of the annulus $A_R$ via the holomorphic covering $z \mapsto e^{2\pi iz}$. Let $\delta=\|F-G\|_{C^0(\R)}$, and for any $t$ in $[-1,1]$ let $G_t:B_R\to\C$ defined as $G_t=G+t\delta$. Each $G_t$ preserves the real line, and its restriction is the lift of a real-analytic critical circle map. Moreover, each $G_t$ commutes with unitary horizontal translation in $B_R$.

Note that $\|G_t-G\|_{C^0(B_R)}=|t|\delta\leq\|F-G\|_{C^0(\R)}$ for any $t\in[-1,1]$. Moreover for any $x\in\R$ the family $\big\{G_t(x)\big\}_{t\in[-1,1]}$ is monotone in $t$, and we have $G_{-1}(x) \leq F(x) \leq G_1(x)$. In particular there exists $t_0\in[-1,1]$ such that:$$\lim_{n\to+\infty}\frac{G_{t_0}^n(0)}{n}=\rho(F)\,,$$and we define $h$ as the projection of $G_{t_0}$ to the annulus $A_R$.
\end{proof}

After the perturbation given by Lemma \ref{controlcomb} we still have the critical point of $h_n$ placed at $1$, but its critical value is no longer placed at $g_n(1)$ (however they are exponentially close). To finish the proof of Proposition \ref{pertanillo} we need to fix this, without changing the combinatorics of $h_n$ in $S^1$. Until now each $H_n$ is $S^1$-symmetric, in the sense that it commutes with $z \mapsto 1/\bar{z}$ in the annulus $A_r$. We will loose this property in the following perturbation, which turns out to be the last one.

\begin{lema}\label{conjmob} For each $n \geq n_0$ consider the (unique) M\"obius transformation $M_n$ which maps the unit disk $\D$ onto itself fixing the basepoint $z=1$, and which maps $H_n(1)$ to $G_n(1)$. Then there exists $\rho\in(1,r)$ such that $\overline{A_{\rho}}\subset M_n(A_r)$ for all $n \geq n_0$. Moreover for each $n \geq n_0$ we have:$$\left\|M_n \circ H_n \circ M_n^{-1}-G_n\right\|_{C^0(A_{\rho})} \leq C\lambda^n.$$
\end{lema}

Note that, when restricted to the unit circle, each $M_n$ gives rise to an orientation-preserving real-analytic diffeomorphism which is, as Lemma \ref{conjmob} indicates, $C^{\infty}$-exponentially close to the identity.

\begin{proof}[Proof of Lemma \ref{conjmob}] Consider the biholomorphism $\psi:\h\to\D$ given by $\psi(z)=\frac{z-i}{z+i}$, whose inverse $\psi^{-1}:\D\to\h$ is given by $\psi^{-1}(z)=i\left(\frac{1+z}{1-z}\right)$. Note that $\psi$ maps the vertical geodesic $\big\{z\in\h:\Re(z)=0\big\}$ onto the interval $(-1,1)$ in $\D$. Since $\psi$ and $\psi^{-1}$ are M\"obius transformations, both extend uniquely to corresponding biholomorphisms of the entire Riemann sphere. The extension of $\psi$ is a real-analytic diffeomorphism between the compactification of the real line and the unit circle, which maps the point at infinity to the point $z=1$. For each $n \geq n_0$ consider the real number $t_n$ defined by:$$t_n=\psi^{-1}\big(G_n(1)\big)-\psi^{-1}\big(H_n(1)\big)=2i\left(\frac{G_n(1)-H_n(1)}{\big(1-G_n(1)\big)\big(1-H_n(1)\big)}\right).$$
Each $t_n$ is finite since for all $n \geq n_0$ both $G_n(1)$ and $H_n(1)$ are not equal to one. Moreover we claim that:$$\inf_{n \geq n_0}\left\{\big|G_n(1)-1\big|\right\}>0\quad\mbox{and}\quad\inf_{n \geq n_0}\left\{\big|H_n(1)-1\big|\right\}>0\,.$$
Indeed, since we have $\big|H_n(1)-G_n(1)\big| \leq C\lambda^n$ for all $n \geq n_0$, is enough to prove that $\inf_{n \geq n_0}\left\{\big|G_n(1)-1\big|\right\}>0$, and this follows by Lemma \ref{koebedist} since $1=P_n(0)$ and $G_n(1)=P_n(-1)$ for all $n \geq n_0$. In particular, again using $\big|H_n(1)-G_n(1)\big| \leq C\lambda^n$ for all $n \geq n_0$, we see that $|t_n| \leq C\lambda^n$ for all $n \geq n_0$. From the explicit formula:$$M_n(z)=\frac{(2i-t_n)z+t_n}{(2i+t_n)-t_nz}=\left(\frac{z-\left(\frac{t_n}{t_n-2i}\right)}{1-\left(\frac{t_n}{t_n+2i}\right)z}\right)\left(\frac{2i-t_n}{2i+t_n}\right)\quad\mbox{for all $n \geq n_0$,}$$we see that the pole of each $M_n$ is at the point $z_n=1+i(2/t_n)$, and since $|t_n| \leq C\lambda^n$ for all $n \geq n_0$, we can take $n_0$ big enough to have that $z_n\in\C\setminus\overline{B(0,2R)}$, where $R=\sup_{n \geq n_0}\{R_n\}<+\infty$ is given by Lemma \ref{controlRn}. A straightforward computation gives:$$\big(M_n-Id\big)(z)=\frac{t_n(z-1)^2}{(2i+t_n)-t_nz}\quad\mbox{for all $n \geq n_0$,}$$and therefore:$$\big\|M_n-Id\big\|_{C^0(A_{R})} \leq C\lambda^n\quad\mbox{for all $n \geq n_0$.}$$
In particular for any fixed $\rho\in(1,r)$ we can choose $n_0$ big enough in order to have $\overline{A_{\rho}}\subset M_n(A_r)$ for all $n \geq n_0$. Moreover given any $z \in A_{\rho}$ we have:
\begin{align}
\big(M_n \circ H_n \circ M_n^{-1}-G_n\big)(z)&=\big(M_n-Id\big)\big((H_n \circ M_n^{-1})(z)\big)+\big(H_n-G_n\big)(z)\notag\\
&+\big(H_n\big(M_n^{-1}(z)\big)-H_n(z)\big).\notag
\end{align}
In particular:
\begin{align}
\left\|M_n \circ H_n \circ M_n^{-1}-G_n\right\|_{C^0(A_{\rho})}&\leq\left\|M_n-Id\right\|_{C^0\big(H_n(A_r)\big)}+\left\|H_n-G_n\right\|_{C^0(A_{\rho})}\notag\\
&+\left\|H_n\right\|_{C^1(A_r)}\left\|M_n^{-1}-Id\right\|_{C^0(A_{\rho})}.\notag
\end{align}
Since $H_n(A_r) \subset A_R$ and $A_{\rho} \subset A_r \subset A_R$, the three terms $\left\|M_n-Id\right\|_{C^0\big(H_n(A_r)\big)}$, $\left\|H_n-G_n\right\|_{C^0(A_{\rho})}$ and $\left\|M_n^{-1}-Id\right\|_{C^0(A_{\rho})}$ are less or equal than $C\lambda^n$ for all $n \geq n_0$.

Finally, since each $H_n$ is holomorphic and we have $\overline{A_r} \subset A_{\underline{r}}$ and $H_n(A_{\underline{r}})\subset\big(G_n \circ P_n\big)\big(\widehat{\mathcal{T}}_n\big)\subset A_{R_n}\subset A_R$ for all $n \geq n_0$, we obtain from Cauchy derivative estimate that:$$\sup_{n \geq n_0}\left\{\big\|H_n\big\|_{C^1(A_r)}\right\}$$is finite, and therefore:$$\left\|M_n \circ H_n \circ M_n^{-1}-G_n\right\|_{C^0(A_{\rho})} \leq C\lambda^n\quad\mbox{for all $n \geq n_0$.}$$
\end{proof}
With Lemma \ref{conjmob} at hand we are done since $\big(M_n \circ H_n \circ M_n^{-1}\big)(1)=G_n(1)$. We have finished the proof of Proposition \ref{pertanillo}.
\end{proof}

\subsection{The shadowing sequence}\label{sub3}

This is the final subsection of Section \ref{Secglueing}, which is devoted to proving Theorem \ref{compacto}. Let us recall what we have done: in Subsection \ref{sub1} we constructed a suitable sequence $\{G_n\}_{n \geq n_0}$ of $S^1$-symmetric $C^3$ extensions of $C^3$ critical circle maps $g_n$ to some annulus $P_n\big(\widehat{\mathcal{T}}_n\big)$. When lifted with the corresponding projection $P_n$ (also constructed in Subsection \ref{sub1}) each $g_n$ gives rise to a $C^3$ critical commuting pair $\big(\widehat{\eta}_n,\widehat{\xi}_n\big)$ exponentially close to $\mathcal{R}^{n}(f)$ and having the same combinatorics at each step (moreover, with complex extensions $C^0$-exponentially close to the ones of $\mathcal{R}^{n}(f)$ produced in Theorem \ref{nuevoE}, see Proposition \ref{Gn} above for more properties).

In Subsection \ref{sub2} we perturbed each $G_n$ in a definite annulus $A_r$, in order to obtain a sequence of real-analytic critical circle maps, each of them having the same combinatorics as the corresponding $\mathcal{R}^{n}(f)$, that extend to holomorphic maps $H_n$ exponentially close to $G_n$ in $A_r$ (see Proposition \ref{pertanillo} above for more properties). Both the critical point and the critical value of each $H_n$ coincide with the ones of the corresponding $G_n$, more precisely, the critical point of each $H_n$ is at $P_n(0)=1 \in P_n\big(V_1(n)\big) \cap S^1$, and its critical value is at $H_n(1)=G_n(1) \in P_n(V) \cap S^1=P_n\big(B_n(V)\big) \cap S^1$. Recall also that $H_n(A_r)\subset P_n(\mathcal{T}_n)$ for all $n \geq n_0$.

In this subsection we lift each $H_n:A_r \to A_{R_n}$ via the holomorphic projection $P_n:\mathcal{T}_n \to A_{R_n}$ in the canonical way: let $\alpha>0$ such that for all $n \geq n_0$ we have that:$$\overline{N_{\alpha}\big([-1,0]\big) \cup N_{\alpha}\big([0,\widehat{\xi}_n(0)]\big)}\subset\widehat{\mathcal{T}}_n\,,$$and that $P_n\big(N_{\alpha}\big([-1,0]\big) \cup N_{\alpha}\big([0,\widehat{\xi}_n(0)]\big)\big)$ is an annulus contained in $A_r$ and containing the unit circle (the existence of such $\alpha$ is guaranteed by Lemma \ref{vert} and Lemma \ref{koebedist}). Let us use the more compact notation $Z_1(n)=N_{\alpha}\big([-1,0]\big)$ and $Z_2(n)=N_{\alpha}\big([0,\widehat{\xi}_n(0)]\big)$. For each $n \geq n_0$ let $\widetilde{\eta}_n:Z_2(n) \to \mathcal{T}_n$ be the $\R$-preserving holomorphic map defined by the two conditions:$$H_n \circ P_n = P_n\circ\widetilde{\eta}_n\mbox{ in $Z_2(n)$, and }\widetilde{\eta}_n(0)=-1\,.$$
In the same way let $\widetilde{\xi}_n:Z_1(n) \to \mathcal{T}_n$ be the $\R$-preserving holomorphic map defined by the two conditions:$$H_n \circ P_n = P_n\circ\widetilde{\xi}_n\mbox{ in $Z_1(n)$, and }\widetilde{\xi}_n(0)=\widehat{\xi}_n(0)\,.$$

$$\begin{diagram}
\node{Z_1(n) \cup Z_2(n)\subset\mathcal{T}_n}\arrow{e,t}{(\widetilde{\eta}_n,\,\widetilde{\xi}_n)}\arrow{s,l}{P_n}
\node{\mathcal{T}_n}\arrow{s,r}{P_n}\\
\node{A_r \subset A_{R_n}}\arrow{e,t}{H_n}\node{A_{R_n}}
\end{diagram}$$
In the next proposition we summarize the main properties of this lift, which are all straightforward:

\begin{prop}[The shadowing sequence]\label{shadow} For each $n \geq n_0$ the pair $f_n=(\widetilde{\eta}_n,\widetilde{\xi}_n)$ restricts to a real-analytic critical commuting pair with domains $I\big(\widetilde{\xi}_n\big)=\big[\widetilde{\eta}_n(0),0\big]=[-1,0]$ and $I\big(\widetilde{\eta}_n\big)=\big[0,\widetilde{\xi}_n(0)\big]=\big[0,\widehat{\xi}_n(0)\big]$, and such that $\rho(f_n)=\rho\big(\widehat{\eta}_n,\widehat{\xi}_n\big)=\rho\big(\mathcal{R}^n(f)\big)\in\R\setminus\Q$. Moreover $\widetilde{\xi}_n$ and $\widetilde{\eta}_n$ extend to holomorphic maps in $Z_1(n)$ and $Z_2(n)$ respectively where we have:
\begin{itemize}
\item $\widetilde{\xi}_n$ has a unique critical point in $Z_1(n)$, which is at the origin and of cubic type.
\item $\widetilde{\eta}_n$ has a unique critical point in $Z_2(n)$, which is at the origin and of cubic type.
\item $\left\|\widetilde{\xi}_n-\widehat{\xi}_n\right\|_{C^0\big(Z_1(n)\cap\Phi_n(\widehat{X}_1(n))\big)} \leq C\lambda^n$.
\item $\left\|\widetilde{\eta}_n-\widehat{\eta}_n\right\|_{C^0\big(Z_2(n)\cap\Phi_n(\widehat{X}_2(n))\big)} \leq C\lambda^n$.
\end{itemize}
\end{prop}

With Proposition \ref{shadow} at hand, Theorem \ref{compacto} follows directly from the following consequence of Montel's theorem:

\begin{lema}\label{compact} Let $\alpha$ be a constant in $(0,1)$ and let $\mathcal{V}$ be an $\R$-symmetric bounded topological disk such that $[-1,\alpha^{-1}]\subset\mathcal{V}$. Let $W_1$ and $W_2$ be topological disks whose closure is contained in $\mathcal{V}$ and such that $[-1,0] \subset W_1$ and $[0,\alpha^{-1}] \subset W_2$. Denote by $\mathcal{K}$ the set of all normalized real-analytic critical commuting pairs $\zeta=(\eta,\xi)$ satisfying the following three conditions:

\begin{itemize}
\item $\eta(0)=-1$ and $\xi(0)\in[\alpha,\alpha^{-1}]$,
\item $\alpha\big|\eta\big([0,\xi(0)]\big)\big|\leq\big|\xi\big([-1,0]\big)\big|\leq\alpha^{-1}\big|\eta\big([0,\xi(0)]\big)\big|$,
\item Both $\xi$ and $\eta$ extend to holomorphic maps (with a unique cubic critical point at the origin) defined in $W_1$ and $W_2$ respectively, where we have:

\begin{enumerate}
\item $N_{\alpha}\big(\xi\big([-1,0]\big)\big)\subset\xi(W_1)$;
\item $N_{\alpha}\big(\eta\big([0,\xi(0)]\big)\big)\subset\eta(W_2)$;
\item $\xi(W_1)\cup\eta(W_2)\subset\mathcal{V}$.
\end{enumerate}
\end{itemize}

Then $\mathcal{K}$ is $C^{\omega}$-compact.
\end{lema}

\section{Concluding remarks}\label{Secfinal}

The set $\A \subset [0,1]$ of de Faria and de Melo (see Theorem \ref{primeroedsonwelington}) is the set of rotation numbers $\rho=[a_0,a_1,...]$ satisfying the following three properties:$$\lim\sup_{n\to\infty}\frac{1}{n}\sum_{j=1}^{n}\log a_j<\infty$$$$\lim_{n\to\infty}\frac{1}{n}\log a_n=0$$$$\frac{1}{n}\sum_{j=k+1}^{k+n}\log a_j\leq\omega_{\rho}\left(\frac{n}{k}\right)$$for all $0 < n \leq k$, where $\omega_{\rho}(t)$ is a positive function (that depends on the rotation number) defined for $t>0$ such that $t\omega_{\rho}(t) \to 0$ as $t\to 0$ (for instante we can take $\omega_{\rho}(t)=C_{\rho}(1-\log t)$ where $C_{\rho}>0$ depends on the number).

The set $\A$ obviously contains all rotation numbers of bounded type, and it has full Lebesgue measure in $[0,1]$ (see \cite[Appendix C]{edsonwelington1}).

It is natural to ask: is there a condition on the rotation number equivalent to the $C^{1+\alpha}$ rigidity? This is not clear even in the real-analytic setting. We remark that $C^{1+\alpha}$ rigidity fails for some Diophantine rotation numbers (for instance with $\rho=[2,2^2,2^{2^2},... ,2^{2^n},... ]$, see \cite{edsonwelington1}).

As we said at the begining, it would be desirable to obtain Theorem \ref{main} for $C^3$ critical circle maps with any irrational rotation number, but we have not been able to do this yet. The main difficulty is to control the distance of the successive renormalizations of two critical commuting pairs with a common unbounded type rotation number (compare Lemma \ref{Lipschitz}). That is why we were able to prove that Theorem \ref{compacto} implies Theorem \ref{expconv} only for bounded type rotation numbers.

If we can prove Theorem \ref{expconv} for any irrational rotation number, then (by Theorem \ref{primeroedsonwelington}) we can extend Theorem \ref{main} to the full Lebesgue measure set $\A$, and using Theorem \ref{condition} (with essentially the same arguments as in \cite{avilamartensdemelo} to obtain exponential convergence in the $C^2$ metric) we would be able to obtain $C^1$-rigidity for all rotation numbers.

Another difficult problem is the following: what can be said, in terms of smooth rigidity, for maps with finitely many non-flat critical points? More precisely, let $f$ and $g$ be two orientation preserving $C^3$ circle homeomorphisms with the same irrational rotation number, and with $k \geq 1$ non-flat critical points of odd type. Denote by $S_f=\{c_1,...,c_k\}$ the critical set of $f$, by $S_g=\{c'_1,...,c'_k\}$ the critical set of $g$, and by $\mu_f$ and $\mu_g$ their corresponding unique invariant measures. Beside the quantity and type of the critical points, new smooth conjugacy invariants appear: the condition $\mu_f\big([c_i,c_{i+1}]\big)=\mu_g\big([c'_i,c'_{i+1}]\big)$ for all $i \in \{1,...,k-1\}$ is necessary (and sufficient) in order to have a conjugacy that sends the critical points of $f$ to the critical points of $g$ (the only one that can be smooth). Are those the unique smooth conjugacy invariants?

\appendix

\section{Proof of Lemma \ref{Lipschitz}}\label{provaLip}

In this appendix we prove Lemma \ref{Lipschitz}, stated at the end of Section \ref{Secren} and used in Section \ref{Secred}. For that we need the following fact:

\begin{lema}\label{comp} Let $f_1,...,f_n$ be $C^1$ maps with $C^1$ norm bounded by some constant $B>0$, and let $g_1,...,g_n$ be $C^0$ maps. Then:$$\big\|f_n \circ ... \circ f_1 - g_n \circ ... \circ g_1\big\|_{C^0} \leq \left(\sum_{j=0}^{n-1}B^j\right)\max_{i\in\{1,...,n\}}\big\{\big\|f_i-g_i\big\|_{C^0}\big\}$$whereas the compositions makes sense.
\end{lema}

\begin{proof} The proof goes by induction on $n$ (when $n=1$ we have nothing to prove). Suppose that:$$\big\|f_{n-1} \circ ... \circ f_1 - g_{n-1} \circ ... \circ g_1\big\|_{C^0} \leq \left(\sum_{j=0}^{n-2}B^j\right)\max_{i\in\{1,...,n-1\}}\big\{\big\|f_i-g_i\big\|_{C^0}\big\}\,.$$
Then for any $t$:
\begin{align}
\big|(f_n \circ ... \circ f_1 - g_n \circ ... \circ g_1)(t)\big|&\leq\big|f_n((f_{n-1} \circ ... \circ f_1)(t))-f_n((g_{n-1} \circ ... \circ g_1)(t))\big|+\notag\\
&+\big|f_n((g_{n-1} \circ ... \circ g_1)(t))-g_n((g_{n-1} \circ ... \circ g_1)(t))\big|\notag\\
&\leq B\big|(f_{n-1} \circ ... \circ f_1 - g_{n-1} \circ ... \circ g_1)(t)\big|+\big\|f_n-g_n\big\|_{C^0}\notag\\
&\leq B\left(\sum_{j=0}^{n-2}B^j\right)\max_{i\in\{1,...,n-1\}}\big\{\big\|f_i-g_i\big\|_{C^0}\big\}+\big\|f_n-g_n\big\|_{C^0}\notag\\
&\leq\left(\sum_{j=0}^{n-1}B^j\right)\max_{i\in\{1,...,n\}}\big\{\big\|f_i-g_i\big\|_{C^0}\big\}\,.\notag
\end{align}
\end{proof}

For $K>1$ and $r\in\{0,1,...,\infty,\omega\}$ recall from Section \ref{Secren} that we denote by $\mathcal{P}^r(K)$ the space of $C^r$ critical commuting pairs $\zeta=(\eta,\xi)$ such that $\eta(0)=-1$ (they are normalized) and $\xi(0) \in [K^{-1},K]$.

\begin{lema}\label{corolario} Given $M\in\nt$, $B>0$ and $K>1$ there exists $L(M,B,K)>1$ with the following property: let $\zeta_1=(\eta_1,\xi_1)$ and $\zeta_2=(\eta_2,\xi_2)$ be two renormalizable $C^3$ critical commuting pairs satisfying the following five conditions:
\begin{enumerate}
\item\label{apAit1} $\zeta_1$, $\mathcal{R}(\zeta_1)$, $\zeta_2$ and $\mathcal{R}(\zeta_2)$ belong to $\mathcal{P}^3(K)$.
\item\label{apAit2} The continued fraction expansion of both rotation numbers $\rho(\zeta_1)$ and $\rho(\zeta_2)$ have the same first term, say $a_0$, with $a_0 \leq M$. More precisely:$$\left\lfloor\frac{1}{\rho(\zeta_1)}\right\rfloor=\left\lfloor\frac{1}{\rho(\zeta_2)}\right\rfloor=a_0\in\big\{1,...,M\big\}\,.$$
\item\label{apAit3} $\max\big\{\|\eta_1\|_{C^1},\|\xi_1\|_{C^1}\big\}<B$.
\item\label{apAit4} $\big(\eta_1\circ\xi_1\big)(0)$ and $\big(\eta_2\circ\xi_2\big)(0)$ have the same sign.
\item\label{apAit5} $$\big|\xi_1(0)-\xi_2(0)\big|<\left(\frac{1}{K^2}\right)\left(\frac{K+1}{K-1}\right).$$
\end{enumerate}
Then we have:$$d_{0}\big(\mathcal{R}(\zeta_1),\mathcal{R}(\zeta_2)\big) \leq L \cdot d_{0}(\zeta_1,\zeta_2)\,,$$where $d_0$ is the $C^0$ distance in the space of critical commuting pairs (see Section \ref{submetric}).
\end{lema}

\begin{proof} Suppose that both $\big(\eta_1\circ\xi_1\big)(0)$ and $\big(\eta_2\circ\xi_2\big)(0)$ are positive, and let $V\subset\R$ be the interval $\big[0,\max\big\{(\eta_1\circ\xi_1)(0),(\eta_2\circ\xi_2)(0)\big\}\big]$. For $\alpha > 0$ denote by $T_{\alpha}$ the (unique) M\"obius transformation that fixes $-1$ and $0$, and maps $\alpha$ to $1$. Note that $p_{\alpha}=\alpha+\alpha\left(\frac{1+\alpha}{1-\alpha}\right)$ is the pole of $T_{\alpha}$. If $\alpha>K/(K+2)$ then $p_{\alpha}\notin[1/K,K]$, and if $\alpha\in[1/K,K/(K+2)]$ then $p_{\alpha}-\alpha\geq\left(\frac{1}{K}\right)\left(\frac{K+1}{K-1}\right)$. By Item \eqref{apAit5} in the hypothesis, and since $\zeta_1$ and $\zeta_2$ belong to $\mathcal{P}^3(K)$ by Item \eqref{apAit1}, there exists $L_0(K)>1$ such that:$$\big\|T_{\xi_1(0)}\big\|_{C^1(V)} \leq L_0\,,$$

$$\big\|T_{\xi_1(0)}-T_{\xi_2(0)}\big\|_{C^0(V)}\leq L_0\big|\xi_1(0)-\xi_2(0)\big|\leq L_0 \cdot d_{0}(\zeta_1,\zeta_2)\,,$$

$$\big|\eta_1^{a_0}\big(\xi_1(0)\big)-\eta_2^{a_0}\big(\xi_2(0)\big)\big| \leq L_0\big|\widetilde{\eta}_1^{a_0}(1)-\widetilde{\eta}_2^{a_0}(1)\big|\quad\mbox{and}$$

$$\big\|\widetilde{\eta}_1\big\|_{C^1([0,1])} \leq L_0\big\|\eta_1\big\|_{C^1([0,\xi_1(0)])} \leq L_0B\,,$$where $\widetilde{\eta}_i=T_{\xi_i(0)}\circ\eta_i\circ T_{\xi_i(0)}^{-1}$ for $i\in\{1,2\}$. By Lemma \ref{comp}:$$\big|\eta_1^{a_0}\big(\xi_1(0)\big)-\eta_2^{a_0}\big(\xi_2(0)\big)\big| \leq L_0^2\left(\sum_{j=0}^{a_0-1}B^j\right)\big\|\widetilde{\eta}_1-\widetilde{\eta}_2\big\|_{C^0([0,1])}\,.$$
Defining $L_1(M,B,K)=L_0^2\left(\sum_{j=0}^{M-1}B^j\right)$ we obtain:$$\big|\eta_1^{a_0}\big(\xi_1(0)\big)-\eta_2^{a_0}\big(\xi_2(0)\big)\big| \leq L_1 \cdot d_0(\zeta_1,\zeta_2)\,.$$
Therefore:
\begin{align}
\big|T_{\xi_1(0)}\big(\eta_1^{a_0}(\xi_1(0))\big)-T_{\xi_2(0)}\big(\eta_2^{a_0}(\xi_2(0))\big)\big|&\leq|T_{\xi_1(0)}\big(\eta_1^{a_0}(\xi_1(0))\big)-T_{\xi_1(0)}\big(\eta_2^{a_0}(\xi_2(0))\big)|\notag\\
&+|T_{\xi_1(0)}\big(\eta_2^{a_0}(\xi_2(0))\big)-T_{\xi_2(0)}\big(\eta_2^{a_0}(\xi_2(0))\big)|\notag\\
&\leq L_0\big|\eta_1^{a_0}\big(\xi_1(0)\big)-\eta_2^{a_0}\big(\xi_2(0)\big)\big|+L_0 \cdot d_{0}(\zeta_1,\zeta_2)\notag\\
&\leq(L_0L_1+L_0) \cdot d_{0}(\zeta_1,\zeta_2)\,.\notag
\end{align}
Defining $L_2(M,B,K)=L_0L_1+L_0$ we obtain:
\begin{equation}\label{segest}
\big|T_{\xi_1(0)}\big(\eta_1^{a_0}(\xi_1(0))\big)-T_{\xi_2(0)}\big(\eta_2^{a_0}(\xi_2(0))\big)\big| \leq L_2 \cdot d_0(\zeta_1,\zeta_2)\,.
\end{equation}
Moreover there exists $L_3(M,B,K) \geq L_2$ with the following four properties:
\begin{itemize}
\item From Item \eqref{apAit1} in the hypothesis, both M\"obius transformations:$$T_{T_{\xi_i(0)}\big(\eta_i^{a_0}(\xi_i(0))\big)}$$and also their inverses have $C^1$ norm bounded by $L_3$ in:$$W=\big[0,\max\big\{T_{\xi_1(0)}\big(\eta_1^{a_0}(\xi_1(0))\big),T_{\xi_2(0)}\big(\eta_2^{a_0}(\xi_2(0))\big)\big\}\big].$$
\item Both M\"obius transformations:$$T_{T_{\xi_i(0)}\big(\eta_i^{a_0}(\xi_i(0))\big)}$$are at $C^0$-distance less or equal than $L_3 \cdot d_0(\zeta_1,\zeta_2)$ in $W$ (this follows from \eqref{segest} and Item \eqref{apAit1} in the hypothesis).
\item The same with their inverses, that is, both M\"obius transformations:$$T_{T_{\xi_i(0)}\big(\eta_i^{a_0}(\xi_i(0))\big)}^{-1}$$are at $C^0$-distance less or equal than $L_3 \cdot d_0(\zeta_1,\zeta_2)$ in $[0,1]$ (again this follows from \eqref{segest} and Item \eqref{apAit1} in the hypothesis).
\item The maps:$$T_{\xi_1(0)} \circ \eta_1^{a_0} \circ \xi_1 \circ T_{\xi_1(0)}^{-1} \quad\mbox{and}\quad T_{\xi_1(0)} \circ \eta_1 \circ T_{\xi_1(0)}^{-1}$$have $C^1$ norm bounded by $L_3$ in $[-1,0]$ and $[0,1]$ respectively (this follows from items \eqref{apAit1}, \eqref{apAit2} and \eqref{apAit3} in the hypothesis).
\end{itemize}

Note that for $i=\{1,2\}$ we have:$$T_{\eta_i^{a_0}\big(\xi_i(0)\big)} \circ \eta_i^{a_0} \circ \xi_i \circ T_{\eta_i^{a_0}\big(\xi_i(0)\big)}^{-1}=$$

$$=T_{T_{\xi_i(0)}\big(\eta_i^{a_0}(\xi_i(0))\big)} \circ \left(T_{\xi_i(0)} \circ \eta_i^{a_0} \circ \xi_i \circ T_{\xi_i(0)}^{-1}\right) \circ T_{T_{\xi_i(0)}\big(\eta_i^{a_0}(\xi_i(0))\big)}^{-1}$$in $[-1,0]$, and:$$T_{\eta_i^{a_0}\big(\xi_i(0)\big)} \circ \eta_i \circ T_{\eta_i^{a_0}\big(\xi_i(0)\big)}^{-1}=$$

$$=T_{T_{\xi_i(0)}\big(\eta_i^{a_0}(\xi_i(0))\big)} \circ \left(T_{\xi_i(0)} \circ \eta_i \circ T_{\xi_i(0)}^{-1}\right) \circ T_{T_{\xi_i(0)}\big(\eta_i^{a_0}(\xi_i(0))\big)}^{-1}$$in $[0,1]$. By Lemma \ref{comp} and the four properties quoted above there exists $L_4(M,B,K) \geq L_3$ such that:$$\big\|T_{\eta_1^{a_0}\big(\xi_1(0)\big)} \circ \eta_1^{a_0} \circ \xi_1 \circ T_{\eta_1^{a_0}\big(\xi_1(0)\big)}^{-1}-T_{\eta_2^{a_0}\big(\xi_2(0)\big)} \circ \eta_2^{a_0} \circ \xi_2 \circ T_{\eta_2^{a_0}\big(\xi_2(0)\big)}^{-1}\big\|_{C^0} \leq$$

$$\leq L_4\max\big\{\big\|T_{T_{\xi_1(0)}\big(\eta_1^{a_0}(\xi_1(0))\big)}-T_{T_{\xi_2(0)}\big(\eta_2^{a_0}(\xi_2(0))\big)}\big\|_{C^0},$$

$$d_0(\zeta_1,\zeta_2),\big\|T_{T_{\xi_1(0)}\big(\eta_1^{a_0}(\xi_1(0))\big)}^{-1}-T_{T_{\xi_2(0)}\big(\eta_2^{a_0}(\xi_2(0))\big)}^{-1}\big\|_{C^0}\big\}$$

$$\leq L_3L_4 \cdot d_0(\zeta_1,\zeta_2)\,.$$in $[-1,0]$, and:$$\big\|T_{\eta_1^{a_0}\big(\xi_1(0)\big)} \circ \eta_1 \circ T_{\eta_1^{a_0}\big(\xi_1(0)\big)}^{-1}-T_{\eta_2^{a_0}\big(\xi_2(0)\big)} \circ \eta_2 \circ T_{\eta_2^{a_0}\big(\xi_2(0)\big)}^{-1}\big\|_{C^0} \leq$$

$$\leq L_4\max\big\{\big\|T_{T_{\xi_1(0)}\big(\eta_1^{a_0}(\xi_1(0))\big)}-T_{T_{\xi_2(0)}\big(\eta_2^{a_0}(\xi_2(0))\big)}\big\|_{C^0},$$

$$d_0(\zeta_1,\zeta_2),\big\|T_{T_{\xi_1(0)}\big(\eta_1^{a_0}(\xi_1(0))\big)}^{-1}-T_{T_{\xi_2(0)}\big(\eta_2^{a_0}(\xi_2(0))\big)}^{-1}\big\|_{C^0}\big\}$$

$$\leq L_3L_4 \cdot d_0(\zeta_1,\zeta_2)\,.$$in $[0,1]$. Therefore we are done by taking $L \geq L_3L_4$.
\end{proof}

\begin{proof}[Proof of Lemma \ref{Lipschitz}] Let $f$ be a $C^3$ critical circle map with irrational rotation number $\rho(f)=[a_0,a_1,...,a_n,a_{n+1},...]$, and recall that we are assuming that $a_n < M$ for all $n\in\nt$. Let $n_0(f)\in\nt$ given by the real bounds, and note that $\mathcal{R}^n(f)\in\mathcal{P}^3(K)$ for all $n \geq n_0$ since $K>K_0$ by hypothesis and therefore $\mathcal{P}^3(K)\supset\mathcal{P}^3(K_0)$. As a well-known corollary of the real bounds (see for instance \cite[Theorem 3.1]{edsonwelington1}) there exists a constant $B>0$ such that the sequence $\{\mathcal{R}^n(f)\}_{n\in\nt}$ is bounded in the $C^1$ metric by $B$, and we are done by taking $L>1$ given by Lemma \ref{corolario}.
\end{proof}

\section{Proof of Proposition \ref{aprox2}}\label{apA} In this appendix we give the proof of Proposition \ref{aprox2} of Section \ref{SecAB}:

\begin{proof}[Proof of Proposition \ref{aprox2}] Assume that each $\mu_n$ is defined in the whole complex plane, just by extending as zero in the complement of the domain $U$, that is:$$\mu_n(z)\,\partial G_n(z)=\overline{\partial}G_n(z)\mbox{ for a.e. $z \in U$, and }\mu_n(z)=0\mbox{ for all $z\in\C\setminus U$.}$$

Fix $n\in\nt$. If $\mu_n \equiv 0$ we take $H_n=G_n|_{V}$, so assume that $\|\mu_n\|_{\infty}>0$ and fix some small $\varepsilon\in\big(0,1-\|\mu_n\|_{\infty}\big)$. Denote by $\Lambda$ the open disk $B\big(0,(1-\varepsilon)/\|\mu_n\|_{\infty}\big)$ centred at the origin and with radius $(1-\varepsilon)/\|\mu_n\|_{\infty}$ in the complex plane (note that $\overline{\D}\subset\Lambda$). Consider the one-parameter family of Beltrami coefficients $\big\{\mu_n(t)\big\}_{t\in\Lambda}$ defined by:$$\mu_n(t)=t\mu_n.$$

Note that for all $t\in\Lambda$ we have $\big\|\mu_n(t)\big\|_{\infty}<1-\varepsilon<1$. Denote by $f^{\mu_n(t)}$ the solution of the Beltrami equation with coefficient $\mu_n(t)$, given by Theorem \ref{morrey}, normalized to fix $0$, $1$ and $\infty$. Note that $f^{\mu_n(0)}$ is the identity and that, by uniqueness, there exists a biholomorphism $H_n:f^{\mu_n(1)}(U)\to G_n(U)$ such that:$$G_n=H_n \circ f^{\mu_n(1)}\mbox{ in $U$.}$$

By Ahlfors-Bers theorem (Theorem \ref{AB}) we know that for any $z\in\C$ the curve $\big\{f^{\mu_n(t)}(z):t\in[0,1]\big\}$ is smooth, that is, the derivative of $f^{\mu_n(t)}$ with respect to the parameter $t$ exists at any $z\in\C$ and any $s \in [0,1]$. Following Ahlfors \cite[Chapter V, Section C]{ahlfors}, we use the notation:$$\dot{f}_n(z,s)=\lim_{t \to 0}\frac{f^{\mu_n(s+t)}(z)-f^{\mu_n(s)}(z)}{t}\,.$$

The limit exists for every $z\in\C$ and every $s\in[0,1]$ (actually for every $s\in\Lambda$), and the convergence is uniform on compact sets of $\C$. Then we have:$$\left\|f^{\mu_n(1)}-Id\right\|_{C^0(U)}=\sup_{z \in U}\left\{\left|f^{\mu_n(1)}(z)-z\right|\right\}\leq\sup_{z \in U}\left\{\int_{0}^{1}\big|\dot{f}_n(z,s)\big|ds\right\}.$$

Moreover, $\dot{f}_n$ has the following integral representation (see \cite[Chapter V, Section C, Theorem 5]{ahlfors} for the explicit computation):$$\dot{f}_n(z,s)=-\left(\frac{1}{\pi}\right)\iint_{\C}\mu_n(w)\,S\big(f^{\mu_n(s)}(w),f^{\mu_n(s)}(z)\big)\left(\partial f^{\mu_n(s)}(w)\right)^2dxdy\,,$$for every $z\in\C$ and every $s\in[0,1]$, where $w=x+iy$ and:$$S(w,z)=\frac{1}{w-z}-\frac{z}{w-1}+\frac{z-1}{w}=\frac{z(z-1)}{w(w-1)(w-z)}\,.$$

Since each $\mu_n$ is supported in $U$ we have:$$\dot{f}_n(z,s)=-\left(\frac{1}{\pi}\right)\iint_{U}\mu_n(w)\,S\big(f^{\mu_n(s)}(w),f^{\mu_n(s)}(z)\big)\left(\partial f^{\mu_n(s)}(w)\right)^2dxdy\,.$$

From the formula:$$\left|\partial f^{\mu_n(s)}(w)\right|^2=\left(\frac{1}{1-|s|^2|\mu_n(w)|^2}\right)\det\big(Df^{\mu_n(s)}(w)\big)$$we obtain:
\begin{align}
\left|\dot{f}_n(z,s)\right|&\leq\frac{1}{\pi}\iint_{U}\left(\frac{|\mu_n(w)|}{1-|s|^2|\mu_n(w)|^2}\right)\det\big(Df^{\mu_n(s)}(w)\big)\big|S\big(f^{\mu_n(s)}(w),f^{\mu_n(s)}(z)\big)\big|dxdy\notag\\
&\leq\frac{1}{\pi}\left(\frac{\|\mu_n\|_{\infty}}{1-|s|^2\|\mu_n\|_{\infty}^2}\right)\iint_{U}\det\big(Df^{\mu_n(s)}(w)\big)\big|S\big(f^{\mu_n(s)}(w),f^{\mu_n(s)}(z)\big)\big|dxdy\notag\\
&=\frac{1}{\pi}\left(\frac{\|\mu_n\|_{\infty}}{1-|s|^2\|\mu_n\|_{\infty}^2}\right)\iint_{f^{\mu_n(s)}(U)}\big|S\big(w,f^{\mu_n(s)}(z)\big)\big|dxdy\,.\notag
\end{align}
Therefore the length of the curve $\big\{f^{\mu_n(t)}(z):t\in[0,1]\big\}$ is less or equal than:$$\frac{1}{\pi}\int_{0}^{1}\left[\left(\frac{\|\mu_n\|_{\infty}}{1-|s|^2\|\mu_n\|_{\infty}^2}\right)\iint_{f^{\mu_n(s)}(U)}\big|S\big(w,f^{\mu_n(s)}(z)\big)\big|dxdy\right]ds\leq$$
$$\leq\left(\frac{1}{\pi}\right)\left(\frac{\|\mu_n\|_{\infty}}{1-\|\mu_n\|_{\infty}^2}\right)\int_{0}^{1}\left[\iint_{f^{\mu_n(s)}(U)}\big|S\big(w,f^{\mu_n(s)}(z)\big)\big|dxdy\right]ds\,.$$

If we define:$$M_n(U)=\left(\frac{1}{\pi}\right)\sup_{z \in U}\left\{\int_{0}^{1}\left[\iint_{f^{\mu_n(s)}(U)}\big|S\big(w,f^{\mu_n(s)}(z)\big)\big|dxdy\right]ds\right\},$$we get:$$\left\|f^{\mu_n(1)}-Id\right\|_{C^0(U)}\leq\left(\frac{\|\mu_n\|_{\infty}}{1-\|\mu_n\|_{\infty}^2}\right)M_n(U).$$

We have two remarks:

First remark: since $\mu_n \to 0$ in the unit ball of $L^{\infty}$, we know by Proposition \ref{convident} that for any $s\in[0,1]$ the normalized quasiconformal homeomorphisms $f^{\mu_n(s)}$ converge to the identity uniformly on compact sets of $\C$, in particular on $\overline{U}$. Therefore the sequence $M_n(U)$ converge to:$$\left(\frac{1}{\pi}\right)\sup_{z \in U}\left\{\iint_{U}\big|S\big(w,z\big)\big|dxdy\right\}<\left(\frac{1}{\pi}\right)\sup_{z \in U}\left\{\iint_{\C}\big|S\big(w,z\big)\big|dxdy\right\}<\infty.$$

For fixed $z\in\C$ we have that $S(w,z)$ is in $L^1(\C)$ since it has simple poles at $0$, $1$ and $z$, and is $O\big(|w|^{-3}\big)$ near $\infty$. The finiteness follows then from the compactness of $\overline{U}$.

Second remark: $x\mapsto x/(1-x^2)$ is an orientation-preserving real-analytic diffeomorphism between $(-1,1)$ and the real line, which is tangent to the identity at the origin. In fact $x/(1-x^2)=x+o(x^2)$ in $(-1,1)$.

With this two remarks we obtain $n_1\in\nt$ such that for all $n \geq n_1$ we have:$$\left\|f^{\mu_n(1)}-Id\right\|_{C^0(U)} \leq M(U)\|\mu_n\|_{\infty},$$where:$$M(U)=\left(\frac{2}{\pi}\right)\sup_{z \in U}\left\{\iint_{U}\big|S\big(w,z\big)\big|dxdy\right\}.$$

Since $V$ is compactly contained in the bounded domain $U$, the boundaries $\partial V$ and $\partial U$ are disjoint compact sets. Let $\delta>0$ be its Euclidean distance, that is, $\delta=d\big(\partial V,\partial U\big)=\min\big\{|z-w|:z\in\partial V,w\in\partial U\big\}$. Again by Proposition \ref{convident} we know, since $\mu_n \to 0$, that there exists $n_0 \geq n_1$ in $\nt$ such that for all $n \geq n_0$ we have $V \subset f^{\mu_n(1)}(U)$ and moreover:$$f^{\mu_n(1)}(U) \supseteq B(z,\delta/2)\quad\mbox{for all $z \in V$.}$$

If we consider the restriction of $H_n$ to the domain $V$ we have:
\begin{align}
\left\|H_n-G_n\right\|_{C^0(V)}&\leq\left\|H'_n\right\|_{C^0(V)}\left\|f^{\mu_n(1)}-Id\right\|_{C^0(U)}\notag\\
&\leq\left\|H'_n\right\|_{C^0(V)}M(U)\|\mu_n\|_{\infty}.\notag
\end{align}
By Cauchy's derivative estimate we know that for all $z \in V$:
\begin{align}
\big|H_n'(z)\big|&=\left|\frac{1}{2\pi i}\int_{\partial B(z,\delta/2)}\frac{H_n(w)}{(w-z)^2}\, dw\right|\notag\\
&\leq\frac{2\|H_n\|_{C^0(f^{\mu_n(1)}(U))}}{\delta}\notag\\
&=\frac{2\|G_n\|_{C^0(U)}}{\delta}\notag\\
&\leq\frac{2R}{\delta}\quad\mbox{for all $n \geq n_0$.}\notag
\end{align}
That is:$$\big\|H'_n\big\|_{C^0(V)}\leq\frac{2R}{d\big(\partial V,\partial U\big)}\quad\mbox{for all $n \geq n_0$,}$$and we obtain that for all $n \geq n_0$:$$\frac{\big\|H_n-G_n\big\|_{C^0(V)}}{\|\mu_n\|_{\infty}}\leq\left(\frac{R}{d\big(\partial V,\partial U\big)}\right)\left(\frac{4}{\pi}\right)\sup_{z \in U}\left\{\iint_{U}\big|S\big(w,z\big)\big|dxdy\right\}.$$
Therefore is enough to consider:$$C(U)=\left(\frac{4}{\pi}\right)\sup_{z \in U}\left\{\iint_{U}\big|S\big(w,z\big)\big|dxdy\right\}.$$
\end{proof}

\end{document}